\documentclass[a4paper,reqno,oneside]{amsart}
\usepackage{geometry}
\usepackage{tgpagella}
\usepackage{aligned-overset}

\usepackage[normalem]{ulem}
\usepackage{amssymb,amsmath}
\usepackage{amsthm}

\usepackage{amsfonts}

\usepackage{subfigure}
\usepackage{graphicx}
\usepackage{epsfig}

\usepackage{esvect} 
\usepackage{enumitem}

\usepackage{dsfont}
\usepackage{mathtools}
\usepackage[hidelinks]{hyperref}
\usepackage{bm}

\usepackage{color}
\usepackage{todonotes}

\usepackage{latexsym}
\usepackage{amsmath}
\usepackage{amsfonts}
\usepackage{amssymb}
\usepackage{cleveref}

\usepackage{url}
\usepackage{algorithm}
\usepackage{algorithmic}
\usepackage{cancel}

\usepackage[most]{tcolorbox}

\usepackage{enumitem}

\newcommand{\ReLU}{{\rm ReLU}}
\newcommand{\RePU}{{\rm RePU}}
\DeclareMathOperator{\size}{size}
\DeclareMathOperator{\depth}{depth}
\DeclareMathOperator{\width}{width}
\DeclareMathOperator{\mpar}{mpar}
\newcommand{\sizedelta}{\size^{\delta}}
\newcommand{\depthdelta}{\depth^{\delta}}
\newcommand{\widthdelta}{\width^{\delta}}

\newcommand{\bsb}{{\bm{b}}}
\newcommand{\bsc}{{\bm{c}}}

\newcommand{\bse}{{\bm{e}}}

\newcommand{\bsg}{{\bm{g}}}

\newcommand{\bsv}{{\bm{v}}}
\newcommand{\bsw}{{\bm{w}}}
\newcommand{\bsx}{{\bm{x}}}
\newcommand{\bsy}{{\bm{y}}}
\newcommand{\bsz}{{\bm{z}}}

\newcommand{\bsZ}{{\bm{Z}}}

\newcommand{\bszero}{{\boldsymbol{0}}} 
\newcommand{\bsone}{{\boldsymbol{1}}}  

\newcommand{\bsgamma}{{\boldsymbol{\gamma}}}
\newcommand{\bsdelta}{{\boldsymbol{\delta}}}

\newcommand{\bszeta}{{\boldsymbol{\zeta}}}
\newcommand{\bseta}{{\boldsymbol{\eta}}}

\newcommand{\bslambda}{{\boldsymbol{\lambda}}}
\newcommand{\bsmu}{{\boldsymbol{\mu}}}
\newcommand{\bsnu}{{\boldsymbol{\nu}}}
\newcommand{\bsxi}{{\boldsymbol{\xi}}}

\newcommand{\bsvarrho}{{\boldsymbol{\varrho}}}

\newcommand{\bsphi}{{\boldsymbol{\phi}}}

\newcommand{\bbC}{{\mathbb{C}}}

\newcommand{\bbN}{{\mathbb{N}}}

\newcommand{\bbR}{{\mathbb{R}}}

\newcommand{\bbT}{{\mathbb{T}}}

\newcommand{\bbZ}{{\mathbb{Z}}}
\newcommand{\C}{{\mathbb{C}}} 
\newcommand{\N}{{\mathbb{N}}} 
\newcommand{\R}{{\mathbb{R}}} 
\newcommand{\NN}{{\mathbb{N}}} 

\DeclareSymbolFont{bbold}{U}{bbold}{m}{n}
\DeclareSymbolFontAlphabet{\mathbbold}{bbold}

\newcommand{\calB}{{\mathcal{B}}}
\newcommand{\calC}{{\mathcal{C}}}
\newcommand{\calD}{{\mathcal{D}}}
\newcommand{\calE}{{\mathcal{E}}}
\newcommand{\calF}{{\mathcal{F}}}
\newcommand{\calG}{{\mathcal{G}}}

\newcommand{\calJ}{{\mathcal{J}}}
\newcommand{\calK}{{\mathcal{K}}}
\newcommand{\calL}{{\mathcal{L}}}

\newcommand{\calN}{{\mathcal{N}}}

\newcommand{\calR}{{\mathcal{R}}}
\newcommand{\calS}{{\mathcal{S}}}

\newcommand{\calX}{{\mathcal{X}}}
\newcommand{\calY}{{\mathcal{Y}}}
\newcommand{\calZ}{{\mathcal{Z}}}

\newcommand{\ceil}[1]{\left\lceil #1 \right\rceil}    

\newcommand{\ii}{{\mathrm i}} 

\newcommand{\eps}{{\varepsilon}}

\newcommand{\KL}{{Karhunen-Lo\`{e}ve }} 

\newcommand{\bsnul}{{\boldsymbol{0}}}

\newcommand{\norm}[2][]{\|#2\|_{#1}}

\newcommand{\normlr}[2][]{\left\|#2\right\|_{#1}}
\newcommand{\set}[2]{\{#1 : #2\}}

\newcommand{\dup}[2]{\langle #1,#2 \rangle}

\DeclareMathOperator{\supp}{supp}

\DeclareMathOperator*{\Id}{Id}
\DeclareMathOperator*{\esssup}{ess\,sup}

\newcommand{\dd}{\;\mathrm{d}}

\newcommand{\be}{\begin{equation}}
\newcommand{\ee}{\end{equation}}

\newtheorem{theorem}{Theorem}[section]
\newtheorem{lemma}[theorem]{Lemma}
\newtheorem{proposition}[theorem]{Proposition}
\newtheorem{corollary}[theorem]{Corollary}
\newtheorem{assumption}[theorem]{Assumption}
\newtheorem{remark}[theorem]{Remark}
\newtheorem{example}[theorem]{Example}
\newtheorem{definition}[theorem]{Definition}

\newcommand{\rF}{\mathrm {F}}
\newcommand{\rS}{\mathrm {S}}

\newcommand{\cG}{\mathcal{G}}
 
\newcommand{\cX}{{\mathcal X}} 
\newcommand{\cY}{{\mathcal Y}}

\newcommand{\bmg}{\boldsymbol{g}}

\begin{document}
	\title{Neural and Spectral Operator Surrogates on Gaussian Spaces}
	\author{Carlo Marcati \and Mario Mari\'{c} \and Christoph Schwab \and Jakob Zech}
	\date{\today}
	\thanks{CM and JZ acknowledge the support of the National Group for Scientific Computing (GNCS - INDAM) through the Visiting Professors program. MM and JZ acknowledge support from the German Research Foundation (DFG) within the Priority Programme SPP 2298, \textit{Theoretical Foundations of Deep Learning} (project number 543965776).}
	 
	\newcommand{\Lip}{{\rm Lip}} 
	\newcommand{\HS}{{\rm HS}} 
	\maketitle
	
	\begin{abstract}
          We prove expression rate bounds of finite-parametric,
			spectral and neural surrogates for holomorphic maps between separable Hilbert spaces.
			The surrogates have an encoder-approximator-decoder architecture, 
                        with Karhunen-Lo\'eve encoders and frame decoders. 
			We prove expression rate bounds for two classes of 
			finite-parametric surrogates: i) spectral surrogates obtained by $N$-term
			truncations of Wiener polynomial chaos expansions and ii) neural surrogates
			obtained by approximation of parametric maps with deep feedforward
			neural networks, ReLU and RePU activation functions and uniformly bounded weights.
			We work under
      an algebraic decay assumption on the eigenvalues of the covariance 
      of the Gaussian measure on the input space.
      We obtain convergence rates for mean-square errors, 
      and additionally in first-order Gaussian Sobolev spaces,
      to account for errors in the approximation of gradients.
	\end{abstract}
	\tableofcontents
	\section{Introduction}
	\label{sec:Intro}
		Neural operators are finite-parametric approximations $\calG_N$ 
		of in general nonlinear maps $\calG:\calX\to\calY$
		between (subsets of) infinite-dimensional, 
		separable function spaces $\calX$ and $\calY$.
		In recent years, several classes of such \emph{surrogate maps} 
		have been proposed and developed. 
		See, e.g., the surveys 
		\cite{li2020fourier,kovachki2024operatorlearningalgorithmsanalysis}
		and references there.
		Starting with the foundational work \cite{ChenChen1993},
		to date 
		a large range of finite-parametric maps have been indentified 
		with mathematically ensured \emph{universal approximation} properties. 
		
		Surrogate families with \emph{expression rate bounds} are less
		well investigated:
		due to the infinite-dimensionality of the input- and output-spaces,
		viability of finite-parametric surrogate maps will incur 
		(and have to overcome, under conditions) 
		the so-called ``curse of dimension'' (CoD).
		
		One possible architecture of neural surrogate operators 
		(considered, e.g., in \cite{LMK21_949,HSZ24})
		takes the generic form
		$\calG_N = \calD \circ g_N \circ \calE$.
		Here, $\calE$ and $\calD$ are en- and decoders on the input and output spaces 
		$\calX$ and $\calY$ respectively.
		The maps $g_N$ depend on $N$ parameters, and 
		leverage approximation properties of suitable approximators.
		The proof of \emph{approximation rate bounds},
		i.e. estimates on the approximation error $\calG - \calG_N$ 
		in terms of $N$, on suitable sets of inputs in $\calX$, 
		involves first selecting en- and decoders,
                and then constructing parametric approximators $g_N$.
		Due to the infinite dimension of $\calX$ and $\calY$, 
		establishing approximation rate bounds for the approximators $g_N$ 
		requires overcoming the CoD.
    
    In this paper,
    we consider the same architecture as in \cite{HSZ24}, 
    and investigate \KL encoder and frame decoder systems, denoted as $\calE$ and $\calD$, respectively.  
    We admit $N$-parametric approximator maps $g_N$ of either \emph{polynomial chaos} or of
    \emph{neural network type}.  
    We relax the condition in \cite{HSZ24} that the parameter ranges of
       sequences resulting from input encoding be bounded
       (following suitable scaling).  Admitting unbounded
       parameter ranges obstructs ``worst-case''
       error bounds, as considered in \cite{HSZ24}.
       Accordingly, 
       we adopt a mean-square error analysis which also
       arises naturally in statistical regression and
       learning.  This setting requires the introduction of a
       \emph{Gaussian Measure} (GM) 
       $\gamma$ on the input space $\calX$ 
       with the sigma-algebra $\calB(\calX)$ of Borel sets.
       This GM will in general only be supported on a 
       (small, in a suitable sense) subset of $\calX$ of admissible inputs 
       (cf. Remark \ref{remark:CMSpc}).

\subsection{Existing Results}\label{sec:ExRes}
Recent years have seen brisk development of analytic and computational
approaches for operator learning.
The fundamental \emph{universal approximation property},
refers to the collection $\{\calG_N \}_{N\in\N}$ being (in principle)
capable of attaining any prescribed approximation error bound $\eps>0$
in a suitable sense, provided that $N$ is sufficiently large
(depending in an unspecific way on $\eps$ similar to the classical
Stone-Weierstrass theorem for polynomial approximation). This has been
verified for various architectures, see,
e.g., \cite{kovachki2024operatorlearningalgorithmsanalysis} and the
references there.  We mention here the pioneering \cite{ChenChen1993}
which established universal approximation of continuous functionals 
on compact subsets $\calK\subset \calX$ in a `worst-case' setting, 
for a so-called ``branch-and-trunk architecture''.
While universality of the
collection $\{\calG_N \}_{N\in\N}$ can be established under rather
mild conditions on the set of admissible inputs, \emph{expression rate
  bounds} require, as a rule, restrictions on the input sets
$\calX^s \subset \calX$ and the operators $\calG$ to be approximated.
Standard assumptions include, for example, the requirement 
that $\calG$ is a Lipschitz map from $\calX^s$ into $\calY$.
	
As it has been shown recently \cite{SLanthPCA,adcock2024learninglipschitzoperatorsrespect}, 
however, merely imposing Lipschitzianity
can imply a ``curse of parametric complexity'', for operator
approximation of Lipschitz maps $\calG$ between separable Hilbert
spaces $\calX$ and $\calY$ over $\bbR$, 
when $\calX$ is endowed with a Gaussian Measure,
using standard (Wiener-Hermite polynomial chaos, or feedforward neural network) 
approximation architectures.
In these references, 
upper and lower rate bounds were established for
Wiener-Hermite polynomial chaos emulations of $\calG$ with encoding
via the PCA of the covariance operator $Q$ of the centered GM on the input space $\calX$. 
The lower bounds suggest
that the emulation rates are, essentially, non-improvable, 
\emph{for the given approximation architecture consisting of $N$-term
  truncated Gauss-Hermite polynomial chaos expansions}, and under the
\emph{Lipschitz assumptions} in
\cite{adcock2024learninglipschitzoperatorsrespect}.  
We note in passing that these results on emulation rates of Lipschitz maps
$\calG$ appear, to some extent, as consequences of restrictive
(polynomial) approximation architectures used to build $g_N$:
\emph{neural approximator architectures} $g_N$ with $N$ parameters and
higher expressivity than polynomials can allow for higher approximation rates
\cite{SSZ23_3044} (albeit possibly at expense of stability) due to
so-called ``super-expressive'', non-polynomial activations in the
approximators $g_N$ similar to constructions in the Kolmogoroff
superposition theorem as explained, e.g., in
\cite{SCHMIDTHIEBER2021119}.  
Mathematical approximation rate results
of this type encounter, upon realization in 
floating-point arithmetic, potential instabilities associated with
rounding of network weights to finite precision. Encoding with PCA of GM on input space with upper and lower
bounds for expression rates is also investigated in \cite{SLanthPCA}.

We remark that operator approximation based also on derivatives, 
has recently been investigated computationally as so-called
``derivative-informed'' operator learning
\cite{luo2025dimensionreductionderivativeinformedoperator,yao2025derivativeinformedfourierneuraloperator,OLEARYROSEBERRY2022114199,OLEARYROSEBERRY2024112555}. In particular in 
\cite{yao2025derivativeinformedfourierneuraloperator}, 
universal approximation
in the Gaussian space $W^{1,2}_{\mu}(\calX,\calY)$ was proved
for a class of operator surrogates being 
``derivative-informed'' versions of the well-established 
Fourier Neural Operators.

\subsection{Contributions}\label{sec:Contr}
In the present paper, we develop a framework for approximation rate
bounds for \emph{holomorphic maps} $\calG$ in
$L_\gamma^2(\calX,\calY)$, where $\calX$, $\calY$ are separable Hilbert
spaces, and $\gamma$ is a Gaussian measure on $\calX$. Specifically:

\begin{itemize}
\item We prove {\bf algebraic convergence of spectral operator surrogates}  
  obtained from finitely truncated Hermite polynomial chaos expansions 
  (cf.~Corollary \ref{cor:convergence-expansion}). 
  Up to logarithmic terms, they emulate the operator
  $\calG$ to order $N^{-\min(s,t)}$ in $L^2_{\gamma}(\calX,\calY)$.
  Here $N$ denotes the number of parameters (polynomial coefficients)
  required for the computation of $\calG_N$. Moreover, the parameters
  $s>0$ and $t > 0$ correspond to the decay rate of eigenvalues of the
  covariance operator of the GM $\gamma$ and to extra regularity in
  the operator output, respectively.
\item We consider {\bf spectral surrogates of the Fr\'{e}chet derivative}
  $X\mapsto D\calG(X) \in \calL(\calX,\calY)$ restricted to suitable
  subspaces $\calX^r \subset \calX$; the parameter $r>0$
  relates to the size of a subspace $\calX^r\subseteq
  \calX$ in which we can compute
  directional derivatives of $\calG$.
  Up to logarithmic terms,
  we show the convergence order $N^{-\min(r,s,t)}$
  for the approximation of $D\calG(X)|_{\calX^r}$
  in $L^2_{\gamma}(\calX,{\rm HS}(\calX^r,\calY))$,
  where ${\rm HS}(\calX^r,\calY)$ denotes the 
  Hilbert space of Hilbert-Schmidt operators from $\calX^r$ to $\calY$ 
  (cf. Corollary \ref{cor:convergence-expansion}).

\item Our spectral operator surrogates imply {\bf the same algebraic
    convergence rates for deep neural operators}. We consider NNs
  with $\ReLU$ and $\RePU$ activation functions and uniformly bounded weights, 
  and discuss convergence
  of the operator and its derivatives in terms of the number of
  trainable network parameters in Theorems \ref{thm:main} and \ref{thm:main2}.
\end{itemize}

\subsection{Layout}
\label{sec:Strct}
In Section \ref{sec:Notat}
we introduce some basic notation and terminology.

Section~\ref{sec:Preliminaries} provides preliminary material and 
definitions for several key concepts used throughout the paper. 
In particular, we recall Gaussian measures on separable Hilbert spaces, 
Hermite polynomials, frames and Riesz bases, as well as weighted Gaussian Sobolev spaces. 
Moreover, our results require the formalization of a scale of subspaces of higher regularity, 
which is introduced in Subsection \ref{sec:SpcHighReg} via weighted summability of frame coefficients.

In Section \ref{sec:Architecture} 
we recapitulate the FrameNet architecture from 
\cite{HSZ24,reinhardt2024statisticallearningtheoryneural},
which will be used in our subsequent analysis.

Section \ref{sec:MainRes} presents our main network based approximation results, 
stated separately for the mean-squared and Sobolev cases.
	
Section \ref{sec:HolDom} follows with the full proofs of said bounds. 
Upon precising the notion of quantified holomorphy in Subsection \ref{sec:DomHol}, 
we address the verification of the quantified, parametric holomorphy on which the main results are based, in an abstract setting. 
In particular, quantified parametric holomorphy is shown to imply weighted summability of 
Hermite expansion coefficients 
of the countably-parametric maps in the discrete-continuous equivalence established in 
Subsection \ref{sec:Architecture}. 
Subsection \ref{sec:TrcH} implies construction of spectral approximations of operator surrogates 
via truncated Hermite polynomial chaos expansions. 
Subsection \ref{sec:PfThm} contains the proofs of the neural approximation rate bounds. 
	
In Section \ref{sec:diffusion}, 
we verify all assumptions for the coefficient-to-solution map of a linear, elliptic divergence form PDE.
	
Section \ref{sec:Concl} provides a brief summary of the main results, 
and indicates several direct generalizations of these.
	
Two appendices, Sec.~\ref{app:NNTheory} and \ref{app:DNNEmulationHermite},
review from \cite{SZ21_982} results on ReLU NN emulation rate bounds of Hermite polynomials, 
in Gaussian Sobolev spaces, and generalize to ReLU and RePU NNs with
uniformly bounded weights.

\subsection{Notation and Terminology}\label{sec:Notat}
        
{\bf Numbers.} 
We write $\N = \{1,2,\dots\}$ and denote by $\R$ and $\C$ the real and complex numbers, respectively. 
We denote by $\ii = +\sqrt{-1}$ the imaginary unit. 
For $z = x + \ii y \in \C$, we denote real and imaginary part of $z$ by $x = \Re(z)$ and $y = \Im(z)$.
Throughout the manuscript, $\log = \log_2$, but the results hold with any
base (larger than $1$) upon potentially adjusting the constants.

{\bf Gaussian Measures.}                
Centered Gaussian measures on separable Hilbert spaces $\calX$
shall be denoted by $\bsgamma$. 
We also use $\gamma_s$ with a parameter $s>0$ to indicate specific spectral decay properties of the covariance; see Section \ref{sec:GM} for a precise definition.
We let $\mu$ be the standard Gaussian measure on $\R$.
Moreover, $\bsmu$ denotes the infinite product Gaussian measure $\otimes_{j\in\N}\mu$ on $\R^\N$ 
equipped with the product $\sigma$-algebra.
	
{\bf Function Spaces.} 
For two separable real Hilbert space $\calX$, $\calY$ 
we write $C(\calX,\calY)$ for the space of continuous maps,
with the subspace $\Lip(\calX,\calY) \subset C(\calX,\calY)$ of Lipschitz continuous maps. 
We denote by $\calL(\calX,\calY) \subset C(\calX,\calY)$ the space of bounded, 
linear operators, by $\calL_{\rm is}(\calX,\calY)$ the subset of boundedly invertible linear operators, and by $\HS(\calX,\calY)\subset \calL(\calX,\calY)$ the (separable) Hilbert space of Hilbert-Schmidt operators. 
When $\calX = \calY$, we write in the preceding definitions $C(\calX)$ etc.

On a Hilbert space $\calX$ we denote by $\calB(\calX)$ the sigma-algebra of Borel sets in $\calX$. 
The space of strongly measurable, $\gamma$-square-integrable functions $\calG$ 
from $\calX$ to $\calY$ is denoted as $L^2_\gamma(\calX, \calY)$. 
Here, $\gamma$ is a Gaussian measure on $\calX$, so that $\calG \in L^2_\gamma(\calX,\calY)$ if
	\begin{equation*}
	 \| \calG \|^2_{L^2_\gamma(\calX,\calY)} 
         := 	\int_{\calX} \| \calG(X) \|_\calY^2 \dd\gamma(X) 
         < \infty.
         \end{equation*}
If $\calY=\R$, we usually write $L^2_\gamma(\calX)$ instead.               

{\bf Sequences.}  We denote sequences $(x_i)_{i\in I}\subset
\bbR$ with (finite or countably infinite) index set $I$ by $\bsx$.
Sequences $(Z_i)_{i\in I}$ in a Hilbert space $\calZ$ are denoted by
$\bsZ$.  The symbols $\bsnul = (0,0,..)$ and $\bsone = (1,1,...)$
denote the sequences with, respectively, zero and one entries
over any index set $I\subseteq \bbN$.  Operations on sequences $\bsx$
are understood element-wise, i.e.  $\bsx^2 = (x_i^2)_{i\in I}$,
$\exp(\bsx) = (\exp(x_i))_{i\in I}$, etc.  Relations between
sequences, such as $\bsx\leq \bsy$ are likewise understood
elementwise, i.e. as $x_i\leq y_i$ for $i\in I$. For $i\in I$, the
symbol $\bsdelta_i = (\delta_{ij})_{j\in I}$ denotes the
Kronecker sequence.

For a countable index set $I$, we use the classical sequence spaces
$\ell^p(I)$, of sequences of real-valued numbers, with summability
index $0<p\leq \infty$.

{\bf Multiindices.} By $\bbN_0^\bbN$, we denote the set of (infinite)
multi-indices, i.e.  of all sequences $\bsnu = (\nu_j)_{j\in\N}$ with
$\nu_j\in \bbN_0$.  For $\bsnu\in \bbN_0^\bbN$, we define the support
of $\bsnu$ as ${\rm supp}(\bsnu) := \{i\in \bbN | \nu_i \ne 0 \}$. 
We will also write, for a set of multi-indices 
$\Lambda\subset\bbN_0^\bbN$, 
$\supp\Lambda = \cup_{\bsnu\in\Lambda}\supp\bsnu$.
By
$\calF\subset\bbN_0^\bbN$ we denote the \emph{countable subset} of
\emph{finitely supported multi-indices}, i.e.
\begin{equation}\label{eq:calF} 
\calF = \{ \bsnu \in \bbN_0^\bbN : |{\rm supp}(\bsnu) | < \infty \} \;.
\end{equation}
For $\bsnu,\bsnu'\in \calF$, we write 
$\delta_{\bsnu,\bsnu'} := \prod_{i\in \bbN} \delta_{\nu_i,\nu'_i} $.  
The index set $\calF$
is endowed with the half-ordering ``$\le$'' introduced above.
For $\bsnu \in \calF$, define $|\bsnu| \coloneqq \sum_{j\in\N} \nu_j < \infty$ and $|\bsnu|_0 \coloneqq |\supp \bsnu|$.

For algorithmic purposes (e.g. to render spectral collocation 
	well-defined and feasible) we recall the
	notion of \emph{``downward closed''}\footnote{Also
		referred to as `lower' sets in some references.} 
	index sets $\Lambda \subset \calF$. 
A multi-index set $\Lambda\subseteq\calF$ is called downward closed 
if
$\bsnu\in\Lambda$ and $\bsmu\le\bsnu$, meaning $\mu_i\le\nu_i$ for all $i\in\N$, 
implies $\bsmu\in\Lambda$. 

	\section{Preliminaries}
	\label{sec:Preliminaries}

        \subsection{Gaussian Measures on Hilbert Spaces}\label{sec:GM}
        We start by recalling the definition of a Gaussian measure in a Hilbert space, e.g., \cite[Def.\ 2.2.1]{Bogachev1998}.
        \begin{definition}\label{def:gamma}
          Let $\calX$ be a separable Hilbert space, and let $\gamma$ be a measure on $(\calX,\calB(\calX))$. We say that $\gamma$ is a (centered) Gaussian measure on $\calX$ iff for every $\varphi\in \calX'$ it holds that $\gamma\circ\varphi^{-1}$ defines a (centered) Gaussian measure on $(\R,\calB(\R))$.
        \end{definition}

        For a Gaussian measure $\gamma$ as in the Definition, its mean $m$ and covariance operator $Q$ are defined as
        \begin{equation*}
          m:=\int_{\calX} X \dd\gamma(X)\in \calX
        \end{equation*}
        and
        \begin{equation*}
          Q:\calX\to \calX\qquad\text{via}\qquad
          \langle Q X,Y\rangle = \int_{\calX} \langle Z-m,X\rangle\langle Z-m,Y\rangle\dd\gamma(Z)
        \end{equation*}
        for all $X$, $Y\in\calX$. One can show that $Q$ is a nonnegative, 
        self-adjoint trace-class operator \cite[Theorem 2.3.1]{Bogachev1998}; 
        in particular $Q$ is a compact operator. If $\gamma$ is centered then $m=0\in\calX$.

        For Gaussian measures $\gamma$,  $m$ and $Q$ uniquely determine
        $\gamma$, and conversely every $( m, Q )$ couple as above has a
        corresponding Gaussian measure, see again \cite[Theorem 2.3.1]{Bogachev1998}.
        Typically $m$ and $Q$ are derived from the measure. 
For a succinct presentation of our results it will be convenient however to 
introduce a sequence of centered Gaussian measures defined by their covariance operator. 
To do so, we introduce two weight sequences $\bsw, \bsv  \in (0,1]^\N$ that will be fixed throughout. 
As explained in the following,
$\bsw$ will be associated with the input space $\calX$, 
and $\bsv$ with the output space $\calY$.
In order to simplify the presentation, we make the following assumptions on the weight sequences.

\begin{assumption}\label{ass:vw}
Let $\bsw$, $\bsv\subset (0,1]^{\N}$ be sequences such that
\begin{enumerate}[label=(\roman*)]
	\item $\bsw, \bsv$ are monotonically decreasing and belong to $\ell^1(\N)$
	\item there exist $C, \, r_0 > 0$ 
              such that $w_j \geq C j^{-r_0}$ for all $j \in \N$.
\end{enumerate} 
\end{assumption}

\begin{remark}
For all $N\in\N$ it holds that $Nw_N\le\sum_{j=1}^N w_j$ so that $w_N\le \norm[\ell^1(\N)]{\bsw}/N$.
\end{remark}

We can then characterize the measures that we consider.
\begin{definition}\label{def:GM} 
  Let $\calX$ be a separable Hilbert space,
  $(\phi_j)_{j\in\N}$ an orthonormal basis of $\calX$
  and let $\bsw$ be as in Assumption \ref{ass:vw}.
  For every $s\ge 0$, set
  \begin{equation*}
    \lambda_j(s):=w_j^{2s+1}\qquad\forall j\in\N.
  \end{equation*}
  Let $\gamma_s$ be the centered Gaussian measure whose covariance
  $Q_s\in\calL(\calX)$ has eigenvalues and eigenvectors given as
  $(\lambda_j(s),\phi_j)$, i.e., for all $X\in\calX$,
  \begin{equation}\label{eq:Qs}
    Q_s X = \sum_{j\in\N} \lambda_j(s)\langle \phi_j,X\rangle \phi_j.
  \end{equation}
  For $s=0$ we also write $\gamma:=\gamma_0$.
  \end{definition}

  Since $(\phi_j)_{j\in\N}$ is an ONB, and $\lambda_j(s)\leq w_j$ defines an $\ell^1$ sequence for each $s\ge 0$,
  $Q_s$ in \eqref{eq:Qs} is non-negative and self-adjoint with trace ${\rm tr}(Q_s)=\sum_{j\in\N}\lambda_j(s)<\infty$. Thus the Gaussian measure $\gamma_s$ is well-defined for all $s\ge 0$. 
We also note that
  \begin{equation}\label{eq:seqlambd}
    \bslambda(s) := (\lambda_j(s))_{j\in\N} \in \ell^{1/(2s+1)}(\N).
  \end{equation}
The parameter $s > 0$ controls the spectral decay properties of $\gamma_s$ 
and will appear prominently in our convergence rate bounds later on.
In particular, a larger value of $s$ implies a faster decrease of the
eigenvalues $\lambda_j(s)$ and a smaller support of $\gamma_s$ (see Remark \ref{remark:CMSpc}). 
We formulate \eqref{eq:seqlambd} in this way as this will simplify the notation in later estimates.

  \begin{remark}[Karhunen-Lo\'eve Decomposition]
  \label{rmk:KL}
    The measure $\gamma_s$ in Definition~\ref{def:GM}
    is the pushforward of $\bsmu$ (cf.~Section \ref{sec:Notat}) 
    under the map
    \begin{equation*}
      \begin{cases}
        \R^\N\to\calX\\
        (\xi_j)_{j\in\N}\mapsto \sum_{j\in\N}\xi_j\sqrt{\lambda_j(s)}\phi_j.
      \end{cases}
    \end{equation*}
    In particular,
    given a sequence of real-valued i.i.d.\ standard Gaussian random variables $(\xi_j)_{j\in\N}$,
    the term $\sum_{j\in\N}\xi_j\sqrt{\lambda_j(s)} \phi_j$ 
    defines a random variable with distribution $\gamma_s$, see, e.g., \cite[Theorem 3.5.1]{Bogachev1998}.
  \end{remark}

        \subsection{Hermite Polynomials}\label{sec:HermitePoly}
        \subsubsection{Univariate Hermite Polynomials}
	For $n\in \bbN_0$, the (probabilists') Hermite
	polynomial of degree $n$ is given by 
	\begin{equation}\label{eq:Hermite}
	H_n: \bbR\to \bbR: x\mapsto \frac{(-1)^n}{\sqrt{n!}} \exp(x^2/2) \frac{d^n}{dx^n} \exp(-x^2/2) 
	\;.
	\end{equation}
	With this normalization, $H_0 \equiv 1$ and $\{ H_n\}_{n\in
          \bbN_0}$ is an ONB of $L^2_{\mu}(\bbR)$
        (e.g., \cite[Prop.~9.4]{DPIntro06}), where $\mu$ is the
        standard Gaussian measure on $\R$.

        \subsubsection{Multivariate Hermite Polynomials}
        Tensorized Hermite
        polynomials will take a central role in our construction of
        finite-parametric operator surrogates.  For $\bsnu\in \calF$,
        we set
	\begin{equation}\label{eq:HermiTP}
	H_\bsnu(\bsx) := \prod_{j\geq 1} H_{\nu_j}(x_j) \;, \quad \bsx = (x_1,x_2,...)\in \bbR^\bbN \;.
	\end{equation}
	Due to $H_0 \equiv 1$, the products \eqref{eq:HermiTP} have
        only finitely many nontrivial factors, for each
        $\bsnu\in \calF$.
        The countable set
        $\{ H_\bsnu: \bsnu \in \calF \}$ forms an ONB of $L^2_{\bsmu}(\R^\N)$
        \cite[Thm.~9.7]{DPIntro06}.

\subsubsection{Hermite Polynomials on Hilbert Spaces}
\label{sec:HerPolHS}
Let $\gamma_s$ be a centered Gaussian measure on $(\calX, \calB(\calX))$ as in
Definition \ref{def:GM} and let $(\lambda_j(s) ,\phi_j)_{j\in\N}$ be the eigenvalues and eigenfunctions of its covariance operator. 
Then, we set
	\begin{equation}\label{eq:HermX}
	H_{\bsnu,\bslambda(s)} (X) := \prod_{j\in\N} H_{\nu_j}(\lambda_j(s)^{-1/2} \langle X,\phi_j\rangle)\;,
	\quad 
	\bsnu\in \calF, \; X\in \calX \;.
	\end{equation}
	The collection $\{ H_{\bsnu,{\bslambda(s)}} : \bsnu \in \calF \}$ 
	is an ONB of $L^2_{\gamma_s}(\calX)$ (see, e.g., \cite[Thm.~9.7]{DPIntro06}).
	The following proposition follows then from the isometric isomorphism
	$L^2_{\gamma_s}(\calX,\calY) \simeq L^2_{\gamma_s}(\calX)\otimes \calY$ 
        (with $\otimes$ indicating the Hilbertian tensor product).
	\begin{proposition} \label{prop:WPCrep}
		Under the assumptions in Section~\ref{sec:GM},
		every $\calG \in L^2_{{\gamma_s}}(\calX,\calY)$ 
		admits the 
		\emph{polynomial chaos (pc) representation}
                \begin{equation}\label{eq:cGpc}
		\calG(X) = \sum_{\bsnu\in \calF} g_{\bsnu,{\bslambda(s)}} H_{\bsnu,\bslambda(s)}(X) \;, 
\end{equation}
		with equality in the sense of $L^2_{{\gamma_s}}(\calX,\calY)$, 
		and
		with the pc coefficients $g_{\bsnu,{\bslambda(s)}}\in \calY$ 
		given by 
    \begin{equation}\label{eq:gbsnu}
		g_{\bsnu,{\bslambda(s)}} 
		:= \int_\calX \calG(X) H_{\bsnu,\bslambda(s)}(X) \, {\rm d}\gamma_s(X) \;, \;\;
		\bsnu\in \calF\;.
                \end{equation}
		Furthermore, Parseval's equality holds:
		\begin{equation}\label{eq:Parsev}
			\| \calG \|_{L^2_{\gamma_s}(\calX,\calY)}^2
			=
			\sum_{\bsnu\in \calF} \| g_{\bsnu,\bslambda(s)} \|^2_{\calY}
			\;.
		\end{equation}
	\end{proposition}
        \subsection{Frames}\label{sec:ParaMap}
	We recap the notion of frames (see also \cite{HSZ18_2635,reinhardt2024statisticallearningtheoryneural}).
	We refer to \cite{Christensen,Heil} for general terminology and 
	basic properties of frames in Hilbert spaces. 
    Riesz bases and orthonormal bases such as wavelet, Fourier, and 
      Karhunen-Lo\`{e}ve bases (described in Remark~\ref{rmk:KL})
      are particular cases of frames.
      Operator surrogates whose decoders are based on these systems are covered by our 
        approximation rate results.
	\begin{definition}
        \label{def:frame}
		A collection $\bseta = \set{\eta_j}{j\in\N}\subset \calY$ 
		is called a \emph{frame for $\calY$}, 
                if the \emph{analysis operator}
		$$
		\rF:
    \begin{cases}
    \calY \to \ell^2(\N)\\ v\mapsto (\dup{v}{\eta_j})_{j\in\N}
    \end{cases}
		$$
		is boundedly invertible between $\calY$ and ${\rm range}(\rF)\subset \ell^2(\N)$.
	The adjoint $\rF'$ of the analysis operator is called the \emph{synthesis operator}.
	It is given by
	\begin{equation}\label{eq:calFp}
		\rF':
    \begin{cases}
    \ell^2(\N) \to \calY\\ \bsv \mapsto \bsv^\top\bseta \;.
    \end{cases}
	\end{equation}
	\end{definition}
	The \emph{numerical stability of frames}
	is quantified by the \emph{frame constants}
	\begin{equation} \label{eq:FrBd}
	\underline{\kappa} \coloneqq \underline{\kappa}(\bseta) := \inf_{0\ne v \in \calY} \frac{\| \rF v \|_{\ell^2}}{\| v \|_\calY} \;,
	\quad
	\overline{\kappa} \coloneqq \overline{\kappa}(\bseta) := \| \rF \|_{\calY\to \ell^2} = \sup_{0\ne v \in \calY}  \frac{\| \rF v \|_{\ell^2}}{\| v \|_\calY} \;.
	\end{equation}
	\emph{Parseval frames} are frames $\bseta$ with ideal conditioning
	$\underline{\kappa} = \overline{\kappa} = 1$.
	\begin{remark}\label{rmk:FrBd}
		Since $\|\rF'\|_{\ell^2\to\calY}=\|\rF\|_{\calY\to\ell^2}$,
		\eqref{eq:FrBd} implies that 
		for all $\bsv\in\ell^2(\N)$
		\begin{equation}\label{eq:rmk:FrBd}
			\normlr[\calY]{\sum_{j\in\N}v_j\eta_j}^2
			=\norm[\calY]{\rF'\bsv}^2
			\le \overline{\kappa}^2\sum_{j\in\N}v_j^2 
			= \overline{\kappa}^2 \norm[\ell^2]{\bsv}^2.
		\end{equation}
        \end{remark}
\begin{remark}\label{rmk:Riesz}
          When the function system 
        $\bseta$ in Definition~\ref{def:frame} is also a basis of $\calY$,
        then it is a Riesz basis. In this case
        the analysis operator $\rF: \calY \to \ell^2(\N)$ is boundedly invertible,
        as is the synthesis operator 
        $\rF': \ell^2(\N) \to\calY$ in \eqref{eq:calFp}. 
        The frame constants $\underline{\kappa}, \overline{\kappa}$ in \eqref{eq:FrBd} are then
        referred to as Riesz constants.
        A particular case of Riesz bases are orthonormal bases, where
        $\rF$ and $\rF'$ are isometries, and $\underline{\kappa} = \overline{\kappa} = 1$.
	\end{remark}

	The \emph{frame operator} $\rS := \rF'\rF: \calY \to \calY$ 
	is
	boundedly invertible, self-adjoint and positive
	(e.g. \cite{Christensen,Heil})
	with
	$\| \rF' \rF \|_{\calY\to \calY} = \overline{\kappa}^2$
	and
	$\| (\rF' \rF)^{-1} \|_{\calY\to \calY} = \underline{\kappa}^{-2}$,
	\cite[Lemma 5.1.5]{Christensen}.
	With the pseudoinverse $(\rF')^\dagger = \rF(\rF'\rF)^{-1}$
	of the synthesis operator, given by
	$$
	(\rF')^\dagger:
  \begin{cases}
  \calY \to \ell^2(\N)\\ f \mapsto \{ \langle f, \rS^{-1} \eta_j \rangle \}_{j\in \N},
  \end{cases}
	$$
	the \emph{frame decomposition theorem} asserts that
	every $f\in \calY$ can be uniquely and stably reconstructed
	from a corresponding sequence $\{  \langle f,\rS^{-1}\eta_j\rangle \}_{j\in\N}$ 
	of frame coefficients via
	$$
	f = \rF'(\rF')^\dagger f 
	= \sum_{j\in \N} \langle f,\rS^{-1}\eta_j \rangle \eta_j
	= \sum_{j\in \N} \langle f,\eta_j\rangle \rS^{-1} \eta_j 
	\;.
	$$
	The collection $\tilde{\bseta}:= \rS^{-1}\bseta$
	is a frame for $\calY$ which is referred to
	as the \emph{canonical dual frame} of $\bseta$,
	with analysis operator 
	$\tilde{\rF} := (\rF')^\dagger =  \rF (\rF'\rF)^{-1}$,
	and frame constants \eqref{eq:FrBd} equal to
	$\overline{\kappa}(\tilde \bseta) = \underline{\kappa}(\bseta)^{-1}$ 
	and 
	$\underline{\kappa}(\tilde \bseta) = \overline{\kappa}(\bseta)^{-1}$, 
	respectively.
	
	The frame decomposition theorem takes the form (e.g. \cite{Christensen,Heil})
	\begin{equation}\label{eq:compid}
		\rF' \tilde{\rF} = I\qquad\text{on}\qquad \calY.
	\end{equation}
	Hence, every $v\in \calY$ has a representation
	$v = \bsv^\top\bseta$ with $\bsv = \tilde\rF(v)\in \ell^2(\N)$,
	and
	\begin{equation}\label{eq:FrStab}
	\overline{\kappa}(\bseta)^{-1} \leq \frac{\| \bsv
		\|_{\ell^2}}{\|v\|_\calY} \leq \underline{\kappa}(\bseta)^{-1} \;.
	\end{equation}
	Property \eqref{eq:FrStab} is equivalent to $\bseta$
	being a frame for $\calY$ (see, e.g., \cite[Thm.~8.29 (b)]{Heil}).

\subsection{Scales of Hilbert Spaces}\label{sec:SpcHighReg}
We now introduce subspaces of $\cX$ and $\cY$ characterized by 
expansion coefficient decay of their elements; 
this can be interpreted as a form of sparsity or regularity. 
To quantify this, 
let in the following the sequences $\bsw$, $\bsv$ be as in Assumption \ref{ass:vw}.

\subsubsection{Input Space $\calX$}
\label{sec:InpSpc}

Let ${\gamma_s}$ be a Gaussian measure on $\calX$ as in Definition \ref{def:GM}. 
In particular $(w_j^{2s+1} ,\phi_j)_{j\in\N}$ 
are the eigenvalues and eigenfunctions of the covariance operator $Q$, 
and $(\phi_j)_{j\in\N}$ forms an orthonormal basis of the separable Hilbert space $\calX$. 
For $r\ge 0$ we let
\begin{equation}\label{eq:Xs}
  \calX^r \coloneqq \set{X \in \calX}{\norm[\calX^r]{X}  < \infty }\qquad\text{where}\qquad
  \norm[\calX^r]{X}^2:=\sum_{j\in\N} w_j^{-2r - 1}  |\langle X, \phi_j \rangle_\calX|^2.
\end{equation}
	The subspace $\calX^r \subset \calX$ is a Hilbert space with norm induced by the inner product
\begin{equation}\label{eq:Xsinner}
		\langle X, Z \rangle_{\calX^r} \coloneqq \sum_{j\in\N} w_j^{-2r - 1} \langle X, \phi_j \rangle_\calX \, \langle Z, \phi_j \rangle_\calX \qquad \forall \, X, Z \in \calX^r,
\end{equation}
              see \cite[Lemma 1]{HSZ24}.
              Note that $\calX^r\supseteq\calX^{r'}$ iff $r\le r'$, and
              $(w_j^{r+1/2} \phi_j)_{j\in\N}$ forms an ONB of $\calX^r$.
\begin{remark}
  \label{remark:CMSpc}
Let $r, s \geq 0$, fix $\gamma_s$ as in Definition \ref{def:GM} 
and let the subspace $\calX^{r} \subset \calX$ be as in \eqref{eq:Xs}. 
We have (see Remark \ref{rmk:KL})
		\begin{equation*}
			\gamma_s(\calX^{r}) 
			=
			\gamma_s(\{ X \in \calX : \sum_{j\in\N}w_j^{-2r-1}\langle X, \phi_j \rangle_\calX^2 < \infty \})
			=
			\bsmu(\{ \bsxi \in \R^\N : \sum_{j\in\N} w_j^{-2r-1} \lambda_j(s) \xi_j^2 < \infty \}).
		\end{equation*}
		Now, from \cite[Example 2.3.6]{Bogachev1998} it follows that $\gamma_s(\calX^{r}) = 1$ if 
		\[
			\sum_{j\in\N} w_j^{-2r-1} \lambda_j(s) = \sum_{j\in\N} w_j^{-2r-1} w_j^{2s+1} < \infty.
		\] 
		This holds true for all $r \geq 0$ such that  $r \leq s-\frac{1}{2}$. Note that $\bsmu$ has unit variances.
	\end{remark}

	\begin{remark}\label{rmk:CMSpc2}
          The Cameron-Martin space of the Gaussian measure $\gamma_s$ (cf.~Definition \ref{def:GM}) can
          be characterized as $Q_s^{1/2}\calX=\set{\sum_{j\in\N}w_j^{1/2+s}\phi_j\xi_j}{(\xi_j)_{j\in\N}\in\ell^2}$, 
          e.g., \cite[Remark 2.3.3]{Bogachev1998}. 
          This corresponds exactly to $\calX^s$ in \eqref{eq:Xs}.
          It follows (see \cite[Prop.~1.27]{DPIntro06}) that 
          $\gamma_s(\calX^s) = 0$. Since the subspaces $\calX^r$ are measurable 
          and nested in the sense that $\calX \supset \calX^r \supset \calX^{r'}$ for $0\leq r \leq r'$
          we obtain $\gamma_s(\calX^r) = 0$ for all $r \geq s$.		
	\end{remark}
        \begin{example}[{\cite[Example 2.3.4]{Bogachev1998}}]
          Let $D\subseteq\R^d$ be a bounded domain. Then
          there exists a GM on $L^2(D)$ with CM space $H^s(D)$ if $s>d/2$. 
          The measure is concentrated on $H^r(D)$  for any $r < s-d/2$.          
        \end{example}

        \subsubsection{Output Space $\calY$}\label{sec:outputspace}
	As for the input space, we require a characterization of the coefficient decay of the output. 
        While for the input space
        we used the eigenbasis of the covariance of $\gamma_s$ (from which the operator inputs will stem),
        the choice of representation system is less critical on the output side. For this reason
        we allow for arbitrary frames. Throughout, we work under the following assumption.

        \begin{assumption}\label{ass:eta}
          Let $(\eta_j)_{j\in\N}$ be a frame of $\cY$ such that $(\tilde\eta_j)_{j\in\N}$ is its canonical dual frame.
        \end{assumption}

        With $\bsv$ as in Assumption \ref{ass:vw}, i.e., $\bsv$ is a monotonically decrasing
        strictly positive $\ell^1$-sequence, we define, similarly as in \eqref{eq:Xs} but with a different scaling,
        for any $t\ge 0$,
\begin{equation}\label{eq:Yt}
  \calY^t \coloneqq \set{Y \in \calY}{\norm[\calY^t]{Y}  < \infty }\qquad\text{where}\qquad
  \norm[\calY^t]{Y}^2:=\sum_{j\in\N} v_j^{-2t}  |\langle Y, \tilde\eta_j \rangle_\calY|^2.
\end{equation}
Similar to \eqref{eq:Xsinner},
this defines a Hilbert scale 
with inner product given by weighted $\ell^2$ inner products of the coefficients.

        \begin{remark}\label{rmk:wgtBesov}
        Weighted summability of expansion coefficients describes membership in
        classical function spaces such as
        Sobolev and Besov spaces with bounded $t$-th weak derivative; see, e.g.,
        \cite{HSZ24}.
      \end{remark}

\subsection{Operator Derivatives}\label{sec:GaussSobolevMap}
Let $\calG:\calX\to\calY$. 
We are interested in the  approximation of $\calG$ in the sense of $L_{\gamma_s}^2(\calX,\calY)$; 
additionally we want to approximate its derivatives. 
While Malliavin calculus as presented, e.g., in \cite{NualartNualart2018,DaPratoIntrMalliavon02}, 
provides the standard mathematical framework for Gaussian-Sobolev spaces and weak derivatives, 
we restrict our attention here to (complex) Fr\'echet differentiable operators. 
This allows us to work with classical derivatives and avoid certain 
technicalities of the full Malliavin framework. 
For a recent treatment focusing on Lipschitz continuous operators, 
we refer to \cite{adcock2024learninglipschitzoperatorsrespect},
and for a universal approximation result for Fourier Neural Operator type
surrogates in first order spaces, to \cite{yao2025derivativeinformedfourierneuraloperator}.

We denote by $D\cG(X)\in\calL(\cX,\cY)$ the Fr\'echet derivative 
of $\calG$ at $X\in\cX$. 
For $r\ge 0$ and $\calX^r$ as defined in \eqref{eq:Xs}, 
it holds that $\calX^r\hookrightarrow\calX$ so that
\begin{equation*}
  D_{\calX^r}\cG(X):=D\cG(X)|_{\calX^r}\in\calL(\calX^r,\calY)
\end{equation*}
is well-defined. 
We will later see that under suitable assumptions and for sufficiently large $r$, 
$D_{\calX^r}\cG(X)$ is a Hilbert-Schmidt operator, 
i.e., it belongs to ${\rm HS}(\calX^r,\calY)$. 
We will consider approximation of $\cG$ w.r.t.~the norm
	\begin{equation}\label{eq:SobolevNorm}
		\|\calG\|_{W^{1,2}_{\gamma_s,r}(\calX,\calY)}^2 \coloneqq
		\int_\calX \|\calG(X)\|_\calY^2 \, {\rm d}\gamma_s(X) 
		+
		\int_\calX \|D_{\calX^r}\calG(X)\|_{{\rm HS}(\calX^r, \calY)}^2 \, {\rm d}\gamma_s(X).
    \end{equation}
We adopt the notation $W^{1,2}_{\gamma_s,r}(\calX,\calY)$ 
to reflect consistency with canonical Gaussian-Sobolev spaces.
They are typically constructed via certain closures of differential
operators \cite{MichalikGoldys2001,DPIntro06,NualartNualart2018}.

\begin{remark}\label{rmk:sizeW12r}
 The $W^{1,2}_{\gamma_s,r}(\calX,\calY)$-norm defined in \eqref{eq:SobolevNorm} 
 becomes weaker as $r$ increases, since the derivative is then restricted to a smaller set.
\end{remark}
  The principal reason for dealing with the Hilbert-Schmidt norm {\eqref{eq:SobolevNorm}}, 
  is that it allows for a simple characterization in terms of weighted coefficient sequences. 
  We recall such a result in the following proposition, 
  see also \cite[Proposition C.7]{adcock2024learninglipschitzoperatorsrespect}. 
  The proof can be found in Appendix \ref{app:AuxiliaryResults/Definitions}.

  \begin{proposition}\label{prop:weightedParseval}
    Let $r,s\ge 0$, and let $\gamma_s$ be as in Definition \ref{def:GM}.
    Let $\calG:\calX\to\calY$ be Fr\'echet differentiable and satisfy $\|\calG\|_{W_{\gamma_s,r}^{1,2}(\calX,\calY)} < \infty$.
	Then 
	\begin{equation}\label{eq:weightedParseval}
    	\norm[W_{\gamma_s,r}^{1,2}(\calX,\calY)]{\calG}^2
		=
		\sum_{\bsnu \in \calF} \Gamma_\bsnu(r,s) \|g_{\bsnu,\bslambda(s)}\|_\calY^2
	\end{equation}
	
	with  $g_{\bsnu,\bslambda(s)}$ defined in \eqref{eq:gbsnu}, 
        and
	\begin{equation}\label{eq:Gamma}
		\Gamma_\bsnu(r,s) \coloneqq 1 + \sum_{i=1}^{\infty} \nu_i w_i^{2r+1} \lambda_i(s)^{-1}=1 + \sum_{i=1}^{\infty} \nu_i w_i^{2(r-s)}.
	\end{equation}
  \end{proposition}
In the sequel we often simplify the notation and write $\Gamma_\bsnu \coloneqq \Gamma_\bsnu(r,s)$.
       
	\section{FrameNet Architecture} \label{sec:Architecture}

        We recall the \emph{FrameNet} architecture as presented in
        \cite{HSZ24,reinhardt2024statisticallearningtheoryneural}.
        It is an encoder-decoder type architecture, with a
        feedforward neural network in the middle: For an encoder
        $\calE_m:\calX\to\R^{m}$,
        a neural network
        $g_N:\R^m\to\R^n$,
        and a decoder
        $\calD_n:\R^n\to\calY$,
        operators $\calG:\calX\to\calY$ are approximated via

        \begin{equation}\label{eq:Defg0}
          \calG\simeq \calD_n\circ g_N\circ \calE_m.
        \end{equation}

        Here $N \in \N$ relates to the size of the neural network (see Sec. \ref{sec:NN}). 
        If the size of the network is not important for a particular statement or definition we omit the index $N$.
        We next discuss each component in more detail.

        \subsection{Encoder}\label{sec:encoder}

        The input of the operator $\calG$ is assumed to be distributed
        according to the Gaussian measure $\gamma_s$ in Definition
        \ref{def:GM}. Consequently, it is natural to work with the
        eigenfunctions of the covariance operator as a representation
        system. With $(\phi_j)_{j\in\N}$ (an ONB) as in Definition \ref{def:GM}
        and for $m\in\N$ we thus let
	\begin{equation}\label{eq:encoder}
          \calE_m(X):=(\dup{X}{\phi_j}_{\calX})_{j=1}^m\qquad\forall X\in\calX.
	\end{equation}
        The finite encoding dimension $m\in \N$ is a parameter to be chosen in the sequel. 

        \begin{remark}[Restriction, Padding]\label{rmk:Pad}
        When $m = \infty$ we denote $\calE_\infty: \calX \to \ell^2(\N)$,
        which is an isometric isomorphism (cf. Remark~\ref{rmk:Riesz}).
        Then, $\calE_m = \calR_m \circ \calE_\infty$
        with the ``\emph{dimension-truncation}'' restriction $\calR_m: \ell^2(\N) \to \R^m$.
        Its adjoint, the ``\emph{zero-padding}'' prolongation $\calR_m^*: \R^m \to \ell^2(\N)$
        completes a vector in $\R^m$ by zero extension to a sequence in $\ell^2(\N)$.
The finite-parametric neural operator approximations in \eqref{eq:Defg0} may 
        then be formally expressed as 

$
        \calD_\infty \circ \calR_n^* \circ g_N \circ \calR_m  \circ \calE_\infty : \calX \to \calY \;.
$

        \end{remark}

        \begin{remark}[Basis Encoder Property] \label{rmk:EncGRF}
          Suppose $X$ is a random variable with law $\gamma_s$ as in Definition \ref{def:GM}.
          Then $\calE_m(X)$ are $m$ \emph{independent} 
          real-valued Gaussian random variables (with differing variances
          depending on $\lambda_1(s), \dots, \lambda_m(s)$).

          While Gaussian measures can be expanded in 
          \emph{any Parseval frame of the Cameron-Martin space},
          cf.~\cite{LPFrame09}, a \emph{general Parseval frame based encoder}
          would not give independent random variables. As this plays a
          crucial role in our analysis, we restrict ourselves to
          orthogonal basis encoders \eqref{eq:encoder}.
        \end{remark}

        \subsection{Decoder}\label{sec:decoder}

        For the representation of the operator output, we 
        allow for an arbitrary frame in $\calY$. 
        This will be the same frame
	$\bseta = (\eta_j)_{j\in\N}$ introduced in Assumption \ref{ass:eta} and used
        to define the spaces $\calY^t$ in Section \ref{sec:outputspace},
        so that membership in $\calY^t$ is implied by weighted summability of frame coefficients.
        For arbitrary, finite $n\in \N$, 
            the decoder $\calD_n = \calD_\infty\circ \calR_n^*$ is given by
	\begin{equation}\label{eq:decoder}
\calD_n :
          \begin{cases}
            \R^n \to \calY \\
            \bsc\mapsto \sum_{j=1}^n c_j \eta_j \;.
          \end{cases}
	\end{equation}
        \begin{remark}\label{rmk:FrmBPX}
        A frame decoder can be realized on regular, unstructured simplicial partitions of polyhedra 
        via the BPX multi-level iteration (see, e.g., \cite{HSS08}, the orginal
	construction due to P.~Oswald in space dimension $d=2$ \cite{OswaldBPXP12d},
	and subsequently in polyhedra in \cite{OswaldML94};
  see also the references there).
      \end{remark}
    \begin{remark}\label{rmk:EncDec}
    	To lighten notation, if the parametric dimensions 
        $m$ and $n$ in \eqref{eq:encoder}, \eqref{eq:decoder} are clear
    	from context or not relevant for a statement we suppress these subscripts
    	and simply write $\calE$ and $\calD$. 
    \end{remark}

	\subsection{Neural Networks}
	\label{sec:NN}
	
	\begin{definition}{{(Feedforward Neural Network)}} 
		\label{def:FFNN}
		For a depth $L\in \N$, 
		let $\{ p_l \}_{l=0}^{L+1}\in \N^{L+2}$ be a sequence of NN widths. 
		A function $f: \R^{p_0} \to \R^{p_{L+1}} : \bsx \mapsto f(\bsx)$ 
		is called 
		\emph{feedforward NN} (FFNN) of depth $L=\depth(f)$ 
			and 
			width $\width(f) \coloneqq \max_{0 \leq l \leq L+1} p_l$ 
		if there
		exists an \emph{activation function} $\sigma:\R \to \R$ and 
		real-valued weights $w^l_{ij}$ and biases $b^l_{j}$ such that 
		for all $\bsx = (x_i)_{i=1}^{p_0} \in \R^{p_0}$ holds
		\begin{equation}\label{eq:NN1}
			\begin{array}{rcl} 
				z^1_{j} & = & \displaystyle 
				\sigma\left( \sum_{i=1}^{p_0} w^1_{ij} x_i + b^1_j \right) , \quad j=1,...,p_1, 
				\\
				z^{l+1}_j & = & \displaystyle 
				\sigma\left( \sum_{i=1}^{p_l} w^{l+1}_{ij} z^l_i + b^{l+1}_j \right) , \quad l=1,..,L-1, \; j=1,...,p_{l+1},  
				\\
				f(\bsx) & = & \displaystyle 
				(z^{L+1}_j)_{j=1}^{p_{L+1}}  
				=  
				\left( \sum_{i=1}^{p_{L}} w^{L+1}_{ij} z^L_i + b^{L+1}_j \right)_{j=1}^{p_{L+1}} \quad .
			\end{array}
		\end{equation}
	\end{definition}
	In order to be consistent with the assigned labels $m$ and $n$ for the output and input dimensions of the encoder $\calE_m$ and decoder $\calD_n$, we write for a FFNN $m \coloneqq p_0$ and $n \coloneqq p_{L+1}$.
	For a FFNN $f:\R^{m} \to \R^{n}$ as in Definition~\ref{def:FFNN},
	we define its
	\begin{itemize}
		\item $\size(f)$ as the number of nonzero parameters, 
		\[
		\size(f) := | \{ (i,j,l): w^l_{ij} \ne 0 \}| + | \{ (j,l) : b^l_j \ne 0 \}|
		\]
		\item $\mpar(f)$, the \emph{maximum parameter} value, 
		as
		\[
		\mpar(f) := \max\{ |w^l_{ij}|, |b^l_j| : i,j,l \}  \in (0,\infty) \;,
		\]
	\end{itemize}
	We will consider for $q\in \N$ the $\RePU$ activation function
	\begin{equation}\label{eq:AcReQ}
		\sigma_q(x) = x_+^q = (\max\{x,0\})^q \;, \quad x\in \R \;.
	\end{equation}
	For $q=1$, $\sigma_1$ is the classical ReLU activation.

        \subsection{FrameNet}
        The
        parametric surrogate maps in the FrameNet are selected from
        the following set: given an activation function
        $\sigma:\R \to \R$, integers $L,P,S \in \N$ and a real number $B\in \R$,
	\begin{equation}\label{eq:Defg1}
		\begin{array}{rcl}
			\bmg_{FN}(\sigma,L,P,S,B) 
			&:=& \displaystyle  
			\left\{ g:\R^{m}\to\R^{n} \;\mbox{is a NN with activation}\; \sigma \;\mbox{s.t.} \right.
			\\
			& & \displaystyle 
			\qquad 
			\depth(g) \leq L, \; \width(g) \leq P,\; \size(g) \leq S, 
			\\
			& & \displaystyle
			\left.
			\qquad  
			\mpar (g) \leq B 
			\right\}.
		\end{array}
              \end{equation}
	Note in particular that $\max\{n,m\} \leq P$.
	With the class $\bmg_{FN}$ at hand, 
	in the setting of {Sec.~\ref{sec:Architecture},}
	the FrameNet class of operator surrogates is given by
	\begin{equation}\label{eq:GFN}
		\left\{ \calD_n \circ g \circ \calE_m : g \in \bmg_{FN}(\sigma,L,P,S,B) \right\}.
	\end{equation}
        Here, the NN classes depend implicitly also on $n,m \in \NN$, the
            number of input- and output-channels.
        To lighten the notation, we do not write those explicitly.

        Throughout the rest of this paper we introduce
          \begin{equation}\label{eq:delta}
            \delta>0
          \end{equation}
          which is intended to be arbitrarily small, fixed.
	For $q \in \N $ and $N \in \N$, 
        in Theorem \ref{thm:main} and \ref{thm:main2} the \emph{FrameNet class} 
	is
	\begin{equation}\label{eq:FrNsp3}
		\{\calD_N \circ \tilde g_N \circ \calE_N : 
                \tilde g_N \in \bsg_{FN}(\sigma_q, \depthdelta_N, \widthdelta_N, \sizedelta_N, B)\}
	\end{equation}
	with 
	\begin{equation}\label{eq:sizeDepthN}
		\sizedelta_N \coloneq C_S \, N^{1+\delta}, \quad \depthdelta_N \coloneq C_L \, N^{\delta}, \quad 
		\widthdelta_N \coloneq \ceil{C_P N^{{1+\delta}}},
	\end{equation}	
        and where, throughout the rest of this manuscript, $C_S$, $C_L$, and $C_P$ are constants
        depending on $\delta$, $q$ and $B$ as introduced above; our theorems below will
        postulate existence of such constants such that certain convergence rates are satisfied
        as the number of network parameters $N\to\infty$ tends to infinity. 
       Crucially these constants will never depend on $N$. 
       We do not mention this anymore in the following when referring to \eqref{eq:sizeDepthN}.
	
	\section{Main Results}
	\label{sec:MainRes}
	We are now ready to state our main results.
	We consider holomorphic operators $\calG$. 
        To make the condition precise,
        we first introduce a complex "strip" of width $\theta > 0$ 
	\[
		\calS_\theta^{\calX} \coloneqq \{ x+\ii y \;:\; x,y \in \calX, \, \|y\|_\calX < \theta \} \subset \calX_\C
              \]
              in the (canonical, see \cite{MujiCplxAnBSpc}) complexification $\calX_\C$ of the Hilbert space $\calX$. 
        The {next} assumption is an adaptation of {\cite[Assumption 5.1]{MSZ25_1135}}.

\begin{assumption}\label{ass:G}
  There exist $C$, $\alpha$, $\theta > 0$, $\tau\in [0,2)$ and $t>0$ such that
  \begin{enumerate}
  \item \label{item:Ghol} $\calG:\calS_{\theta}^\calX\to \calY_{\bbC}$ is holomorphic, and
  \item \label{item:G-bound}$\norm[\calY_{\bbC}^t]{\calG(a)}\le  C\exp(\alpha\norm[\calX_{\bbC}]{a}^\tau)$ for all $a\in\calS_{\theta}^\calX$.
  \end{enumerate}
\end{assumption}
\begin{remark}
  The two items in Assumption \ref{ass:G} imply that
  $\calG:\calS_\theta^\calX\to \calY^t_{\C}$ is holomorphic. To see this, remark
  that
  \begin{enumerate}
  \item $\calY^t_\C\hookrightarrow \calY_{\C}$ by definition,
  \item $\calG:\calS^\calX_\theta \to \calY_\C^t$ is continuous by item
    \ref{item:G-bound} of Assumption \ref{ass:G} and using \cite[Lemma 5.3]{MSZ25_1135}.
  \end{enumerate}
  Then, the holomorphy of $\calG$ as an operator $\calS_\theta^\calX\to
  \calY^t_{\C}$ follows from \cite[Remark 5.2]{MSZ25_1135}. Note that the space $Z$ in that
  remark is the $\calY_\C$ of this paper, and $\calY_\C$ there is $\calY^t_\C$ here.
\end{remark}

	\subsection{Mean-Squared Approximation Rates}
	\label{sec:MainResL2}
The following result states that a holomorphic operator $\calG$ as in Assumption \ref{ass:G} can be approximated by FrameNet in a mean-squared sense with respect to a Gaussian measure $\gamma_s$ as in Definition \ref{def:GM}. 
The FrameNet architecture uses the eigenfunctions of the covariance of $\gamma_s$ to encode the input, and an arbitrary fixed frame to represent the operator output. The expression rates depend critically on the parameters $s > 0$ and $t > 0$: $s$ describes the spectral decay properties of the measure $\gamma_s$, and $t$ encodes smoothness in the output, in the sense of fast coefficient decay w.r.t. the chosen frame.
	
\begin{theorem}\label{thm:main}
         Let $\calX$, $\calY$ be two separable Hilbert spaces,
         and let $\calG:\calX\to\calY$ be such that for some $s>0$, $t>0$:
         \begin{itemize}
         \item $\gamma_s$ is a Gaussian measure as in Definition \ref{def:GM}, in particular the covariance eigenvalues belong to $\ell^{1/(2s+1)}(\N)$,
         \item $\calG$ satisfies Assumption \ref{ass:G}, in particular it is holomorphic 
         (w.r.t.\ the topology on $\calY_\C$) and maps to $\calY_\C^t$.
         \end{itemize}

          Then for any fixed $q\in\N$ and any $\delta>0$, there
          exists a sequence $\{ \Phi_N \}_{N\geq 1}$ such that for each $N\in \N$,
          $\Phi_N \in \bmg_{FN}(\sigma_q,\depthdelta_N, \widthdelta_N , \sizedelta_N ,B_q)$ with $\sizedelta_N$,
          $\depthdelta_N$ and $\widthdelta_N$ 
          defined in \eqref{eq:sizeDepthN}, 
          $B_q > 0$ a constant only depending on $q\in\N$, and such
            that $\calG$ can be approximated with the error bound
		\begin{equation}\label{eq:OpEmRate}
			\norm[L_{\gamma_s}^2(\calX,\calY)]{\calG-\calD_N\circ\Phi_N\circ\calE_N} = O(N^{-\min\{s, t\}})
                      \end{equation}
                      as $N\to \infty$.
                      Here $\calE_N$, $\calD_N$ are as in Sections \ref{sec:encoder}-\ref{sec:decoder}. 
         Furthermore, we have $B_1 = 1$, i.e., 
         for $q = 1$ (ReLU) the network weights are bounded by $1$.
	\end{theorem}

	The proof of this result will be given in Sec.~\ref{sec:PfThm}.
	To establish Thm.~\ref{thm:main}, we use from \cite{DNSZ23_2957} 
	a notion of parametric holomorphy of the coefficient maps.
	On way to establishing Thm.~\ref{thm:main}, 
	we introduce in Sec.~\ref{sec:TrcH} 
	the construction of a family of spectral operator networks
	obtained from $N$-term truncations $G_N$ of Hermite expansions. 
	
	Neural operators $\calG_N$ are subsequently obtained 
	by approximating coordinate maps $G_N$ in terms of Wiener-Hermite
	polynomial chaos expansions of the component functions of $G_N$.
	
	\subsection{Gauss-Sobolev Approximation Rates}
        \label{sec:GSSobRates}
	The next result yields expression rate bounds 
with respect to the Gauss-Sobolev norm $\|\cdot\|_{W^{1,2}_{\gamma_s,r}(\calX,\calY)}$ in \eqref{eq:SobolevNorm}. 
The expression rate now depends on the additional parameter $r > 0$, 
associated with the derivative norm, cf.~Section \ref{sec:GaussSobolevMap}; 
in particular, larger $r$ corresponds to a smaller norm. 
        The approximation rates worsen compared to the previous result in $L_{\gamma_s}^2(\calX,\calY)$ if $r < s$.

\begin{theorem}\label{thm:main2}
          Consider the setting of Theorem \ref{thm:main} and let $r > 0$.
          
		Then for any fixed $q \in \N$ and $\delta > 0$, 
there exists a sequence $\{\Phi_N\}_{N\in\N}$ such that for each $N \in \N$ it holds that 
$\Phi_N \in \bsg_{FN}(\sigma_q, \depthdelta_N, \widthdelta_N, \sizedelta_N, B_q)$ 
with $\sizedelta_N$, $\depthdelta_N$ and $\widthdelta_N$ defined in \eqref{eq:sizeDepthN}, 
$B_q > 0$ a constant only depending on $q\in\N$, 
and such that 
		\begin{equation}
			\|\calG - \calD_N \circ \Phi_N \circ \calE_N\|_{W_{\gamma_s,r}^{1,2}(\calX,\calY)} = O(N^{-\min(r, s, t)})
    \end{equation}
    as $N\to\infty$.
    Here $\calE_N$, $\calD_N$ are as in Sections \ref{sec:encoder}-\ref{sec:decoder}.  
Furthermore, we have $B_1 = 1$,  
i.e., for q = 1 (ReLU) the network weights are bounded by 1.
\end{theorem}
The proof of this result is also postponed to Section \ref{sec:HolDom} and follows a similar 
strategy as the proof of the aforementioned $L_{\gamma_s}^2$ result, 
but accounting for the additional weights occuring in \eqref{eq:Gamma}. 
	
	\section{Proofs of Expression Rates}
	\label{sec:HolDom}
        We shall prove expression rate bounds for holomorphic operators.
	As in \cite{DNSZ23_2957}, we adopt a notion of \emph{parametric holomorphy}. 
	\subsection{Domains of Parametric Holomorphy}
	\label{sec:DomHol}
	In this section we recall the notion of parametric 
        $(\bsb,\xi, \chi, {\epsilon}, Y)$-holomorphy from
	\cite[Definition 2.2]{MSZ25_1135}, and discuss the domains of
	holomorphy for (parametrized versions of) operators 
        satisfying Assumption \ref{ass:G}.
	To this end we denote, for 
	$ N \in \N$ and for $\bsvarrho\in (0,\infty)^N$, 
	\begin{equation*}
		\calS_\bsvarrho^\C 
		:=
		\set{\bsz\in\C^N}{|\Im(z_i)|<\varrho_i~\forall i} \subseteq \C^N \quad \text{and}\quad \calB(\bsvarrho) \coloneqq \{z\in \C^N: |z_i| < \varrho_i \; \forall i \}.
	\end{equation*}
  We omit the dependence on $N$ of these sets as it will be clear from the context.

	\begin{definition}[($\bsb,\xi,\chi, \epsilon, Y$)-Holomorphy]
		\label{def:bdXHol}
		Let $Y$ be a complex, separable Hilbert space, 
		let 
		$\bsb=(b_j)_{j\in\N} \in (0,\infty)^\infty$ be a sequence
		and 
		let $\xi>0$, $\epsilon > 0$ be parameters.
		
		For $N\in\N$ 
		and $\chi\in [1,\infty]$,
		the sequence $\bsvarrho\in (0,\infty)^N$ 
		is called \emph{$(\bsb,\xi,\chi)$-admissible} 
		if
		\begin{equation}\label{eq:adm}
				\sum_{j=1}^N b_j^{\chi} \varrho_j^{\chi} \leq \xi^{\chi} \;\;\mbox{if}\;\; \chi < \infty,
				\quad \max_{1\leq j \leq N} b_j\varrho_j \leq \xi \;\;\mbox{if}\;\; \chi = \infty.
		\end{equation}
		When $\chi = 1$, we refer to \emph{$(\bsb,\xi,1)$-admissible} sequences 
		simply as ``\emph{$(\bsb,\xi)$-admissible}''.

		A function $u\in L^2_\bsmu(\R^\N,Y)$ is called \emph{$(\bsb,\xi, \chi, \epsilon,Y)$}-holomorphic 
		if
		\begin{enumerate}
			\item\label{item:hol} for every $N\in\N$ there exists
			$u_N:\R^N\to Y$, which, for every $(\bsb,\xi, \chi)$-admissible
			$\bsvarrho\in (0,\infty)^N$, admits a holomorphic extension
			(denoted again by $u_N$) from $\calS_\bsvarrho^\C\to Y$.
			Furthermore,
			for all $N,K \in \N$ with $N<K$
			\begin{equation}\label{eq:un=um}
				u_N(y_1,\dots,y_N)={u_K}(y_1,\dots,y_N,0,\dots,0)\qquad\forall (y_j)_{j=1}^N\in\R^N,
			\end{equation}
			
			\item\label{item:varphi} for every $N\in\N$ there exists
			$\varphi_N:\R^N\to\R_+$ such that
			$\norm[L_{\bsmu_N}^2(\R^N)]{\varphi_N}\le \epsilon$ and
			\begin{equation*} \label{ineq[phi]}
				\sup_{\substack{\bsvarrho\in(0,\infty)^N\\
						\text{is $(\bsb,\xi,\chi)$-adm.}}}~\sup_{\bsz\in \calB(\bsvarrho)}
                                \norm[Y]{u_N(\bsy+\bsz)}\le \varphi_N(\bsy)\qquad\forall\bsy\in\R^N,
			\end{equation*}
			\item\label{item:vN} 
			with $\tilde u_N:\R^\N\to Y$ defined by
			$\tilde u_N(\bsy) :=u_N(y_1,\dots,y_N)$ for $\bsy = (y_j)_{j\geq 1} \in \R^\N$ 
			it holds
			\begin{equation*}
				\lim_{N\to\infty}\norm[L^2_\bsmu(\R^\N,Y)]{u-\tilde u_N}=0.
			\end{equation*}
		\end{enumerate}
	\end{definition}
		In \cite{DNSZ23_2957}, sufficient conditions were given 
		for coefficient-to-solution operators of linear, elliptic divergence-form PDEs 
                with log-Gaussian random field coefficients
		to be ($\bsb,\xi,1,\epsilon,Y$)-holomorphic.
		In \cite{MSZ25_1135} it has been shown that Assumption \ref{ass:G} 
                on an operator $\calG$ is a sufficient condition for ($\bsb,\xi,2,\epsilon,Y$)-holomorphy. 
    We recall this result here.
As in Sec.~\ref{sec:Architecture},
$\bsphi \coloneqq (\phi_j)_{j\in\N} \subset \calX$ 
is the eigenbasis of the covariance operator of $\gamma_s$.

	\begin{proposition}{{\cite[Proposition 5.2]{MSZ25_1135}}}
		\label{prop:holomorphy}
		Suppose that $\calG$ satisfies Assumption \ref{ass:G} 
		and let $\bsb\in\ell^2(\N)$ be a sequence of positive reals. 
		Then, the map $u:\R^\N \to \calY_{\C}^{t}$ 
                given by
		\begin{equation} \label{eq:ubsy}
			u(\bsy) := \calG\Big(\sum_{j\in\N}y_jb_j\phi_j\Big)
		\end{equation}
		is $(\bsb,\theta,2,\epsilon, \calY_{\C}^t)$-holomorphic 
                for some $\epsilon = \epsilon(\alpha,\bsb,\tau)$ 
                where $\alpha > 0, \tau \in [0,2)$ are the constants in Assumption \ref{ass:G}.
	\end{proposition}
		
	\subsection{Weighted Summability Results} \label{sec:WgtSum}
	
	We recall the definition of the Hermite coefficients $u_\bsnu$ 
	of 
	$u\in L_{\bsmu}^2(\R^\N, \calY)$ (see Prop.~\ref{prop:WPCrep}).
	For $\bsnu\in \calF$, they are given by
	\begin{equation}
		\label{eq:unu}
		u_\bsnu = \int_{\R^\N}u(\bsy)H_\bsnu(\bsy)\dd\bsmu(\bsy) \in \calY \;.
	\end{equation}
	As shown in \cite{DNSZ23_2957,MSZ25_1135},
	the Hermite coefficients of $(\bsb,\xi,\chi, \epsilon,Y)$-holomorphic
	functions are weighted summable.
	As in \cite{DNSZ23_2957}, 
	weighted summability of Hermite coefficients 
	will allow in Sec.~\ref{sec:TrcH} ahead to construct spectral operator surrogates.

    In the following we 
    let
  \begin{equation}\label{eq:us}
    u^s(\bsy) \coloneqq \calG\left(\sum_{j\in\N} y_j \sqrt{\lambda_j(s)} \phi_j\right)
    =
    \calG\left(\sum_{j\in\N} y_j w_j^{s+1/2} \phi_j\right).
  \end{equation}
  Furthermore, we denote the Hermite coefficients of $u^s$ by 
  \begin{equation}\label{eq:unus}
  u_\bsnu^s \coloneqq \int_{\R^\N} u^s(\bsy) H_\bsnu(\bsy) \, \dd \bsmu(\bsy) \quad \in \calY^t.
  \end{equation}
  According to Proposition \ref{prop:holomorphy}, 
  the function $u^s$ is well defined in $L_\bsmu^2(\R^\N, \calY^t)$ 
  and is 
  $(\bslambda^{1/2}(s), \theta, 2, \epsilon, \calY_\C^t)$-holomorphic.
  \par
  A collection of weights
  $(c_\bsnu)_{\bsnu\in\calF} \subset \R$ 
  is called 
  \emph{monotonically increasing}, 
  if
  \begin{equation}\label{eq:moninc}
  	\bsnu\le\bsmu\qquad\text{implies}
  	\qquad{c_\bsnu\le c_\bsmu}\qquad \text{for all }\bsnu,\bsmu\in\calF.
  \end{equation}
  A collection $(c_\bsnu)_{\bsnu\in\calF} \subset \R$ 
  is 
  \emph{monotonically decreasing} if $(-c_\bsnu)_{\bsnu\in\calF}$ 
  is monotonically increasing. 
  The following theorem is the basis for establishing approximation rates with respect to 
  the weighted Sobolev norm in \eqref{eq:SobolevNorm}.

  \begin{theorem}\label{thm:bdXSum}
    Let $r,s > 0$ and let $\epsilon > 0, M \in \N, M \geq 2s+1$,
    and let $u^s$ be as in \eqref{eq:us} with $\calG$ as in Assumption \ref{ass:G}. 

    Then there exists a constant $K = K(M,\bsw,r,s,\theta) > 0$ 
    such that the sequence 
  	\begin{equation}\label{eq:cnudef}
          c_\bsnu 
          \coloneqq \prod_{j\in\supp\bsnu} \nu_j^M \max(1,K
          w_j^{-2\min(s,r)}
          )
  	\end{equation}
  	is monotonically increasing in the sense of \eqref{eq:moninc}, satisfies 
  	\begin{equation}\label{eq:bdXsum1}
  		\sum_{\bsnu\in\calF} c_\bsnu^{-\frac{1}{2\min(r,s)}} < \infty,
  	\end{equation}
  	and, with $\Gamma_\bsnu = \Gamma_\bsnu(r,s)$ defined in \eqref{eq:Gamma},
  	\begin{equation}\label{eq:bdXsum2}
  		\sum_{\bsnu\in\calF} \Gamma_\bsnu c_\bsnu \|u_\bsnu^s\|_{\calY_\C^t}^2 < \infty.
  	\end{equation}
  	Moreover,
  	\begin{equation}\label{eq:bdXsum3}
  		u^s(\bsy) = \sum_{\bsnu\in\calF} u_\bsnu^s H_\bsnu(\bsy) \in L_\bsmu^2(\R^\N, \calY_\C^t).
  	\end{equation}
  \end{theorem}
  \begin{proof}
    The proof proceeds in four steps, and is largely an application of the analysis in \cite[Section 6]{DNSZ23_2957} and \cite{MSZ25_1135}.
    
    \textbf{Step 1.} We show all claims except for \eqref{eq:bdXsum2} (with $K>0$ in \eqref{eq:cnudef}
    assumed arbitrary but fixed for the moment).

    The monotonicity \eqref{eq:moninc} of $(c_\bsnu)_{\bsnu\in\calF}$ in \eqref{eq:cnudef}
    follows directly from its definition, which implies
    $\nu_j\mapsto c_\bsnu$ to be monotonically increasing in $\nu_j$
    for each $j$.

    Next, we show the summability result \eqref{eq:bdXsum1}. 
    Since $\bsw \in \ell^1(\N)$, it (trivially) holds that
  	\begin{equation*}
          ((w_j^{-\min(s,r)})^{-1})_{j\in\N}\in \ell^{\frac{1}{\min(r,s)}}(\N).
  	\end{equation*}
  	A direct application of \cite[Lemma 6.6]{DNSZ23_2957} yields 
  	\begin{equation*}
  		(c_\bsnu^{-1})_{\bsnu\in\calF} \in \ell^{\frac{1}{2\min(r,s)}}(\calF).
              \end{equation*}
    
    Finally, \eqref{eq:bdXsum3} holds because $\calY_\C^t$ is a separable Hilbert space, 
    $(H_\bsnu)_{\bsnu\in\calF}$ is an ONB of $L_\bsmu^2(\R^\N)$, 
    and $u^s$ belongs to $L_\bsmu^2(\R^\N,\calY_\C^t)$ by Proposition \ref{prop:holomorphy} 
    (see the discussion around \cite[Eqn. (4.5)]{DNSZ23_2957} for more details).
    
    \textbf{Step 2.} It remains to show \eqref{eq:bdXsum2}. 
    We start by recalling a weighted summability result from \cite{MSZ25_1135}.

    Due to $M \geq 2s+1$ and the
    $(\bsw^{s+1/2}, \theta, 2, \epsilon, \calY_\C^t)$-holomorphy of $u^s$
    (see Proposition \ref{prop:holomorphy}), according to
    \cite[Thm. 3.1]{MSZ25_1135}\footnote{With $X \coloneqq \calY_\C^t$ and
    $p \coloneqq 1/(s+1/2)$.} there exists a constant
    $C = C(\bsw, \theta, M) > 0$ such that the following holds:
    With
  	\begin{equation}\label{eq:varrho}
  		\bsvarrho = (\varrho_j)_{j\in\N}, \qquad \varrho_j \coloneqq C w_j^{-s}
  	\end{equation}
  	and
  	\begin{equation}\label{eq:beta}
  		\beta_\bsnu(M+1,\bsvarrho) \coloneqq 
  		\sum_{\|\bsnu'\|_{\ell^\infty} \leq M+1} \binom{\bsnu}{\bsnu'} \bsvarrho^{2\bsnu'}
  		=
  		\prod_{j\in\N} \left(\sum_{\ell = 0}^{M+1} \binom{\nu_j}{\ell} \varrho_j^{2\ell}\right), \qquad \bsnu \in \calF,
  	\end{equation}
        we have
  	\begin{equation}\label{eq:betabsnu}
  		\sum_{\bsnu\in\calF}\beta_\bsnu(M+1,\bsvarrho) \|u_\bsnu^s\|_{\calY_\C^t}^2 < \infty.
  	\end{equation}

  	\textbf{Step 3.} Assume $r\ge s$, i.e.\ $\min(r,s)=s$. We show that there exists
        $K>0$ in \eqref{eq:cnudef} such that \eqref{eq:bdXsum2} holds.

  	According to \cite[Lemma 6.5]{DNSZ23_2957},
        there exist constants $C_0$, $K>0$ depending on $\bsw$ and $s$,
        such that
  	\begin{equation}\label{eq:hatcbsnu}
          \hat c_\bsnu := \prod_{j\in\supp\bsnu} \nu_j^M \max(1, K w_j^{-2s}),
  	\end{equation}
        satisfies with $\rho_\bsnu:=\prod_{j\in\N}(1+\nu_j)$
        \begin{equation}\label{eq:L65}
          C_0 \hat c_\bsnu \rho_\bsnu\le\beta_\bsnu(M+1,\bsvarrho)\qquad\forall\bsnu\in\calF.
        \end{equation}
        It holds $0 < w_j^{2(r-s)} \leq 1$ for all $j\in\N$ since $\bsw \in (0,1]^\N$.
        Therefore
  	\begin{equation*}
          \Gamma_\bsnu(r,s) = 1 + \sum_{j\in\N} \nu_j w_j^{2(r-s)} \leq
          1 + \sum_{j\in\N} \nu_j\le 
          \prod_{j\in\supp\bsnu} (1 + \nu_j) = \rho_\bsnu.
  	\end{equation*}
  	Thus \eqref{eq:betabsnu} and \eqref{eq:L65} give
  	\begin{equation*}
          \sum_{\bsnu\in\calF} \Gamma_\bsnu \hat c_\bsnu \|u_\bsnu^s\|_{\calY_\C^t}^2
          \le
          \frac{1}{C_0} \sum_{\bsnu\in\calF} \beta_\bsnu(M+1,\bsvarrho) \|u_\bsnu^s\|_{\calY_\C^t}^2
          < \infty.
  	\end{equation*}
        Finally note that $\hat c_\bsnu$ in \eqref{eq:hatcbsnu} coincides
        with $c_\bsnu$ in \eqref{eq:cnudef} due to $r\ge s$. 
        Thus \eqref{eq:bdXsum2} holds for $r\ge s$.
        
  	\textbf{Step 4.} 
        Let $K$ and $\hat c_\bsnu$ be as in Step 3 (depending on $s$ but independent of $r$). 
        For the rest of this step assume $r<s$, i.e.\ $\min(r,s)=r$. 
        We show that \eqref{eq:bdXsum2} holds.

        Since $w_j\in (0,1]$ and $r<s$ it holds $w_j^{2(r-s)}\ge 1$. 
        Hence for $\bsnu\in\calF$
        \begin{align}\label{eq:Gbsnubound}
          \Gamma_\bsnu(r,s) &= 1 +\sum_{j\in\N}\nu_jw_j^{2(r-s)}
          \le \prod_{j\in\supp\bsnu}(1+\nu_jw_j^{2(r-s)})\nonumber\\
          &\le \prod_{j\in\supp\bsnu}(1+\nu_j)\prod_{j\in\supp\bsnu}w_j^{2(r-s)}=\rho_\bsnu\prod_{j\in\supp\bsnu}w_j^{2(r-s)},
        \end{align}
        where we let again $\rho_\bsnu=\prod_{j\in\supp\bsnu}(1+\nu_j)$.
        Moreover, since $(w_j)_{j\in\N}\in\ell^1(\N)$, it holds $w_j^{-2r}\to\infty$ as $j\to\infty$. 
        Thus there exists $J\in\N$ such that $\max(1,Kw_j^{-2r})=Kw_j^{-2r}$ for all $j\ge J$. 
        Then, using again $w_j^{2(r-s)}\ge 1$,
        \begin{equation*}
          \prod_{j\in\supp\bsnu}w_j^{2(r-s)}\max(1,Kw_j^{-2r})\le
          \prod_{j=1}^J w_j^{2(r-s)}
          \prod_{j\in\supp\bsnu}\max(1,Kw_j^{-2s})
          = \tilde C \prod_{j\in\supp\bsnu}\max(1,Kw_j^{-2s}),
        \end{equation*}
        where $\tilde C:=\prod_{j=1}^J w_j^{2(r-s)}$.

        Using \eqref{eq:Gbsnubound} we get
        \begin{equation*}
          \Gamma_\bsnu c_\bsnu \le \rho_\bsnu \prod_{j\in\supp\bsnu}w_j^{2(r-s)} \prod_{j\in\supp\bsnu}\nu_j^M \max(1,Kw_j^{-2r})
          \le \tilde C \rho_\bsnu \prod_{j\in\supp\bsnu}\nu_j^M \max(1,Kw_j^{-2s})=\tilde C \rho_\bsnu \hat c_\bsnu,
        \end{equation*}
        with $\hat c_\bsnu$ as in \eqref{eq:hatcbsnu}. As in Step 3,
        \eqref{eq:L65} and \eqref{eq:betabsnu} give
        \begin{equation*}
          \sum_{\bsnu\in\calF}\Gamma_\bsnu c_\bsnu \norm[\calY_\C^t]{u_\bsnu^s}^2
          \le \tilde C \sum_{\bsnu\in\calF} \rho_\bsnu \hat c_\bsnu \norm[\calY_\C^t]{u_\bsnu^s}^2<\infty,
        \end{equation*}
        and thus \eqref{eq:bdXsum2} in case $r<s$.
  \end{proof}
  
 The next corollary immediately follows from Theorem \ref{thm:bdXSum}. 
	\begin{corollary}\label{cor:bdXSum-L2}
		Consider the setting of Theorem \ref{thm:bdXSum}. 
                Then the weights $(c_\bsnu)_{\bsnu\in\calF}$ satisfy
		\[
			\sum_{\bsnu\in\calF} c_\bsnu \|u_\bsnu^s\|_{\calY_\C^t}^2 < \infty, \qquad 
			\sum_{\bsnu\in\calF} c_\bsnu^{-\frac{1}{2s}} < \infty,
		\]
		and 
		\[
			u(\bsy)^s = \sum_{\bsnu\in\calF} u_\bsnu^s H_\bsnu(\bsy) \in L_\bsmu^2(\R^\N,Z).
		\]
	\end{corollary}
To see this, we observe that $r=s$ in Theorem \ref{thm:bdXSum} 
implies in \eqref{eq:bdXsum1} the second assertion. 
The weigthed square summability of $\|u_\bsnu^s\|_{\calY_\C^t}$ follows from \eqref{eq:bdXsum2} 
upon observing that in \eqref{eq:Gamma} it holds that $\Gamma_\bsnu(r,s) \geq 1$.

	\subsection{Spectral Operator Surrogates: Expression Rates}
	\label{sec:TrcH}

	We are now in position to introduce a first family $\{ \calG_N \}_{N\in \N}$ 
	of finite-parametric operator surrogates. 
        These will be obtained from $N$-term truncations
	of Hermite pc expansions. 
	 
	Using the weighted summability of the Hermite
	coefficients $u_\bsnu^s$ from Theorem \ref{thm:bdXSum} 
        and Corollary \ref{cor:bdXSum-L2}, 
        we can then deduce best
	$N$-term approximation rates in the usual way, 
	by truncating the Hermite expansion of $u^s$ (see again \eqref{eq:us} and \eqref{eq:bdXsum3}). 
	
	We will
	now show such a result, where we additionally expand $u^s(\bsy)$ 
	in the frame $\bseta = (\eta_j)_{j\geq 1} \subset \calY$. 
	That is, 
	with the Hermite coefficients $u_\bsnu^s$ from \eqref{eq:unus}
	of the parametric map $u^s(\bsy)$,
	we consider in the following the $\calY^t$-convergent 
	expansion
	\begin{equation}\label{eq:uexpansion}
		u^s(\bsy) = \sum_{j\in\N} \eta_j \sum_{\bsnu\in\calF}\dup{u_\bsnu^s}{\tilde\eta_j}_{\calY} H_\bsnu(\bsy).
	\end{equation}
	We build finite-parametric approximations of $\calG$ 
	by restricting the sum over $\bsnu\in \calF$ in \eqref{eq:uexpansion} 
	to finite index sets $\Lambda \subset \calF$. 
		\begin{definition}[maximal polynomial degree and effective dimension]
\label{def:dmLambda}
		For a finite index set $\Lambda\subseteq\calF$,
		we denote the maximal polynomial degree $m(\Lambda)$ 
		and the effective dimension $d(\Lambda)$ as
		\begin{equation} \label{eq:dmLambda}
			m(\Lambda):=\max_{\bsnu\in\Lambda}|\bsnu|,\qquad
			d(\Lambda):=\max_{\bsnu\in\Lambda}|\set{j}{\nu_j\neq 0}|.
		\end{equation}
	\end{definition}
	
	\begin{proposition}\label{prop:truncate}
		Let $r,s,t>0$, 
    let $\calG$ satisfy Assumption \ref{ass:G}, and let $\delta>0$ be given.
		Let $u_\bsnu^s$ be the Hermite coefficients of $u^s$ as in \eqref{eq:unus}.

		Then, there exists
		$C=C(r,s,t, \delta) < \infty$ such that for every $N\in\N$ there exist nested
		and downward closed multiindex sets\footnote{Only finitely many sets 
		in the sequence $\{ \Lambda_j^N \}_{j\geq 1}$ are nonempty 
                (see Remark \ref{remark:index-sets} below).}
		\begin{equation}\label{eq:mddef}
			\calF\supseteq\Lambda_1^N\supseteq\Lambda_2^N\supseteq\dots\qquad
			\text{s.t.}\qquad \sum_{j\in\N}|\Lambda_j^N|=N
		\end{equation}
		and such that 
		\begin{equation}\label{eq:tailsum}
			\sum_{j\in\N}\sum_{\bsnu\in\calF\backslash\Lambda_j^N} {\Gamma_\bsnu} |\dup{u_\bsnu^s}{\tilde\eta_j}_{\calY}|^2
			\le C N^{-2\min\{r,s,t\}}.
		\end{equation}
    In addition, as $N\to\infty$, 
		\begin{equation}\label{eq:md}
			m(\Lambda_1^N)= O(N^{\delta}),\quad d(\Lambda_1^N)=o(\log(N))
			\;.
		\end{equation}
	\end{proposition}
	\begin{proof}
		We proceed in two steps. 
                In the first we show the convergence rate bounds \eqref{eq:tailsum};
		in the second we show the bounds \eqref{eq:md} on the number of active coefficients.

		{\bf Step 1.} 
		We simplify notation and do not explicitly denote $s$-dependence of the eigenvalues by writing 
    $\bslambda \coloneqq \bslambda(s)$. 
    We do the same for the dependence on $r$ and $s$ of $\Gamma_\bsnu$ (compare \eqref{eq:Gamma}), 
    which is also suppressed in the notation here for convenience.

		According to Proposition \ref{prop:holomorphy},
		$u^s$ is $(\bslambda^{1/2},\theta,2, \epsilon , \calY_\C^{t})$-holomorphic. 
		By Theorem \ref{thm:bdXSum}, 
		for any integer $M \geq 2s +1 $ 
                it holds that
		\begin{equation*}
			\sum_{\bsnu\in\calF} 
                        \Gamma_\bsnu c_\bsnu \norm[\calY_\C^{t}]{u^s_\bsnu}^2 < \infty,
		\end{equation*}
		where the coefficients 
		$c_\bsnu = c_\bsnu(\bsw,r,s,M)$ are defined in \eqref{eq:cnudef}. 
		Thus
		\begin{equation}\label{eq:C0}
			C_0 :=  \sum_{\bsnu\in\calF}\sum_{j\in\N} 
                                 \Gamma_\bsnu v_j^{-2t}c_\bsnu |\dup{u_\bsnu^s}{\tilde\eta_j}_{\calY}|^2 < \infty.
		\end{equation}
		Let now $(\bsnu_k,i_k)_{k\in\N}$ be an enumeration of $\calF\times\N$
		such that
		\begin{enumerate}
			\item $(c_{\bsnu_k} v_{i_k}^{-2t})_{k\in\N}$ is a monotonically increasing sequence,
			\item the set $\Lambda_j^N:=\set{\bsnu_k}{k\le N,~i_k=j}$ is downward
			closed for all $j\in\N$, $N\in\N$,
			\item it holds that $\Lambda_{j}^N\supseteq\Lambda_{j+1}^N$ for all $j\in\N$, $N\in\N$.
		\end{enumerate}
		Such an enumeration exists
			because $(c_\bsnu)_{\bsnu\in\calF}$ is a monotonically
			increasing sequence (in the sense of \eqref{eq:moninc}), and
			$({v}_i^{-2t})_{i\in\N}$ is also a monotonically increasing sequence.

		Observe that also
		\begin{equation*}
			\sum_{j\in\N}|\Lambda_j^N|=\sum_{j\in\N}|\set{\bsnu_k}{k\le N,~i_k=j}|
			= 
                        |\set{(\bsnu_k,i_k)}{k\le N}|=N.
		\end{equation*}
		By Assumption \ref{ass:vw} on the weight sequence $\bsv$, 
		we have $({v}_j^{2t})_{j\in\N}\in\ell^{\frac{1}{2t}}(\N)$.
		
		Moreover, \eqref{eq:bdXsum1} gives
		\begin{equation*}
			(c_\bsnu^{-1})_{\bsnu\in\calF}\in\ell^{\frac{1}{2\min\{r,s\}}}(\calF).
		\end{equation*}
		Therefore
		\begin{equation*}
			({v}_j^{2t}c_\bsnu^{-1})_{\bsnu\in\calF,j\in\N}\in
			\ell^{\max \{\frac{1}{2r} ,\frac{1}{2s},\frac{1}{2t}\}}(\calF\times \N).
		\end{equation*}
		Define 
		$$
    q:=\max \left\{   \frac{1}{2r}, \frac{1}{2s},\frac{1}{2t}\right\} = \frac{1}{2 {\min}\{r,s,t\}}.
                $$
		Since $k\mapsto ({v}_{i_k}^{2t}c_{\bsnu_k}^{-1})^q$ is monotonically decreasing, 
                for any $k\in\N$
		\begin{equation*}
			k(v_{i_k}^{2t}c_{\bsnu_k}^{-1})^q 
                        \le \sum_{l=1}^k({v}_{i_l}^{2t}c_{\bsnu_l}^{-1})^q
			\le \sum_{l\in\N}({v}_{i_l}^{2t}c_{\bsnu_l}^{-1})^q=:C_1<\infty.
		\end{equation*}
		Therefore, for all $k\in\N$
		\begin{equation}\label{eq:C1}
			v_{i_k}^{2t}c_{\bsnu_k}^{-1}\le C_1^{\frac1q}k^{-\frac1q} 
			\;.
		\end{equation}
		This allows us to conclude that for any $N\in\N$
		\begin{align*}
			\sum_{j\in\N}\sum_{\bsnu\in\calF\backslash\Lambda_{j}^N}
                              {\Gamma_\bsnu}|\dup{u_\bsnu^s}{\tilde\eta_j}_{\calY}|^2
			&= \sum_{k>N}\Gamma_\bsnu |\dup{u_{\bsnu_k}^s}{\tilde\eta_{i_k}}_{\calY}|^2 
			c_{\bsnu_k}c_{\bsnu_k}^{-1}{v}_{i_k}^{-2t}{v}_{i_k}^{2t}\nonumber\\
			&\le \sup_{k>N}(c_{\bsnu_k}^{-1}{v}_{i_k}^{2t}) \sum_{j\in\N,\bsnu\in\calF} 
			\Gamma_\bsnu |\dup{u_{\bsnu}^s}{\tilde\eta_{j}}_\calY|^2c_{\bsnu}{v}_{j}^{-2t}.
		\end{align*}
		With \eqref{eq:C0} and \eqref{eq:C1} we obtain that 
    there exists $C = C(r,s,t, M) > 0$ such that for all $N\in \N$
		\begin{equation*}
			\sum_{j\in\N}\sum_{\bsnu\in\calF\setminus\Lambda_{j}^N} 
			\Gamma_\bsnu |\dup{u_\bsnu^s}{\tilde\eta_j}_\calY|^2
			\le 
			C N^{-2\min\{r,s,t\}}.
		\end{equation*}

		{\bf Step 2.} 
		It remains to show \eqref{eq:md}.
		We construct a sequence $(\hat c_\bsnu)_{\bsnu\in\calF}$ 
                similar to \eqref{eq:cnudef} based on \cite[Eqn. (6.13)]{DNSZ23_2957} 
                in which we now choose $k \coloneqq 1, r-\tau \coloneqq M$ and $\varrho_j \coloneqq w_j^{-2s}$:
		\begin{equation*}
			\hat c_{\bsnu} = \prod_{j\in\supp\bsnu}\max\{1,K w_j^{-2s} \}^2\nu_j^{M}, \qquad \bsnu\in\calF.
		\end{equation*}
		Let $(k_j)_{j\in\N}\subseteq\N$ be the monotonically increasing sequence satisfying
		\begin{equation*}
			(k_j)_{j\in\N} = \set{k_j}{i_{k_j}=1,~j\in\N}.
		\end{equation*}
		Then by definition of $(\bsnu_k,i_k)_{k\in\N}$ it is clear that
		$(\hat c_{\bsnu_{k_j}})_{j\in\N}$ is a monotonically increasing
		rearrangement of $(\hat c_\bsnu)_{\bsnu\in\calF}$.
		Moreover
		\begin{equation*}
			\Lambda_1^N = \set{\bsnu_{k_j}}{k_j\le N}
			\subseteq   \set{\bsnu_{k_j}}{j\le N}=: \tilde\Lambda_1^N.
		\end{equation*}
		For sets of type $\tilde\Lambda_1^N$, the quantities in \eqref{eq:dmLambda}
		were bounded in \cite[Lemma 4.3]{MSZ25_1135}. 
                Note that item $(i)$ of \cite[Lemma 4.3]{MSZ25_1135} is satisfied due to Assumption \ref{ass:vw}.
		Applying this result yields

		\begin{equation*}
			m(\Lambda_1^N)\le
			m(\tilde\Lambda_1^N)=O(N^{2s/M})\qquad\text{and}
			\qquad
			d(\Lambda_1^N)\le
			d(\tilde\Lambda_1^N)=o(\log(N))
		\end{equation*}
		as $N\to\infty$. 
		Since $M \geq 2s+1$ can be chosen arbitrary large, this concludes the proof.
	\end{proof}

		The following result can be shown analogously to Proposition \ref{prop:truncate} 
                by invoking Corollary \ref{cor:bdXSum-L2} instead of Theorem \ref{thm:bdXSum}.  

	\begin{corollary}\label{cor:tailSumBoundL2}
		Consider the same setting as in Proposition \ref{prop:truncate}. 
        Then, there exists $C = C(s,t, \delta ) < \infty$ such that for every $N \in \N$ 
        there exists nested and downward closed multiindex sets
		\[
			\calF\supseteq\Lambda_1^N\supseteq\Lambda_2^N\supseteq\dots\qquad
			\text{s.t.}\qquad \sum_{j\in\N}|\Lambda_j^N|=N
		\]
	and such that 
	\[
	\sum_{j\in\N} \sum_{\bsnu \in \calF\backslash\Lambda_j^N} |\langle u_\bsnu^s, \tilde{\eta}_j \rangle_\calY|^2 
        \leq C N^{-2\min(s,t)}.
	\]
	Furthermore,
        as $N \to \infty$,
	\[
		m(\Lambda_1^N) = O(N^\delta), \quad d(\Lambda_1^N) = o(\log(N)).
	\]
	\end{corollary}

  \begin{remark}
   \label{remark:index-sets} 
   For later reference, we collect
   properties of the index sets of Proposition
   \ref{prop:truncate} and Corollary \ref{cor:tailSumBoundL2}.
   By definition, these sets are nested, so that
   \begin{equation*}
     \supp\bigcup_{j=1}^N\Lambda^N_j = \supp\Lambda^N_1.
   \end{equation*}
   We also always have $\bszero \in \Lambda_1^N$ and, 
   whenever $i\in \supp\Lambda^N_1$, 
   then $\bsdelta_i \coloneqq (\delta_{ij})_{j\in\N}\in
   \Lambda^N_1$ by downward closedness. 
   It follows that
   \begin{equation*}
     \left|\supp\Lambda^N_1\right| + 1 \leq |\Lambda^N_1 | \leq N.
   \end{equation*}
   From the reasoning above and since the cardinalities of the index sets sum to $N$, one also obtains that
   \begin{equation*}
     \Lambda_j^N = \emptyset \quad \text{for all }j>N-\left|\supp\Lambda^N_1\right|.
   \end{equation*}
  \end{remark}

	Since the analysis in this section is carried out on mappings of the form 
        $u \colon \R^\N \to \calY$, 
        we define Gaussian-Sobolev norms for such mappings
        based on the definition of the norms for operators
        \eqref{eq:SobolevNorm}. 
Let $T_s \colon \R^\N \to \calX$ with $T_s(\bsy) \coloneqq \sum_{j\in\N}y_j \sqrt{\lambda_j(s)}\phi_j$. 
Observe that $\gamma_s = (T_s)_\# \bsmu$ (see Remark \ref{rmk:KL}). 
We set 
	\begin{equation*}
		\|u\|_{W_{\bsmu, r, s}^{1,2}(\R^\N,\calY)} \coloneqq \|u \circ T_s^{-1}\|_{W_{\gamma_s,r}^{1,2}(\calX,\calY)}.
	\end{equation*}
	The proof of the following lemma can be found in Appendix \ref{app:AuxiliaryResults/Definitions}.
	\begin{lemma}\label{lemma:uGaussSobolevNorm}
	Let $u \colon \R^\N \to \calY$ such that $\|u\|_{W_{\bsmu,r,s}^{1,2}(\R^\N,\calY)} < \infty$. 
        Let $r,s \geq 0$ and let $\bsw = (w_j)_{j\in\N} \in (0,1]^\N$ be as in Assumption \ref{ass:vw}. 

Then 
		\begin{equation*}
			\|u\|_{W_{\bsmu,r,s}^{1,2}(\R^\N,\calY)}^2 
			=
			\int_{\R^\N} \|u(\bsy)\|_\calY^2 \, \dd\bsmu(\bsy) 
			+
			\int_{\R^\N} \sum_{j=1}^{\infty} w_j^{2(r-s)} \left\|\partial_{y_j} u(\bsy)\right\|_\calY^2 \, \dd\bsmu(\bsy).
		\end{equation*}
		Furthermore,  
		\begin{equation*}
			\|u\|_{W_{\bsmu,r,s}^{1,2}(\R^\N,\calY)}^2 = \sum_{\bsnu\in\calF} \Gamma_\bsnu(r,s) \|u_\bsnu\|_\calY^2
		\end{equation*}
		with $u_\bsnu$ as in \eqref{eq:unu} and $\Gamma_\bsnu(r,s)$ as in \eqref{eq:Gamma}.
	\end{lemma}
	
	As a first operator surrogate approximation rate result, 
	we give the following best $N$-term rate bound for 
	spectral operator surrogates realized by truncated frame-Hermite expansions.
	
	\begin{corollary}[Spectral Surrogate Convergence Rate]
		\label{cor:convergence-expansion}
		Consider the setting of Proposition \ref{prop:truncate} and Corollary \ref{cor:tailSumBoundL2}.
		Then
		there exist constants $C_1 = C_1(r,s,t,\bseta, \delta) > 0$ and $C_2 = C_2(s,t,\bseta, \delta ) > 0$
		such that for every $N\in\N$, 
		with the $N$-parametric, polynomial approximator maps $g_N :\R^\N\to\R^N$ given by 
		\[
                  g_N(\bsy) :=
                   \Big(\sum_{\bsnu\in \Lambda_j^N} \dup{u_\bsnu}{\tilde\eta_j}_\calY H_\bsnu(\bsy)\Big)_{j=1}^N
		\]
		it holds that
		\begin{equation*}
			\left\|u(\bsy)-\calD_N \circ g_N(\bsy)\right\|_{W_{\bsmu,r,s}^{1,2}(\R^\N, \calY)}
			= 
			\left\|u(\bsy)-\sum_{j=1}^N\eta_j (g_N)_j(\bsy)\right\|_{W_{\bsmu,r,s}^{1,2}(\R^\N, \calY)}
			\le C_1 N^{-\min\{r,s,t\}}
		\end{equation*}
		and
		\begin{equation*}
			\|u(\bsy) - \calD_N \circ g_N(\bsy)\|_{L_{\bsmu}^2(\R^\N,\calY)}
			=
			\left\| u(y) - \sum_{j=1}^N \eta_j (g_N)_j(\bsy) \right\|_{L_{\bsmu}^2(\R^\N,\calY)}
			\leq 
			C_2 \, N^{-\min\{s,t\}}.
		\end{equation*}
	\end{corollary}

	\begin{proof}
		Let $N \in \N$ and $\bsy \in \R^\N$.
		Consider the expansion \eqref{eq:uexpansion} of $u^s(\bsy)$, the definition
    of $g_N(\bsy)$ and note that $\Lambda_j^N = \emptyset$ for all $j > N$ since
    $\calF \supseteq \Lambda_1^N \supseteq \Lambda_2^N \supseteq \dots $ and
    $\sum_{j\in\N} |\Lambda_j^N| = N$. 
    Note that the latter property holds true
    for both, the index sets from Proposition \ref{prop:truncate} and those in Corollary \ref{cor:tailSumBoundL2}; 
    in this proof we do not distinguish between these index sets in notation. This implies
		\begin{align*}
			u^s(\bsy) - \calD_N \circ g_N(\bsy)
			&=
			\sum_{j\in\N} \eta_j \sum_{\bsnu\in\calF} \langle u_\bsnu^s, \tilde\eta_j \rangle_\calY H_\bsnu(\bsy)
			-
			\sum_{j=1}^N \eta_j \sum_{\bsnu\in\Lambda_j^N} \langle u_\bsnu^s,\tilde\eta_j \rangle_\calY H_\bsnu(\bsy)	\\
			&=
			\sum_{j=1}^N \eta_j \sum_{\calF\backslash\Lambda_j^N} \langle u_\bsnu^s, \tilde \eta_j \rangle_\calY H_\bsnu(\bsy)
			+ 
			\sum_{j > N} \eta_j \sum_{\bsnu\in\calF} \langle u_\bsnu^s,\tilde\eta_j \rangle_\calY H_\bsnu(\bsy) \\
			&=
			\sum_{j\in\N} \eta_j \sum_{\bsnu\in\calF\backslash\Lambda_j^N} \langle u_\bsnu^s,\tilde\eta_j \rangle_\calY H_\bsnu(\bsy).
		\end{align*}
	In the following, 
let $\overline{\kappa} \coloneqq \overline{\kappa}(\bseta) > 0$ 
be the upper frame constant of the frame $\bseta = (\eta_j)_{j\in\N} \subset \calY$. 
\par
	First consider the error estimate w.r.t.~the norm $\|\cdot\|_{W_{\bsmu,r,s}^{1,2}(\R^\N, \calY)}$. 
The identity \eqref{eq:weightedParseval} and Proposition \ref{prop:truncate} 
with constant $0 < \hat{C} = \hat{C}(r,s,t,\delta ) < \infty$ imply 
		\begin{align*}
			\|u^s(\bsy) - \calD_N \circ g_N(\bsy)\|_{W_{\bsmu,r,s}^{1,2}(\R^\N,\calY)}^2
			&=
			\left\| \sum_{j\in\N}\eta_j\sum_{\bsnu\in\calF\backslash\Lambda_j^N} \langle u_\bsnu^s, \tilde\eta_j \rangle_\calY H_\bsnu(\bsy) \right\|_{W_{\bsmu,r,s}^{1,2}(\R^\N,\calY)}^2 \\
			&\leq
			\overline{\kappa}^2 \sum_{j\in\N}\sum_{\bsnu\in\calF\backslash\Lambda_j^N} \Gamma_\bsnu |\langle u_\bsnu^s,\tilde\eta_j \rangle_\calY|^2 \\
			&\leq
			\overline{\kappa}^2 \hat{C} N^{-2\min\{r,s,t\}}. 
		\end{align*}
		 
We proceed similarly for the error estimate in $L_\bsmu^2(\R^\N,\calY)$ 
and now denote by $\hat{C} = \hat{C}(s,t,\delta)$ 
the constant from Corollary \ref{cor:tailSumBoundL2}. 
Parseval's identity \eqref{eq:Parsev}, 
Corollary \ref{cor:tailSumBoundL2}, 
and the frame bound \eqref{eq:rmk:FrBd} imply
		\begin{align*}
			\|u^s(\bsy) - \calD_N \circ g_N(\bsy)\|_{L_\bsmu^2(\R^\N,\calY)}^2 
			&=
			\left\| \sum_{j\in\N}\eta_j\sum_{\bsnu\in\calF\backslash\Lambda_j^N} 
                        \langle u_\bsnu^s,\tilde\eta_j \rangle_\calY H_\bsnu(\bsy) \right\|_{L_\bsmu^2(\R^\N,\calY)}^2 
                        \\
			&\leq 
			\overline{\kappa}^2 \sum_{j\in\N}\sum_{\bsnu\in\calF\backslash\Lambda_j^N} 
                        |\langle u_\bsnu^s,\tilde\eta_j \rangle_\calY|^2 \\
			&\leq 
			\overline{\kappa}^2 \hat{C} N^{-2\min\{s,t\}},
		\end{align*}
where we used $n=N$ in \eqref{eq:decoder}.
	\end{proof}

	\subsection{DNN Expression Rates}
	\label{sec:PfThm}
	The polynomial parametric operator surrogates $g_N$ in Cor.~\ref{cor:convergence-expansion}
		imply the existence of corresponding neural operator surrogates, 
		also referred to as operator networks, or `ONets'. 
		A first result follows directly from Cor.~\ref{cor:convergence-expansion} 
		with the observation that $\RePU$ activation functions
    allow a mathematically
		exact emulation of tensorized polynomials. 
		This observation was used recently e.g. in \cite{ReQuOnet} in the case of
		bounded parameter domains.
	
	We next address strict $\ReLU$ neural emulators.
	These will be constructed from the spectral surrogates 
	$g_N$ in Cor.~\ref{cor:convergence-expansion}, 
        by expressing univariate Hermite polynomials and their products.
	A key ingredient is the following adaptation of a theorem from
	\cite[Theorem 3.9]{SZ21_982}. 
        It provides an approximation result for univariate Hermite polynomials 
        by ReLU feedforward networks. 
	A proof along the lines of the proof of \cite[Theorem 3.9]{SZ21_982} 
	is detailed in Appendix \ref{app:DNNEmulationHermite}.
        The result shown here is stronger than \cite{SZ21_982} in that we prove 
        the same rates as shown in \cite{SZ21_982} under normalization of NN weights.

	We recall that for $\bsy\in\R^\N$, 
        the multivariate Hermite polynomial 
        $H_\bsnu(\bsy)=\prod_{i\in\N}H_{\nu_i}(y_i)=\prod_{i\in{\rm supp}(\bsnu)}H_{\nu_i}(y_i)$
        since $H_0\equiv 1$. 
        In particular, $H_\bsnu$ only depends on the input variables with indices in $\supp \bsnu$.
    Similarly, with a slight abuse of notation, the network approximations
    $\widetilde{H}_{\varepsilon, \bsnu}$ in the following proposition
    can be interpreted as maps from $\bsy\in \R^\N$ to $\R$, that effectively only depend on $(y_i)_{i\in{\supp}(\bsnu)}$.

	\begin{proposition}\label{prop:reluhermite}
		Let $\Lambda\subseteq\calF$ be finite and downward closed. Then
		for every $\eps\in (0,1/2 )$ there exists a ReLU neural network
		$\Phi=\{\tilde
		H_{\eps,\bsnu}\}_{\bsnu\in\Lambda}:\R^{\left|\supp\Lambda\right|}\to\R^{|\Lambda|}$
		depending solely on the variables $(y_i)_{i\in\supp\Lambda}$,
		such that 
		\begin{equation*}
			\max_{\bsnu\in\Lambda} \;\;
			\norm[H^1(\R^\N;\bsmu)]{H_\bsnu-\tilde H_{\eps,\bsnu}}\le \eps,
		\end{equation*}
		and there exists a positive constant $C$ (independent of
		$m(\Lambda)$, $d(\Lambda)$ and of $\eps\in (0,{1/2})$) such
		that
		\begin{align*}
			\size(\Phi)&\le C |\Lambda|m(\Lambda)^3\log(1+m(\Lambda))^{2}d(\Lambda)^{3}\log(\eps^{-1})^{2}{\log(1+d(\Lambda))^{3}}\log(m(\Lambda)d(\Lambda)),
			\\
			\depth(\Phi)
			&\le C m(\Lambda)^{3}\log(1+m(\Lambda))^{2}d(\Lambda)^{2}\log(1+d(\Lambda))^{3}\log(\eps^{-1})^{2}\log(m(\Lambda)d(\Lambda)),
			\\
			\mpar(\Phi) &\le 1.
		\end{align*}
	\end{proposition}
	
	The next result will imply Theorems \ref{thm:main} and \ref{thm:main2}.
	
	\begin{theorem}
		\label{thm:NN-reapprox}
		Let $r$, $s$, $t > 0$, 
                let $\calG$ satisfy Assumption \ref{ass:G}, and let $q \in \N$.
		Let $u^s$ be 
		as in \eqref{eq:us}.
		
		Let 
                $\delta > 0$ be given.
                Then there exists constants 
                $B_q>0$ and $C=C(r,s,t,\calG,\delta) > 0$, 
		such that 
		for every $N$ there exists a $\sigma_q$-NN 
		$$\tilde{g}_N\in \bmg_{FN}(\sigma_q,\depthdelta_N , \widthdelta_N ,\sizedelta_N , B_q)$$
		with $\depth_N^{{\delta}}, \, \widthdelta_N , \, \size_N^{\delta}$ 
                as defined in \eqref{eq:sizeDepthN}, 
                and 
		such that,
		\begin{equation}\label{eq:us-DGN}
			\norm[{W_{\bsmu, r,s}^{1,2}(\R^\N;\calY)}]{u^s-\calD_N \circ \tilde{g}_N} 
                \le C N^{-\min\{r,s,t\}}
		\end{equation}
		and
		\begin{equation}\label{eq:us-DGN2}
			\norm[L_{\bsmu}^2(\R^\N;\calY)]{u^s-\calD_N \circ \tilde{g}_N}\le C N^{-\min\{s,t\}}.
		\end{equation}
		Furthermore, we have $B_1 = 1$, i.e., for $q=1$ (ReLU)
                the network weights are bounded by $1$.
	\end{theorem}
	\begin{proof}
		According to Proposition \ref{prop:holomorphy},
		the function $\bsy\mapsto u^s(\bsy)$ defined in
		\eqref{eq:us} is $(\bslambda^{1/2},\theta,2,\epsilon,\calY_\C^{t})$-holomorphic.
		Hence, there is a parametric map such that 
		\begin{equation}\label{eq:usexpansion}
			u^s(\bsy) = \calD \circ g(\bsy) 
			:= \sum_{j\in\N}\eta_j \sum_{\bsnu\in\calF}\dup{u_\bsnu^{s}}{\tilde\eta_j}_\calY H_\bsnu(\bsy)
			\in \calY
		\end{equation}
		with convergence in $L_\bsmu^2(\R^\N;\calY)$.
		Throughout the remainder of this proof, 
		for any $N\in \bbN$, 
		let for $j=1,2,...$ the sets $\Lambda_j^N\subset \calF$ 
		be as in Proposition \ref{prop:truncate}, so that
		\begin{equation}\label{eq:truncation}
			\sum_{j\in\N}\sum_{\bsnu\notin\Lambda_j^N} {\Gamma_\bsnu} 
                        |\dup{u_\bsnu^s}{\tilde\eta_j}_\calY|^2
			\le 
			C N^{-2\min\{r,s,t\}}.
		\end{equation}
		We use the same notation to denote the index sets from 
                Corollary \ref{cor:tailSumBoundL2} and make use of the property
		\begin{equation}\label{eq:truncationL2}
			\sum_{j\in\N}\sum_{\bsnu \notin \Lambda_j^N} |\dup{u_\bsnu^s}{\tilde \eta_j}_\calY|^2 
                        \leq C N^{-\min(s,t)}.
		\end{equation}
    It will be clear from context which family of index sets 
    is meant since \eqref{eq:truncation} is used to derive the approximation result in 
    $W_{\bsmu,r,s}^{1,2}(\R^\N;\calY)$ whereas \eqref{eq:truncationL2} 
    is used to derive the corresponding result in $L_\bsmu^2(\R^\N;\calY)$.
    Recall from Remark \ref{remark:index-sets} that, 
    for both families of index sets, 
    $\Lambda_j^N = \emptyset$ for $j>N$ and 
    the cardinality of $\supp\Lambda^N_1$ is bounded by $N$. 
     Without loss of generality we assume therefore for the rest of this proof that
		\begin{equation}\label{eq:suppbsnu}
			\supp(\bsnu)\subseteq\{1,\dots,N\}\qquad\forall\,\bsnu\in\Lambda^N_j,~j\in\N.
		\end{equation}
    In what follows we
    prove the claimed result separately for $q = 1$ (ReLU NNs) and $q>1$ (RePU NNs). 
    The first three steps address the former, the last one addresses the latter.
\par
		{\bf Step 1.} 
		We construct the network $\tilde{g}_N$ by emulating with $\ReLU$ NNs 
                the Hermite polynomials $H_\bsnu(\bsy)$ 
		in the spectral operator approximations of Cor.~\ref{cor:convergence-expansion}.
		
		Let $\alpha$ be either $\min\{s,t\}$ or $\min\{r,s,t\}$. Given $N\in \N$, define
		\begin{equation}\label{eq:defeps}
			\eps(N,\alpha) := N^{-\frac{1}{2}-\alpha}
		\end{equation}  
		By Prop.~\ref{prop:reluhermite}, 
		there exists a constant $C>0$ such that, for every $N\in \N$, 
                there exists a ReLU NN
		$\Phi_{\eps(N,\alpha)} \coloneq \{\tilde H_{\eps(N,\alpha),\bsnu}\}_{\bsnu\in\Lambda_{1}^N}$ 
		satisfying the error bound 
		\begin{align}\label{eq:HermReLU}
			\norm[H^1(\R^\N;\bsmu)]{H_\bsnu - \tilde H_{\eps(N,\alpha), \bsnu}}&\le\eps(N,\alpha) \qquad \forall \, \bsnu \in \Lambda_1^N,
		\end{align}
		and the size, depth, and weight bounds
		\begin{align*}
			\size(\Phi_{\eps(N,\alpha)})
                        &\le C Nm(\Lambda_1^N)^3\log(1+m(\Lambda_1^N))^{2}
                        d(\Lambda_1^N)^{3}\log(\eps(N,\alpha)^{-1})^{2}{\log(1+d(\Lambda_1^N))^{3}}\log(m(\Lambda_1^N)d(\Lambda_1^N))
                 \\
			&\le C N^{1+3\delta}\log(N)^{11}
                 \\
			\depth(\Phi_{\eps(N,\alpha)})
			&\le C m(\Lambda_1^N)^{3}\log(1+m(\Lambda_1^N))^{2}d(\Lambda_1^N)^{2}
                          \log(1+d(\Lambda_1^N))^{3}\log(\eps(N,\alpha)^{-1})^{2}\log(m(\Lambda_1^N)d(\Lambda_1^N)) 
                 \\
                         &\le C N^{3\delta}\log(N)^{10}, 
                  \\
			\mpar(\Phi_{\eps(N,\alpha)}) 
                       &\le 1.
		\end{align*}
		Here we have used \eqref{eq:md} and \eqref{eq:defeps}.
                
    Possibly adjusting the value of the constant $\delta>0$
    we obtain the bounds claimed in \eqref{eq:sizeDepthN}.
		
		Using \eqref{eq:suppbsnu}, 
                we now introduce the ReLU surrogate map
		$\tilde{g}_N:\R^N\to\R^N$ such that $\bsy\in\R^\N$ is mapped to
		\begin{equation}\label{eq:defGN}
			\tilde{g}_N(\bsy) 
			= 
			\Big(\sum_{\bsnu\in\Lambda_{j}^N}\dup{u_\bsnu^s}{\tilde\eta_j}_\calY\tilde 
                                                    H_{\eps(N,\alpha),\bsnu}(\bsy)\Big)_{1\le j\le N}.
		\end{equation} 
		
		{\bf Step 2.} We bound the expression error of the NN approximation. 
		
		By definition of $\calD_N:\R^N\to\calY$ in Section~\ref{sec:decoder}
		and of the neural approximation $\tilde{g}_N$ of the spectral surrogate $g_N$, \eqref{eq:defGN},
    we get
    \begin{equation*}
      \calD_N\circ \tilde{g}_N(\bsy) 
      = \sum_{j=1}^N\eta_j \sum_{\bsnu\in\Lambda_j^N}\dup{u_\bsnu^s}{\tilde\eta_j}_\calY\tilde H_{\eps(N,\alpha),\bsnu}(\bsy).
    \end{equation*}
    As $\Lambda_j^N = \emptyset$ for all $j > N$,
    \begin{equation}\label{eq:pointwiseError}
      \begin{aligned}
        u^s(\bsy)-\calD_N\circ \tilde{g}_N(\bsy) 
        &=
          \sum_{j\in\N}\sum_{\bsnu\notin\Lambda_j^N}\dup{u_\bsnu^s}{\tilde\eta_j}\eta_j H_\bsnu(\bsy) \\  
        &\qquad +\sum_{j=1}^N\sum_{\bsnu\in\Lambda_j^N}\dup{u_\bsnu^s}{\tilde\eta_j}\eta_j (H_\bsnu(\bsy)-\tilde H_{\eps(N,\alpha),\bsnu}(\bsy))\;.
      \end{aligned}
    \end{equation}
    We can bound the norm of the first term in \eqref{eq:pointwiseError} 
    using Proposition \ref{prop:truncate}, Corollary \ref{cor:tailSumBoundL2}, 
    and Lemma \ref{lemma:uGaussSobolevNorm},
      \begin{align}
        \norm[W_{\bsmu,r,s}^{1,2}(\R^\N;\calY)]{\sum_{j\in\N}\sum_{\bsnu\notin\Lambda_j^N}\dup{u_\bsnu^s}{\tilde \eta_j}_\calY \eta_j H_\bsnu}^2
        &\le \bar\kappa(\bseta)
          \sum_{j\in\N}\sum_{\bsnu\notin\Lambda_j^N} \Gamma_\bsnu|\dup{u_\bsnu^s}{\tilde \eta_j}_\calY|^2 
          \le
          C N^{-2\min\{r,s,t\}},\label{eq:truncationErrorSobolev} \\
        \norm[L_\bsmu^2(\R^\N;\calY)]{\sum_{j\in\N}\sum_{\bsnu\notin\Lambda_j^N}\dup{u_\bsnu^s}{\tilde \eta_j}_\calY \eta_j H_\bsnu}^2
        &\le \bar\kappa(\bseta)
          \sum_{j\in\N}\sum_{\bsnu\notin\Lambda_j^N} |\dup{u_\bsnu^s}{\tilde \eta_j}_\calY|^2 
          \le
          C N^{-2\min\{s,t\}},\label{eq:truncationErrorL2}
      \end{align}
    where $\bar\kappa(\bseta)$ is the upper frame constant of $\bseta$
    and the sets $\Lambda_j^N$ are not necessarily the same between
    \eqref{eq:truncationErrorSobolev} and \eqref{eq:truncationErrorL2}.
    For ease of notation, define 
    	\begin{equation*}
    	f(\bsy) \coloneqq \sum_{j=1}^{N}\eta_j \sum_{\bsnu\in\Lambda_j^N} 
        \dup{u_\bsnu^s}{\tilde\eta_j}_\calY (H_\bsnu(\bsy) - \tilde H_{\eps(N,\alpha), \bsnu}(\bsy)).
    	\end{equation*}
    	
        Then
    	\begin{align*}
    		\norm[\calY]{f(\bsy)}^2 
    		&\le
    		\overline{\kappa}(\bseta)^2 \sum_{j=1}^N \left|\sum_{\bsnu\in\Lambda_j^N} \dup{u_\bsnu^s}{\tilde\eta_j}_\calY(H_\bsnu(\bsy)-\tilde H_{\eps(N,\alpha),\bsnu}(\bsy))\right|^2 \\
    		&\le
    		\overline{\kappa}(\bseta)^2 \sum_{j=1}^N \left(\sum_{\bsnu\in\Lambda_j^N}|\dup{u_\bsnu^s}{\tilde\eta_j}_\calY|^2\sum_{\bsnu\in\Lambda_j^N}|H_\bsnu(\bsy)-\tilde H_{\eps(N,\alpha),\bsnu}(\bsy)|^2\right).
    	\end{align*}
    	This implies 
    	\begin{equation}\label{eq:L2TermError}
    		\int_{\R^\N} \norm[\calY]{f(\bsy)}^2 \, \dd\bsmu(\bsy)
    		\lesssim
    		\sum_{j=1}^N \sum_{\bsnu\in\Lambda_j^N} |\dup{u_\bsnu^s}{\tilde\eta_j}_\calY|^2 \eps(N,\alpha)^2 |\Lambda_j^N|
    		\lesssim
    		\eps(N,\alpha)^2 N \sum_{j\in\N}\sum_{\bsnu\in\calF} |\dup{u_\bsnu^s}{\tilde\eta_j}_\calY|^2.
    	\end{equation}
    	Note that $\sum_{j\in\N}\sum_{\bsnu\in\calF} |\dup{u_\bsnu^s}{\tilde\eta_j}_\calY|^2 < \infty$. 
        Now \eqref{eq:truncationErrorL2} and $\alpha = \min(s,t)$ imply \eqref{eq:us-DGN2}.\par
   In order to show \eqref{eq:us-DGN} we need to bound a (weighted) $L^2$-norm of derivatives of $f$ 
in addition to the already established bound \eqref{eq:L2TermError} (see Lemma \ref{lemma:uGaussSobolevNorm}). 
In this part of the proof we fix the choice $\alpha = \min(r,s,t)$ and 
shorten notation by $\eps(N) := \eps(N,\min(r,s,t))$.

First we 
observe that $\partial_i f\equiv 0$ whenever $f$ does not depend on $y_i$; 
in particular, $\partial_i f\equiv 0$ for all $i$ not belonging 
to the finite set of ``active dimensions''
  \begin{equation*}
    \set{i\in\N}{\exists j,~\bsnu\in\Lambda^N_j\text{ s.t.\ }\nu_i>0} = \supp(\Lambda_1^N).
  \end{equation*}
For $i$ belonging to this set, we get for all $\bsy\in \R^\N$
    	\begin{align*}
    		\norm[\calY]{\partial_i f(\bsy)}^2
    		&=
    		\normlr[\calY]{\sum_{j=1}^N \eta_j \sum_{\bsnu\in\Lambda_j^N}\dup{u_\bsnu^s}{\tilde\eta_j}_\calY \partial_i [H_\bsnu(\bsy)-\tilde H_{\eps(N),\bsnu}(\bsy)]}^2 \\
    		&\leq 
    		\overline{\kappa}(\bseta)^2 \sum_{j=1}^{N} \left|\sum_{\bsnu\in\Lambda_j^N}\dup{u_\bsnu^s}{\tilde\eta_j}_\calY \partial_i [H_\bsnu(\bsy)-\tilde H_{\eps(N),\bsnu}(\bsy)]\right|^2 \\
    		&\leq 
    		\overline{\kappa}(\bseta)^2 \sum_{j=1}^N \left(\sum_{\bsnu\in\Lambda_j^N}|\dup{u_\bsnu^s}{\tilde\eta_j}_\calY|^2\sum_{\bsnu\in\Lambda_j^N}|\partial_i[H_\bsnu(\bsy)-\tilde H_{\eps(N),\bsnu}(\bsy)]|^2\right).
    	\end{align*}
    	This implies 
    	\begin{align*}
    		\int_{\R^\N} \sum_{i=1}^\infty w_i^{2(r-s)} \norm[\calY]{\partial_i f(\bsy)}^2 \, \dd\bsmu(\bsy) 
    		&\lesssim
    		\sum_{j=1}^N \sum_{\bsnu\in\Lambda_j^N} \sum_{{i\in\supp\bsnu}} w_i^{2(r-s)} |\dup{u_\bsnu^s}{\tilde\eta_j}_\calY|^2 \eps(N)^2 |\Lambda_j^N| \\
    		&\lesssim
    		\sum_{j\in\N} \sum_{\bsnu\in\calF} \sum_{{i\in\supp\bsnu}} w_i^{2(r-s)} |\dup{u_\bsnu^s}{\tilde\eta_j}_\calY|^2 \eps(N)^2 N \\
    		&\lesssim
    		\sum_{j\in\N}\sum_{\bsnu\in\calF} \Gamma_\bsnu(r,s) |\dup{u_\bsnu^s}{\tilde\eta_j}_\calY|^2 \eps(N)^2N.
    	\end{align*}
    	Since $\|u^s\|_{W_{\bsmu,r,s}^{1,2}(\R^\N,\calY)} < \infty$ 
        we have by Lemma \ref{lemma:uGaussSobolevNorm}
    	\begin{align*}
    		\|u^s\|_{W_{\bsmu,r,s}^{1,2}(\R^\N,\calY)}^2 
    		=
    		\sum_{\bsnu\in\calF}\Gamma_\bsnu \|\sum_{j\in\N}\eta_j \langle u_\bsnu^s, \tilde \eta_j \rangle_\calY\|_\calY^2
    		\geq 
    		\underline{\kappa}(\bseta)^2 \sum_{\bsnu\in\calF}\sum_{j\in\N}\Gamma_\bsnu |\langle u_\bsnu^s, \tilde \eta_j \rangle_\calY|^2
    	\end{align*}
    	with lower frame constant $\underline{\kappa}(\eta)$ and hence
        $\sum_{j\in\N}\sum_{\bsnu\in\calF}\Gamma_\bsnu |\dup{u_\bsnu^s}{\tilde\eta_j}_\calY|^2 < \infty$. 
        Therefore, our choice $\eps(N) = N^{-\frac{1}{2}-\min(r,s,t)}$ in the above calculation, 
        together with \eqref{eq:L2TermError} (with $\alpha = \min(r,s,t)$), 
        and \eqref{eq:truncationErrorSobolev} imply \eqref{eq:us-DGN}.

		{\bf Step 3.} 
		We bound the size, depth, and the magnitude of the weights of $\tilde{g}_N$ in \eqref{eq:defGN}. 
		By definition, the ReLU NN $\tilde{g}_N$ is a composition of
		\begin{equation}
			\Phi_{{\eps(N,\alpha)}} : \R^N\to \R^{|\Lambda_1^N|}
		\end{equation}
		with the linear transformation $A:\R^{|\Lambda_1^N|}\to\R^N$
		\begin{equation*}
			A_{\bsnu,j} =\begin{cases}
				\dup{u_\bsnu^s}{\tilde\eta_j}_\calY &\bsnu\in\Lambda_j^N\\
				0 &\text{otherwise}
			\end{cases}\qquad\qquad\forall\,\bsnu\in\Lambda_1^N,~1\le j\le N.
		\end{equation*}
		We derive a possible realization of $\tilde g_N$ as a ReLU NN with
    $\mpar(\tilde g_N) \leq 1$. 
    Let
    \begin{equation*}
     a := \sup\{\dup{u_\bsnu^s}{\tilde\eta_j} : \bsnu\in\calF, \, j\in\N\} < \infty.
    \end{equation*}
    W.l.o.g. assume $a > 1$. 
    We set $\tilde A := a^{-1}A \colon
    \R^{|\Lambda_1^N|} \to \R^N$ and denote by $\phi_a \colon \R^N \to \R^N$ a
    ReLU NN as in Lemma \ref{lemma:scalarMultNetwork} satisfying
    $\phi_a(\bsx) = a\bsx$ for all $\bsx \in \R^N$ and $\mpar(\phi_a) \leq 1$. 
    We can now write $\tilde g_N = \phi_a \bullet \tilde A \bullet \Phi_{\eps(N,\alpha)}$ 
    in the sense of sparse concatenation of ReLU networks (see Definition \ref{def:sparseConc}). 
    From this we obtain a realization 
     $\tilde g_N = \phi_a \bullet \tilde A \bullet \Phi_{\eps(N,\alpha)}$ such that
		\begin{align*}
			\size(\tilde g_N) &\lesssim \size(\phi_a) + \size(\tilde A) + \size(\Phi_{\eps(N,\alpha)}) \lesssim N^{1+3\delta}\log(N)^{11},\\
			\depth(\tilde g_N) &\lesssim 1 + \depth(\phi_a) + \depth(\tilde A) + \depth(\Phi_{\eps(N,\alpha)}) \lesssim N^{3\delta}\log(N)^{10},\\
			\mpar(\tilde g_N) &\leq \max\{\mpar(\phi_a), \mpar(\tilde A), \mpar(\Phi_{\eps(N,\alpha)})\} \leq 1.
		\end{align*}
Up to a change of constants, 
this shows the claimed bounds on size and depth of the theorem. 
This completes the proof for the case of ReLU activation, i.e. $q=1$.

	{\bf Step 4.}
	We next consider $q\in\N, \, q \geq 2$ fixed. 
        Let $\Phi := \{\tilde H_\bsnu\}_{\bsnu \in \Lambda_1^N}$ be a
        neural network with activation $\sigma_q$ as in 
        Proposition \ref{prop:RePUMultHermite} such that
 		\begin{equation}\label{eq:exactRealizationHermite}
 			\tilde H_\bsnu(\bsy) 
                        = H_\bsnu(\bsy) \qquad \forall \, \bsnu \in \Lambda_1^N, \, \bsy\in\R^\N,
 		\end{equation}
 		and that for a constant $C = C(q, \delta) > 0$,
 		\begin{align*}
 			\size(\Phi) &\leq C |\Lambda_1^N| m(\Lambda_1^N)^2 d(\Lambda_1^N)
 			\leq C N^{1+2\delta}\log(N), \\
 			\depth(\Phi) &\leq C m(\Lambda_1^N)\log(d(\Lambda_1^N)+1) \leq C N^{\delta}\log(N), \\
 			\mpar(\Phi) &\leq C. 
 		\end{align*}
 		Next define the neural network 
 		\begin{equation*}
 			\tilde g_N(\bsy) := \left(\sum_{\bsnu\in\Lambda_j^N} \dup{u_\bsnu^s}{\tilde\eta_j}_\calY \tilde H_\bsnu(\bsy)\right)_{1\leq j\leq N}.
 		\end{equation*}
 		Compared to \eqref{eq:defGN},
                we replaced the ReLU neural networks $\tilde H_{\eps(N,\alpha),\bsnu}$ 
                with the neural networks $\tilde H_\bsnu$ with activation $\sigma_q$.
 		Due to \eqref{eq:exactRealizationHermite} 
                and $\Lambda_j^N = \emptyset$ for all $j > N$,
                we have for all $\bsy\in\R^\N$
 		\begin{align*}
 			u^s(\bsy) - \calD_N \circ \tilde g_N(\bsy)
 			&= 
 			\sum_{j=1}^N \eta_j \sum_{\bsnu \notin \Lambda_j^N} \dup{u_\bsnu^s}{\tilde\eta_j}_\calY H_\bsnu(\bsy)
 			+
 			\sum_{j=1}^N \eta_j \sum_{\bsnu\in\Lambda_j^N} \dup{u_\bsnu^s}{\tilde\eta_j}_\calY (H_\bsnu(\bsy) - \tilde H_\bsnu(\bsy)) \\
 			&+
 			\sum_{j>N} \eta_j \sum_{\bsnu\in\calF}\dup{u_\bsnu^s}{\tilde\eta_j}_\calY H_\bsnu(\bsy) \\
 			&=
 			\sum_{j\in\N} \eta_j \sum_{ \bsnu\in\calF \backslash \Lambda_j^N} \dup{u_\bsnu^s}{\tilde\eta_j}_\calY H_\bsnu(\bsy).
 		\end{align*}
 		By Proposition \ref{prop:truncate} and Corollary \ref{cor:tailSumBoundL2} we have
 		\begin{equation*}
 			\norm[L_\bsmu^2(\R^\N;\calY)]{u^s - \calD_N \circ \tilde g_N}^2
 			\lesssim
 			\sum_{j\in\N}\sum_{\bsnu\notin\Lambda_j^N} |\dup{u_\bsnu^s}{\tilde\eta_j}_\calY|^2
 			\lesssim
 			N^{-2\min(s,t)}
 		\end{equation*}
 		and
 		\begin{equation*}
 			\norm[W_{\bsmu,r,s}^{1,2}(\R^\N;\calY)]{u^s - \calD_N \circ \tilde g_N}^2
 			\lesssim
 			\sum_{j\in\N}\sum_{\bsnu\notin\Lambda_j^N}\Gamma_\bsnu |\dup{u_\bsnu^s}{\tilde\eta_j}|^2
 			\lesssim
 			N^{-2\min(r,s,t)}.
 		\end{equation*}
 		We can again realize $\tilde g_N$ as a sparse concatenation (now of
    $\sigma_q$ networks; see Definition \ref{def:qsparseconcatenation}) such
    that $\mpar(\tilde g_N) \leq C$ with a constant $C = C(q) > 0$. Let $a :=
    \sup\{\dup{u_\bsnu^s}{\tilde\eta_j}_\calY \;:\; \bsnu\in\calF, \, j\in\N\} <
    \infty$. Let $\tilde A := a^{-1} A$ with $A$ as in step 3 and let $\phi_a
    \colon \R^N \to \R^N$ such that $\phi_a(\bsx) = a\bsx$ for all
    $x\in\R^N$ and $\mpar(\phi_a) \leq C$ with $C$ only depending on $q$
    (cf.~Lemma \ref{lemma:RePUMultNN}). 
    We now have $\tilde g_N = \phi_a \bullet \tilde A \bullet \Phi$ 
    as the sparse concatenation of $\sigma_q$ networks. 
    Analogously to the argument in step 3 
    (but now  based on Definition \ref{def:qsparseconcatenation}) 
    we obtain
 		\begin{align*}
 			\size(\tilde g_N) &\lesssim  (\size(\phi_a) + \size(\tilde A) + \size(\Phi)) \lesssim N^{1+2\delta}\log(N), \\
 			\depth(\tilde g_N) &\lesssim  1 + \depth(\phi_a) + \depth(\tilde A) + \depth(\Phi) \lesssim N^{\delta} \log(N), \\
 			\mpar(\tilde g_N) &\lesssim \max\{\mpar(\phi_a),\mpar(\tilde A), \mpar(\Phi)\} \lesssim 1.
 		\end{align*}
 	This completes the proof.
	\end{proof}

	We have now all the elements needed to prove the main theorems.
	\begin{proof}[Proof of Theorems \ref{thm:main} and \ref{thm:main2}]
		Recall that $(\phi_j,\lambda_j{(s)})_{j\in\N}$ are the (orthonormal) eigenfunctions
		and eigenvalues of the covariance operator. Since $\calE$ is the encoder
		corresponding to this orthonormal basis, it holds for $a\sim\gamma_s$ that
		\begin{equation*}
			\calE(a)\sim \otimes_{j\in\N}\sqrt{\lambda_j{(s)}}\calN(0,1).
		\end{equation*}
		Hence we can write
		\begin{equation}
			\norm[W_{\gamma_s,r}^{1,2}(\calX,\calY)]{\calG(a)-\calD\circ g_N\circ\calE(a)}
			=
			\norm[W_{\bsmu,r,s}^{1,2}(\R^{\N},\calY)]{\calG(\sum_{j\in\N}\sqrt{\lambda_j(s)}y_j\phi_j)-
				\calD\circ g_N((y_j)_{j\in\N})}
		\end{equation}
		and
		\begin{equation}
			\norm[L_{\gamma_s}^2(\calX,\calY)]{\calG(a)-\calD\circ g_N\circ\calE(a)}
			=
			\norm[L_{\bsmu}^2(\R^{\N},\calY)]{\calG(\sum_{j\in\N}\sqrt{\lambda_j(s)}y_j\phi_j)-
				\calD\circ g_N((y_j)_{j\in\N})}
		\end{equation}
		Using Theorem~\ref{thm:NN-reapprox} concludes the proof.
	\end{proof}
	
	\section{Coefficient-to-solution map for a linear elliptic PDE}
	\label{sec:diffusion}
	In this section we consider as a domain the $d$-dimensional torus $\bbT^d = \bbR^d/\bbZ^d$.
	Let $\zeta_1 \leq \zeta_2\leq \dots$ and $(\xi_j)_{j\in\N}$be the eigenvalues and
	$L^2(\bbT^d)$ orthonormal eigenfunctions of the operator $-\Delta + \Id$, such that
	\begin{equation*}
		(-\Delta +\Id) \xi_j = \zeta_j \xi_j.
	\end{equation*}
	By Weyl's law (and explicit computation), we have
	\begin{equation}
		\label{eq:zeta-torus}
		\zeta_j \sim j^{2/d}, \qquad \forall j\in\bbN.
	\end{equation}
	For $p\in\R$, we define the Sobolev spaces
	\begin{equation}
		\label{eq:Hdef-torus}
		H^p = \{ v\in \calD(\bbT^d) :  \sum_{j\in\N}\zeta^p_j \langle v, \xi_j \rangle^2<\infty\},
	\end{equation}
	where $\langle \cdot, \cdot \rangle$ is the duality pairing that extends the
	$L^2(\bbT^d)$ inner product. 
	On the torus  
	these are the standard periodic Sobolev spaces \cite[Volume 1, Chapter 1, Remark 7.6]{LionsMagenes}.

	\paragraph{\textbf{Exact solution operator}}
	We consider the map
	\begin{equation*}
		\calG :
		\begin{cases}
			L^\infty(\bbT^d) \to \{v\in H^1: \int_{(0,1)^d} v = 0\}\\
			a\mapsto u
		\end{cases},
	\end{equation*}
	where $u$ is the (weak)
	solution to
	\begin{equation}
		\label{eq:diffusion-torus}
		-\nabla\cdot(e^a\nabla u) = f \text{ in }\bbT^d,
	\end{equation}
	for fixed, smooth $f$ with zero average in the unit cell.

  \paragraph{\textbf{Derivative}}
  For $\phi\in L^\infty(\bbT^d)$ with ${\rm essinf}_{x\in\bbT^d}\phi(x)>0$,
  let $L(\phi)$ denote the bounded, linear
  operator $H^1(\bbT^d)/\R\to H^{-1}(\bbT^d)/\R$ 
  such that
  \begin{equation*}
    L(\phi)u = -\nabla\cdot(\phi \nabla u).
  \end{equation*}
  Then $L(\phi)$ is an isomorphism due to the Lax-Milgram lemma. In particular,
  for small enough $h\in L^\infty(\bbT^d)$, also $L(\phi+h)=L(\phi)+L(h):H^1(\bbT^d)/\R\to H^{-1}(\bbT^d)/\R$ is an isomorphism, and
  \begin{equation*}
    (L(\phi)+L(h))^{-1} = \sum_{n\in\N_0} (-1)^n (L(\phi)^{-1}L(h))^n L(\phi)^{-1}
  \end{equation*}
  where we used the Neumann series representation of
  $(L(\phi)+L(h))^{-1} = (L(\phi) (I+L(\phi)^{-1}L(h)))^{-1}$. 
  Inspecting the first order term in $h$
  shows in particular that the first Fr\'echet derivative of the operator 
  $L(\phi)^{-1}:H^{-1}/\R\to H^1/\R$ w.r.t.\ $\phi$ evaluated in direction $h\in L^\infty(\bbT^d)$ 
  equals 
  \begin{equation*}
    D_\phi (L(\phi)^{-1})(h) = - L(\phi)^{-1}L(h) L(\phi)^{-1}:H^{-1}/\R\to H^1/\R.
  \end{equation*}

  Since $\calG(a)=L(e^a)^{-1}f$, and $D(e^a)(h) = he^a$, applying the chain rule we thus find
  \begin{equation}
    \label{eq:DG-example}
    D\calG(a) :
    \begin{cases}
      L^\infty(\bbT^d)\to H^1(\bbT^d)/\R\\
      h\mapsto -L(e^a)^{-1}L(h e^a)L(e^a)^{-1}f.
    \end{cases}
  \end{equation}

  \paragraph{\textbf{Spaces and measure}}
	We suppose in this section
  that we are given $s_0>d/2$ and $s_1>d/2$ such that
	\begin{equation}
    \label{eq:X-example}
		\calX = H^{s_0}, \qquad \calX^r = H^{s_0+s_1},
	\end{equation}
  and, for some $\beta>d/2$,  a centered Gaussian measure
  $\gamma$ on $\calX$ with covariance operator
	\begin{equation*}
		\calC = (-\Delta + \Id)^{-\beta}.
	\end{equation*}
	By the same argument as above, $\calC$ has eigenvalues $(\lambda_j)_{j\in\N}$
	such that
	\begin{equation*}
		\lambda_j \sim j^{-2\beta/d} \quad \mbox{as} \;\; j\to \infty\;.
	\end{equation*}
	It follows that $\bslambda \in \ell^{1/(2s+1)}(\N)$ for any
	\begin{equation}
		\label{eq:s-torus}
		s < \frac{\beta}{d} -\frac{1}{2} \;.
	\end{equation}
    Recall that the scalar product in $\calX = H^{s_0}$ is given by
  \begin{equation*}
    ( v, w )_{H^{s_0}} = \sum_{j\in\N_0} \zeta_j^{s_0} (v,\xi_j)_{L^2}(w,\xi_j)_{L^2}
  \end{equation*}
  and that therefore $(\zeta_j^{-s_0/2}\xi_j)_{j\in\N_0}$ is an orthonormal basis
  for $\calX$.
  By definition, 
  \begin{align*}
    \calX^r &= \left\{ v\in \calX: \sum_{j\in \N^0} \lambda_j^{-(2r+1)/(2s+1)}\zeta_j^{s_0}(v, \xi_j)^2_{L^2}<\infty\right\} \\
    &=
    \left\{v\in\calX : \sum_{j\in\N^0} \lambda_j^{-(2r+1)/(2s+1)-s_0/\beta}(v,\xi_j)_{L^2}^2 < \infty\right\}
  \end{align*}
  since $\lambda_j = \zeta_j^{-\beta}$ for all $j\in\N$. Therefore,
  \begin{equation*}
  \calX^r = H^{\beta(2r+1)/(2s+1) + s_0}.
  \end{equation*}
  It follows that for any $s_1>d/2$, there exist $s>0$ satisfying
  \eqref{eq:s-torus} and $r>0$ satisfying 
  \begin{equation*}
    r < \frac{s_1}{d} - \frac{1}{2}
  \end{equation*}
  so that
  \begin{equation*}
    \calX^r= H^{s_0+s_1}.
  \end{equation*}
	We now fix $\calY = H^1$ and choose, as a frame for $\calY$, 
	the orthornormal (w.r.t.~the $\calY$ scalar product) basis
	\begin{equation*}
		\tilde\eta_j = \eta_j = \zeta_j^{-1/2}\xi_j, \qquad j\in\N.
	\end{equation*}
	We infer from \cite{MSZ25_1135} 
	that Assumption \ref{ass:G} can be satisfied with, 
	for any $\delta_1>0$,
	\begin{equation}
		\label{eq:Yt=H-torus}
		\calY^t = H^{s_0-d/2+1-\delta_1}.
	\end{equation}
	Remark in addition that, by the definitions above,
	\begin{equation*}
		(v, \eta_j)_{\calY} \coloneqq \sum_{k\in\N}\zeta_k(v, \xi_k)_{L^2}(\eta_j, \xi_k)_{L^2} = \zeta_j^{1/2}(v, \xi_j)_{L^2}.
	\end{equation*}
  By definition \eqref{eq:Yt},
	\begin{equation}
		\label{eq:Yt-torus-L2-1}
		\calY^t = \{w\in H^1: \sum_{j\in \N}v_j^{-2t} \zeta_j (w, \xi_j)^{2}_{L^2}\}.
	\end{equation}
  For any $\delta_2>0$, it follows from \eqref{eq:zeta-torus} that $\bszeta^{-d/2-\delta_2}\in \ell^{1}(\N)$; choosing
	$\bsv = \bszeta^{-d/2-\delta_2}$ in \eqref{eq:Yt-torus-L2-1} gives
	\begin{equation}
		\label{eq:Yt-torus-L2-2}
		\calY^t = \{v\in H^1: \sum_{j\in \N}\zeta_j^{td + 1+ 2t\delta_2} (v, \xi_j)^{2}_{L^2}\}.
	\end{equation}
  \paragraph{\textbf{Approximation}}
	Comparing \eqref{eq:Hdef-torus}, \eqref{eq:Yt=H-torus}, and \eqref{eq:Yt-torus-L2-2} gives then that
	Assumption \ref{ass:G} is satisfied for any
	\begin{equation}
		\label{eq:t-torus}
		t < \frac{s_0}{d}-\frac{1}{2}.
	\end{equation}
	Theorems \ref{thm:main} and \ref{thm:main2} then imply that, for all
  \begin{equation*}
   q<\frac{\min\{s_0, s_1, \beta\}}{d}-1/2,
  \end{equation*}
  there exists $C>0$ such that for all $N\in\N$,
	\begin{equation*}
		\norm[L^2(H^{s_0},H^1;\gamma)]{\calG-\calD\circ\Phi_N\circ\calE} +
		\norm[L^2(H^{s_0},\mathrm{HS}(H^{s_0+s_1}, H^1);\gamma)]{D_{H^{s_0+s_1}}\left(\calG-\calD\circ\Phi_N\circ\calE  \right)} \leq C N^{-q},
	\end{equation*}
	where
	\begin{equation*}
		\calE : v\mapsto \big(\zeta_j^{-s_0/2}(v, \xi_j)_{H^{s_0}}\big)_{j\in\N}
		=\big(\zeta_j^{s_0/2}(v, \xi_j)_{L^2}\big)_{j\in\N},
	\end{equation*}
	and
	\begin{equation*}
		\calD : (x_j)_{j\in\N}\mapsto \sum_{j\in\N} x_j \eta_j
		= \sum_{j\in\N} \zeta_j^{-1/2}x_j  \xi_j.
	\end{equation*}
	
	\begin{remark}
		The discussion above can also be used to derive optimal convergence rates if,
		e.g., the measure on the input is given, but the encoder can be chosen. Suppose,
		for example, that $\tilde\gamma$ is a centered Gaussian measure on $L^2$, with
		covariance
		\begin{equation*}
			\tilde\calC = (-\Delta+\Id)^{-\tilde\beta},
		\end{equation*}
		with $\tilde\beta>d$. This implies that the measure is concentrated on
		$H^{\tilde\beta-d/2-\delta}$ for any $\delta>0$, see \cite[Examples 2.3.4 and
		2.3.6]{Bogachev1998}. The optimal encoder from the point of view of the rate of
		convergence proven in Theorem \ref{thm:main} is then 
                given by the choice $s_0 = \tilde\beta/2$, i.e.,
		\begin{equation*}
			\tilde\calX = H^{\tilde\beta/2}.
		\end{equation*}
		Note that the condition $\tilde\beta>d$ implies $s_0 < \tilde\beta-d/2$ (i.e.,
		$\tilde\gamma$ is concentrated on $H^{s_0} =\tilde\calX$). 
		It follows that 
                for any $q<\tilde\beta/(2d)-1/2$ there exists $C>0$ such that
		\begin{equation*}
			\norm[L^2(L^2,H^1;\tilde\gamma)]{\calG-\widetilde\calD\circ\widetilde\Phi_N\circ\widetilde\calE} = 
			\norm[L^2(\tilde\calX,\calY;\tilde\gamma)]{\calG-\widetilde\calD\circ\widetilde\Phi_N\circ\widetilde\calE}
			\leq C N^{-q}, \qquad \forall N\in\N.
		\end{equation*}
    Here $\widetilde\calE$ and $\widetilde\calD$ are, respectively, the encoder
    and decoder on orthonormal bases of $\widetilde\calX$ and $\calY = H^1$.
    We also have, upon possibly adjusting the constant $C$,
      \begin{equation*}
        \norm[L^2(\tilde\calX ,\mathrm{HS}(H^{\tilde\beta}, \calY);\tilde\gamma)]{D_{H^{\tilde\beta}}\left(\calG-\widetilde\calD\circ\widetilde\Phi_N\circ\widetilde\calE  \right)} \leq C N^{-q}.
      \end{equation*}
      Note that $D_{H^{\tilde\beta}}\calG$ is the Malliavin derivative of $\calG$,
      since $H^{\tilde\beta}$ is the Cameron-Martin space of $\tilde\gamma$.
     \end{remark}
     \section{Conclusion} \label{sec:Concl}
        
        We studied the expression rates of several families of operator surrogates
             of holomorphic maps between separable Hilbert spaces $\calX$ and $\calY$.
             As in the recent work \cite{HSZ24} of some of the authors,
             in an encode-approximate-decode framework,
             we considered two broad classes of finite-parametric surrogate operators: 
             spectral and neural surrogate operators.
             The spectral surrogates were obtained by suitable truncation of 
             Wiener-Hermite polynomial chaos expansions, and the neural
             operators
             were obtained by feedforward neural approximator maps. 

             Encoders are based on PCAs for covariance operators of Gaussian measures on the 
             input spaces.
             More general Decoder maps were allowed, in particular frames comprising 
             wavelet-bases or multi-level splittings of Finite Element spaces generated
             by BPX type algorithms on nested, simplicial complexes in polytopes 
             as typically arise in Finite Element discretizations of PDEs.

             In our main results, Theorems~\ref{thm:main} and \ref{thm:main2}, 
	we provided sufficient conditions
	for operator approximation rates \eqref{eq:OpEmRate},
	in the norm $L^2_\gamma(\calX,\calY)$ with respect to a GM $\gamma$
	as in Definition \ref{def:gamma} on the input space $\calX$,
	for parametric surrogate operators with $N$ neurons.
        We also obtained expression rate bounds in Gaussian Sobolev spaces.
        The present approximation rate bounds complement recent works 
         \cite{luo2025dimensionreductionderivativeinformedoperator,yao2025derivativeinformedfourierneuraloperator}
        and \cite{adcock2024learninglipschitzoperatorsrespect} 
        of mathematical analyses of operator learning under Gaussian measures, in that
        a periodic function space setting is not required, and orthogonality of 
        decoders is likewise not necessary.
        The (strong) assumption of holomorphy of the operators to be emulated by the 
        surrogates is, in view of lower bound results in \cite{SLanthPCA,adcock2024learninglipschitzoperatorsrespect}, 
        almost necessary to establish approximation rates that are not subject to the CoD,
        at least when considering function classes defined in terms of classical notions of regularity.
        We conclude by observing that we confined the present work to input/output pairs  that are
            (subset of) separable Hilbert spaces. This admittedly restrictive assumption was 
            adopted for simplicity of exposition; 
            we do expect similar results for (subsets of) suitable pairs of separable Banach spaces. 
            Arguments and technicalities will be more involved then due to the more delicate geometry
            of Banach spaces which may force additional technical assumptions. 
	\bibliographystyle{abbrv}
	\bibliography{library}
	
	\newpage
	
	\appendix
	\section{ReLU/RePU Neural Network Theory}
	\label{app:NNTheory}
	We now review basic operations and approximation results for ReLU neural networks (ReLU NN in short), 
	following \cite[Section 2.2]{OSZ21} and \cite[Section II and Appendix A]{elbrachter2021deep}. 
	These
	theoretical tools are well known but often stated without explicitly mentioning bounds
        on the weights and biases, 
        although some exceptions can be found in \cite{de2021approximation,elbrachter2021deep,schmidt2020supplement}. 
        We also provide the RePU versions of some constructions.
        The activation functions are denoted by $\sigma_q(x) \coloneq \max\{0,x\}^q$ with $q \in \N$.
	This includes the ReLU activation ($q = 1$) and the RePU activation ($q \geq 2$). 
	\begin{definition}\label{def:Parallelization}
	Let $\phi_1$ and $\phi_2$ be two neural networks with the same depth $L$ and activation $\sigma_q$ with $q \in \N$. 
	Furthermore, define the input dimensions of $\phi_1$ and $\phi_2$ as $m_1$ and $m_2$ 
        and the output dimensions as $n_1$ and $n_2$. 
        Then there exists a $\sigma_q$-network $(\phi_1,\phi_2)$, 
        called \emph{parallelization of $\phi_1$ and $\phi_2$},
        which simultaneously realizes $\phi_1$ and $\phi_2$, i.e., 
	\[
	(\phi_1,\phi_2)\colon \R^{m_1} \times \R^{m_2} \to \R^{n_1} \times \R^{n_2}; \quad (x, \widetilde{x}) \mapsto (\phi_1(x), \phi_2(\widetilde{x})).
	\]
	It holds 
	\[
	\size(\phi_1,\phi_2) = \size(\phi_1) + \size(\phi_2), \quad
	\depth(\phi_1,\phi_2) = L, \quad
	\mpar(\phi_1,\phi_2) = \max\{ \mpar(\phi_1), \, \mpar(\phi_2) \}.
	\]
	\end{definition}
	Now Definition \ref{def:Parallelization} 
	can be extended to parallelize $J \in \N,\, J\geq 3$ $\sigma_q$-networks $\phi_i, \, i = 1, \dots, J$ 
	assuming equal depth, given by $L$, for each $\phi_i$. 
	The resulting $\sigma_q$-network is denoted by $(\{\phi_i\}_{i=1}^J)$ and it holds that
	\[
	\size(\{\phi_i\}_{i=1}^J) = \sum_{i=1}^J \size(\phi_i), \quad 
	\depth(\{\phi_i\}_{i=1}^J) = L, \quad
	\mpar(\{\phi_i\}_{i=1}^J) = \max_{i=1,\dots,J} \mpar(\phi_i).
	\]
	
	Recall the concatenation of two neural networks \cite[Definition 2.2]{PtVgtl2018}.
	
	\begin{lemma}
		Let $\phi_1$ and $\phi_2$ be two $\sigma_q$-networks with $q \in \N$. 
		Furthermore, let the output dimension $n_2$ of $\phi_2$ be equal to the input dimension $m_1$ of $\phi_1$. Then there exists a $\sigma_q$-network $\phi_1 \circ \phi_2$ realizing the composition $\phi_1 \circ \phi_2 \colon x \mapsto \phi_1(\phi_2(x))$ of the functions $\phi_1$ and $\phi_2$. 
		
		It holds 
		\[
		\depth(\phi_1 \circ \phi_2) = \depth(\phi_1) + \depth(\phi_2), \quad 
		\width(\phi_1 \circ \phi_2) = \max\{ \width(\phi_1), \width(\phi_2) \}. 
		\]
	\end{lemma}
	
	In order to define \textit{sparse concatenations}, 
        we need to realize the identity mapping. 
        The proof of the following lemma can be found in \cite[Remark 2.4]{PtVgtl2018}.
		\begin{lemma}
      \label{lemma:Id-ReLU}
		Let $d \in \N$ and $L \in \N$. 
                Then there exists a ReLU NN $\widetilde{\mathrm{Id}}_{\R^d}$ of depth $L$ 
                that exactly realizes the identity mapping $\mathrm{Id}_{\R^d}$ on $\R^d$. 
		Furthermore, it holds that
			\[
			\size(\widetilde{\mathrm{Id}}_{\R^d}) \leq 2d (L+1), \quad 
			\mpar(\widetilde{\mathrm{Id}}_{\R^d}) \leq 1.
			\]
		\end{lemma}
		The following lemma follows from a construction in 
                \cite[Theorem 2.5 (2)]{betterapproximationsLiTangYu}.
		\begin{lemma}
			Let $d \in \N, \, q \in \N, \, q \geq 2,$ and $L \in \N$. 
                        Then there exists a $\sigma_q$-network $\widetilde{{\rm Id}}_{\R^d}$ of depth $L$
                        that exactly realizes the identity mapping ${\rm Id}_{\R^d}$. 
			Furthermore, 
                        there exists a constant $C = C(q) > 0$ such that 
			\[
			\size(\widetilde{{\rm Id}}_{\R^d}) \leq C d L,
			\quad
			\mpar(\widetilde{{\rm Id}}_{\R^d}) \leq C.
			\]
		\end{lemma}
		Now, \textit{sparse concatenation} can be defined (see \cite[Definition 2.5]{PtVgtl2018}).
		\begin{definition}\label{def:sparseConc}
			Let $\phi_1$ and $\phi_2$ be two ReLU NNs. 
			Furthermore, let the output dimension $n_2$ of $\phi_2$ equal the input dimension $m_1$ of $\phi_1$. 
			Then there exists a ReLU NN $\phi_1 \bullet \phi_2$
			realizing the composition $\phi_1 \circ \phi_2 : x \mapsto \phi_1(\phi_2(x))$ of the functions $\phi_1$ and $\phi_2$. 
			It holds
			\begin{align*}
			\size(\phi_1 \bullet \phi_2) &\leq 2 (\size(\phi_1) + \size(\phi_2)), \\
			\depth(\phi_1 \bullet \phi_2) &= \depth(\phi_1) + \depth(\phi_2) + 1,\\
			\mpar(\phi_1 \bullet \phi_2) &\leq \max\{ \mpar(\phi_1), \mpar(\phi_2) \}.
			\end{align*}
		\end{definition}
		Next, we consider the sparse concatenation of $\sigma_q$-networks 
		in which we replace the ReLU realization of the identity mapping by its RePU realization.
		
		\begin{definition}\label{def:qsparseconcatenation}
			Let $q \in \N, \, q \geq 2$ and $\phi_1, \, \phi_2$ be two $\sigma_q$-networks. 
			Furthermore, let the output dimension $n_2$ of $\phi_2$ equal 
			the input dimension $m_1$ of $\phi_1$. 
			
			Then, there exists a $\sigma_q$-network $\phi_1 \bullet \phi_2$
		realizing the composition $\phi_1 \circ \phi_2 \colon x \mapsto \phi_1(\phi_2(x))$ 
                of the functions $\phi_1$ and $\phi_2$. 
                There exists a constant $C = C(q) > 0$ such that
			\begin{align*}
				\size(\phi_1 \bullet \phi_2) &\leq C  (\size(\phi_1) + \size(\phi_2)), \\
				\depth(\phi_1 \bullet \phi_2) &\leq \depth(\phi_1) + \depth(\phi_2) + 1, \\
				\mpar(\phi_1 \bullet \phi_2) &\leq C  \max\{\mpar(\phi_1), \mpar(\phi_2)\}.
			\end{align*}
		\end{definition}

	The next two results enable us to normalize the weights of any ReLU NN 
and can be found in \cite[Appendix A]{elbrachter2021deep}. 
In \cite{elbrachter2021deep} the size bounds are not explicitly mentioned but are added here for convenience.
\begin{lemma}[{\cite[Lemma A.1]{elbrachter2021deep}}]\label{lemma:scalarMultNetwork}
	Let $d \in \N$ and $a \in \R$. 
	There exists a ReLU NN $\phi_a \colon \R^d \to \R^d$ satisfying 
	\[ 
	\phi_a(x) = ax 
	\]
	for all $x \in \R^d$. 
	Furthermore, there exists a constant $C > 0$ independent of $d$ and $a$ such that
	\[
	\size(\phi_a) \leq Cd \, (\log(\max\{1, |a|\}) + 1), \quad
	\depth(\phi_a) \leq C \, (\log(\max\{1, |a|\}) + 1), \quad
	\mpar(\phi_a) \leq 1.
	\]
\end{lemma}

\begin{lemma}[{\cite[Proposition A.3]{elbrachter2021deep}}]\label{lemma:normalizeWeights}
	Let $d, d' \in \N$ and $\phi \colon \R^d \to \R^{d'}$ a ReLU NN with $\mpar(\phi) > 1$. 
	There exists a ReLU NN $\psi \colon \R^d \to \R^{d'}$ satisfying 
	\[
	\psi(x) = \phi(x)
	\]
	for all $x \in \R^d$. 
	In addition, 
	there exists a constant $C > 0$ independent of $d, \, d'$, and $\phi$ such that
	\begin{equation*}
		\begin{aligned}
			\size(\psi) &\leq C \, ( \size(\phi) + d' ( (\depth(\phi)+1)\log(\mpar(\phi)) +1 ) ), \\
			\depth(\psi) &\leq C \, (\log(\mpar(\phi))+1)\,(\depth(\phi)+1), \\
			\mpar(\psi) &\leq 1.
		\end{aligned}
	\end{equation*}
\end{lemma}
Similar results hold for the normalization of network weights in RePU-networks. 
\begin{lemma}[{\cite[Definition C.8]{reinhardt2024statisticallearningtheoryneural}}]\label{lemma:RePUMultNN}
	Let $d \in \N, \, q \in \N, \, q \geq 2,$ and $a \in \R$. 
	There exists a $\sigma_q$-network $\phi_a \colon \R^d \to \R^d$ satisfying 
	\[
	\phi_a(\bsx) = a \bsx
	\]
	for all $\bsx \in \R^d$. Furthermore, there exists a constant $C = C(q) > 0$, not depending on $d$ and $a$, such that 
	\begin{align*}
		\size(\phi_a) &\leq C d \log(\max\{1, |a|\}), \\
		\depth(\phi_a) &\leq C \log(\max\{1, |a|\}), \\
		\mpar(\phi_a) &\leq C.
	\end{align*}
\end{lemma}
	
Next, we build towards $W^{1, \infty}$ approximation error bounds.
In \cite{Yarotsky2017}, it is shown that deep ReLU NNs allow for an efficient approximation of the square function $x \mapsto x^2$ in the sense that the approximation error decays exponentially w.r.t.~the depth of the network. This leads to an efficient approximation of the product operation with two inputs by considering the identity
\[
a b = 2 D^2 \left(\frac{|a+b|}{2D}\right)^2 - \left(\frac{|a|}{2D}\right)^2 - \left(\frac{|b|}{2D}\right)^2
\]
for some $D \geq 1$ and $|a|, \, |b| \leq D$ and by replacing the square terms by the NN approximation. 
Based on this, efficient approximations of polynomials are possible. 
\par

In \cite{Yarotsky2017} the approximation of $x \mapsto x^2$ in $[0,1]$ is achieved by considering the continuous, piecewise linear spline interpolation $f_n$ of $x^2$ at the uniformly spaced nodes $j 2^{-n}$ for $j = 0, \dots , 2^n$. The functions $f_n$ are exactly expressed by a ReLU NN via 
\begin{align*}
	f_n(x) \coloneqq x - \sum_{k = 1}^n \frac{g_k(x)}{2^{2k}} 
	= f_{n-1}(x) - \frac{g_n(x)}{2^{2n}} 
	= \sigma_1 \left( f_{n-1}(x) - \frac{g_n(x)}{2^{2n}} \right)
\end{align*}
for all $n \geq 1$ and $f_0(x) \coloneqq x = \sigma_1(x)$ for $x \in [0,1]$. 
The functions $g_n \coloneqq g \, \circ \, \dots \, \circ \, g$ (also known as "sawtooth functions") 
are defined through the $n$-fold composition of $g$ which is given by 
\begin{equation*}
	g(x) = \sigma_1(2x) - \sigma_1(4x-2) 
	= 
	\begin{cases}
		2x \quad &\text{if } x \in [0,\frac{1}{2}] \\
		2 - 2x \quad &\text{if } x \in [\frac{1}{2},1].
	\end{cases}
\end{equation*}
An efficient approximation to the product operation with two inputs can be obtained through
\begin{equation}\label{def:2MultNN}
	\widetilde{\times}_{D,\varepsilon}(a,b) \coloneqq 2 D^2 
	\left( f_n\left(\frac{|a+b|}{2D}\right) - f_n\left(\frac{|a|}{2D}\right) - f_n\left(\frac{|b|}{2D}\right) \right).
\end{equation}
Note that the absolute value can be expressed through the ReLU activation function $\sigma_1(x) = \max\{0,x\}$ as $|x| = \sigma_1(x) + \sigma_1(-x)$. 
Here, $\varepsilon > 0$ denotes the desired approximation accuracy. This allows for an efficient approximation of the multiplication of two numbers in $I \coloneqq [-D,D]$ with accuracy $\varepsilon > 0$ by a ReLU NN $\widetilde{\times}_{D,\varepsilon} \colon I \times I \to \R$ of size and depth bounded by $C \, ( \log(\max\{1,D\}/\varepsilon) + 1)$.The following proposition is a restatement of \cite[Proposition 3.1]{SchwabZechAA} with network weights bounded by one.
\begin{proposition}\label{prop:H1Mult}
	Let $D \geq 1$ and $\varepsilon \in (0,{1/2})$. 
        There exists a ReLU NN $\widetilde{\times}_{D,\varepsilon}$ 
        with two-dimensional input and one-dimensional output such that 
	\begin{equation*}
		\begin{aligned}
			&\underset{|a|, \, |b| \leq D}{\sup} \left| ab - \widetilde{\times}_{D,\varepsilon}(a,b) \right| \leq \varepsilon, \\
			&\underset{|a|, \, |b| \leq D}{\esssup} \; \max \left\{ \left|a - \frac{\mathrm{d}}{\mathrm{d}b}\widetilde{\times}_{D,\varepsilon}(a,b)\right|, \left| b - \frac{\mathrm{d}}{\mathrm{d}a} \widetilde{\times}_{D, \varepsilon}(a,b) \right| \right\} \leq \varepsilon.
		\end{aligned}
	\end{equation*}
  Furthermore, there exists a constant $C > 0$ independent of $D$ and $\varepsilon$ such that 
	\begin{equation*}
		\begin{aligned}
			\size(\widetilde{\times}_{D,\varepsilon}) &\leq C \, \left( \log\left(\frac{D}{\varepsilon}\right) + 1 \right), \\
			\depth({\widetilde{\times}_{D,\varepsilon}}) &\leq C \, \left( \log\left(\frac{D}{\varepsilon}\right) + 1 \right), \\
			\mpar(\widetilde{\times}_{D,\varepsilon}) &\leq 1.
		\end{aligned}
	\end{equation*}
\end{proposition}
\begin{proof}
	First, we define a ReLU NN via 
	\[
		\psi_{D,\eps}(a,b) \coloneqq f_n\left(\frac{|a+b|}{2D}\right) - f_n\left(\frac{|a|}{2D}\right) - f_n\left(\frac{|b|}{2D}\right)
	\]
	for all $a,b \in [-D,D]$ and for $n \coloneqq \ceil{-\log(\eps/(2D))}$. 
        The proof of \cite[Proposition 3.1]{SchwabZechAA} shows 
	\begin{align*}
		\size(\psi_{D,\eps}) \leq \log(1/\eps) \;, 
                 \quad 
		\depth(\psi_{D,\eps}) \leq \log(1/\eps) \quad\text{ as } \eps \to 0.
	\end{align*}
	Note that \cite[Proposition 3.1]{SchwabZechAA} considers the additional factor $2D^2$ in the output, 
        i.e. $2D^2 \, \psi_{D,\eps}(a,b)$. 
        This however does not change the bounds on size and depth. 
        A direct calculation shows that the weights in the network expressing $f_n$ 
        can be bounded by a constant independent of $n \in \N$. 
        This implies 
	\[
		\mpar(\psi_{D,\eps}) \leq C
	\]
	for some $C > 0$ independent of $n \in \N$.
	With Lemma \ref{lemma:scalarMultNetwork} 
        we define the multiplication network $\phi_{2D^2} \colon \R \to \R$ 
        realizing the multiplication with $2D^2$, i.e. $\phi_{2D^2}(x) = 2D^2 x$ for all $x \in \R$. 
Furthermore, it holds, with a constant $C > 0$ independent of $D$, that
	\begin{align*}
		\size(\phi_{2D^2}) \leq C (\log(D) + 1), \;\;
		\depth(\phi_{2D^2}) \leq C (\log(D) + 1), \;\;
		\mpar(\phi_{2D^2}) \leq 1.
	\end{align*}
        For all $a,b \in [-D,D]$, we define the sparse concatenation (cf. Definition \ref{def:sparseConc}) by
	\[
		\widetilde{\times}_{D,\eps}(a,b) \coloneqq (\phi_{2D^2} \bullet \psi_{D,\eps}) (a,b)
	\] 
        Definition \ref{def:sparseConc} implies 
	\begin{align*}
		\size (\widetilde{\times}_{D,\eps}) &\leq C \left(\log\left(\frac{D}{\eps}\right)+1\right), \\
		\depth (\widetilde{\times}_{D,\eps}) &\leq C \left(\log\left(\frac{D}{\eps}\right)+1\right), \\
		\mpar(\widetilde{\times}_{D,\eps}) &\leq C. 
	\end{align*}
	Lemma \ref{lemma:normalizeWeights} allows for the normalization of the networks weights 
        and only changes the above bounds in the constant factor.
	Let $\hat{\times}_{D,\eps}$ denote the network from \cite[Propostion 3.1]{SchwabZechAA}. 
        We have now constructed a ReLU NN such that 
	\[
		\widetilde{\times}_{D,\eps}(a,b) = \hat{\times}_{D,\eps}(a,b)
	\]
	for all $a,b \in [-D,D]$. 
        This shows the remaining claimed uniform bounds 
        since they already hold for $\hat{\times}_{D,\eps}$ by \cite[Proposition 3.1]{SchwabZechAA}.
\end{proof}
The next proposition is based on \cite[Proposition III.5]{elbrachter2021deep} 
in which an analogous polynomial approximation result in $L^\infty$-norm is established. 
Our version here can be shown in exactly the same way by replacing 
all occurrences of the $L^\infty$-norm by the $W^{1,\infty}$-norm and using 
Proposition \ref{prop:H1Mult}.
\begin{proposition}\label{prop: W1inftyPolynomApprox}
Let $n \in \N$ and $c = (c_i)_{i=0}^n \in \R^{n+1}$. 
In addition, let $M \geq 1$ and $\varepsilon \in (0,{1/2})$. 
Define $c_{\infty} \coloneqq \max\{1, \|c\|_{\infty}\}$. 

Then, for any polynomial $p(x) \coloneqq \sum_{i=0}^n c_i x^i$ 
there exists a ReLU NN $\tilde{p}_{\varepsilon, M}$ satisfying 
	\[
	\|p - \tilde{p}_{\varepsilon, M}\|_{W^{1,\infty}([-M,M])} \leq \varepsilon.
	\]
	Furthermore, there exists a constant $C > 0$ independent of $n,\, c_i,\, M,$ and $\varepsilon$ such that 
	\begin{align*}
		\depth(\tilde{p}_{\varepsilon, M}) &\leq C n \, (n \log(M) + \log(\varepsilon^{-1}) + \log(n) + \log(c_{\infty})), \\
		\size({\tilde{p}_{\varepsilon, M}}) &\leq C n \, (n \log(M) + \log(\varepsilon^{-1}) + \log(n) + \log(c_{\infty})), \\
		\mpar({\tilde{p}_{\varepsilon, M}}) &\leq 1.
	\end{align*}
\end{proposition}
We now extend Proposition \ref{prop:H1Mult} to the multiplication of $J$ numbers. 
The proof of the following result is similar to that of \cite[Proposition 2.6]{OSZ21}.

\begin{proposition}\label{prop:H1MultWithNNumbers}
Let $J, D \in \N, \, J \geq 2$, $D \geq 1$ and $\varepsilon \in (0, {1/2})$. 

Then there exists a ReLU NN $\widetilde{\prod}_{J,D,\varepsilon} \colon [-D,D]^J \to \R$ such that 
	\[
	\underset{(y_i)_{i=1}^J \in [-D,D]^J}{\sup} \left|\prod_{i=1}^J y_i - \widetilde{\prod}_{J,D,\varepsilon}(y_1,\dots,y_J)\right| \leq \varepsilon
	\]
	and 
	\[
	\underset{(y_i)_{i=1}^J \in [-D,D]^J}{\esssup} \, \underset{i = 1, \dots,
    J}{\max} \left\{\left|\prod_{\substack{j=1 \\ j \neq i}}^Jy_j - \frac{\partial}{\partial y_i}\widetilde{\prod}_{J,D,\varepsilon}(y_1,\dots,y_J)\right|\right\} \leq \varepsilon,
	\]
	where $\partial/\partial y_i$ denote the weak derivatives. 
        Furthermore, there exists a constant $C > 0$ independent of $J, D$ and $\varepsilon$ such that 
	\begin{align*}
		\size\left(\widetilde{\prod}_{J,D,\varepsilon}\right) & \leq C J\, (\log(J) + J\log(D) + \log(\varepsilon^{-1})) \leq C \, \left(1 + J \log\left(\frac{JD^J}{\varepsilon}\right)\right), \\
		\depth\left(\widetilde{\prod}_{J,D,\varepsilon}\right) &\leq C \log(J) \, (\log(J) + J\log(D) + \log(\varepsilon^{-1})) \leq C \, \left(1 + \log(J) \log\left(\frac{JD^J}{\varepsilon}\right)\right), \\
		\mpar\left(\widetilde{\prod}_{J,D,\varepsilon}\right) &\leq 1.
	\end{align*}
\end{proposition}
	
	We follow now with exact realizations of products and polynomials in the RePU-case. 
A proof of the following result can, for example, 
be found in \cite[Appendix C]{reinhardt2024statisticallearningtheoryneural}.
	\begin{proposition}\label{prop:RePUMultOp}
	Let $J, q \in \N$ with $J, q \geq 2$. 
         Then there exists a $\sigma_q$-network $\widetilde{\prod}_J \colon \R^J \to \R$ such that 
		\[
		\widetilde{\prod}_J(y_1,\dots,y_J) = \prod_{j=1}^{J}y_j.
		\]
		Furthermore, there exists a constant $C = C(q) > 0$ independent of $J$ such that 
		\[
		\depth\left(\widetilde{\prod}_J\right) \leq C \log(J), \;
		\size\left(\widetilde{\prod}_J\right) \leq C J, \;
		\mpar\left(\widetilde{\prod}_J\right) \leq C. 
		\]
	\end{proposition}

	In the RePU-case, 
	polynomials can be realized without error (see \cite{PowerNetLiTangYu2019}).
	\begin{proposition}\label{prop:RePUPolynomialEmulation}
		Let $n \in \N, \, c = (c_i)_{i=0}^n \in \R^{n+1},$ and $q \in \N, \, q \geq 2$. 
		Define $c_\infty \coloneq \max\{1, \|c\|_{\infty}\}$. 
		
		Then, for any polynomial $p(x) \coloneq \sum_{i=0}^n c_i x^i$, 
		there exists a $\sigma_q$-network $\tilde p \colon \R \to \R$ satisfying 
		\[
		\tilde p(x) = p(x) 
		\]
		for all $x \in \R$. 
		Furthermore, there exists a constant $C = C(q) > 0$ independent of $n$ such that 
		\[
		\depth(\tilde p) \leq C (\log(c_\infty)+n), \quad 
		\size(\tilde p) \leq C (\log(c_\infty)+n), \quad
		\mpar(\tilde p) \leq C.
		\]
	\end{proposition}

		\section{DNN Emulation of Hermite Polynomials with Weight Bounds}
		\label{app:DNNEmulationHermite}
		In this section, we follow a similar line of reasoning as in \cite[Section 3]{SZ21_982} 
		in order to derive expression rates for NN approximation of Hermite polynomials, additionally
    bounding the weights and biases of the NNs.
	\subsection{Univariate Hermite Polynomials}
	The following two lemmas yield a construction of the restriction of ReLU NNs
  to compact supports, with quantitative control on the Lipschitz constant.
	\begin{lemma}\label{lemma:upperBoundNN}
		Let $R, M, a, b > 0$. 
                Then there exists a ReLU NN $q \colon \R \to \R$ 
                given by 
		\[
		q(x) \coloneqq \sigma_1(R - \sigma_1(a(x-M))-\sigma_1(-b(x+M)))
		\]
		such that 
		\begin{equation*}
			q(x) = 
			\begin{cases}
				R, \qquad  &x \in [-M,M] \\
				-a(x-M)+R, &x \in [M,\frac{R}{a}+M] \\
				b(x+M)+R, &x \in [-\frac{R}{b}-M, -M] \\
				0, &x \in \R\backslash[-\frac{R}{b}-M,\frac{R}{a}+M]
			\end{cases}.
		\end{equation*}
		Furthermore, it holds that 
		\begin{align*}
			\size(q) &= 8, \\
			\depth(q) &= 2, \\
			\width(q) &= 2, \\
			\mpar(q) &= \max\{1, R, a, b, aM, bM\}.
		\end{align*}
	\end{lemma}
	\begin{proof}
		This follows from direct calculations.
	\end{proof}
	\begin{lemma}\label{lemma:truncationNN}
		Let $\phi \colon \R \to \R$ be a ReLU NN and 
                let $q\colon\R\to\R$ be the ReLU NN from Lemma \ref{lemma:upperBoundNN} 
                with $M>0$ fixed, $R \coloneqq \sup_{x\in[-M,M]}|\phi(x)|$ and $a=b=R$. Define 
                the ReLU NN 
		\begin{equation}\label{eq:maxminqNN}
			\psi(x) \coloneqq \max\{\min\{q(x), \phi(x)\},-q(x)\}.
		\end{equation}
		Then, it holds that 
		\begin{enumerate}[label=(\roman*)]
			\item $\psi(x) = \phi(x)$ for all $x \in [-M,M]$
			\item $\psi(x) = 0$ for all $x \in \R\backslash[-M-1, M+1]$
			\item $\sup_{x \in \R} |\psi(x)| \leq R$
			\item ${\rm Lip}(\psi) \leq \max\{R, {\rm Lip}(\phi_{|[-M-1,M+1]})\}$ with ${\rm Lip}(\cdot)$ denoting the Lipschitz constant.
		\end{enumerate}
		Furthermore, there exists a constant $C > 0$, independent of $\phi$ and $M$, such that
		\begin{align*}
			\size(\psi) &\leq C \, (1 + \size(\phi)), \\
			\depth(\psi) &\leq C \, (1 + \depth(\phi)), \\
			\mpar(\psi) &\leq \max\{1, R, RM, \mpar(\phi)\}. 
		\end{align*}
	\end{lemma}
	\begin{proof}
		The items $(i), (ii)$ and $(iii)$ follow directly from the definition of
    $\psi$. Since $\psi$ is continuous and piecewise linear, item $(iv)$ holds
    true. In addition, $\psi$ can be represented as a ReLU NN because $\max\{\cdot,\cdot\}$ and $\min\{\cdot,\cdot\}$ can be realized as ReLU NN, which implies that $\psi$ can be constructed as the parallelization and (sparse) concatenation of ReLU NN $\max\{\cdot,\cdot\}, \, \min\{\cdot,\cdot\},\,q$ and $\phi$. More precisely, both $\max$ and $\min$ can be constructed as ReLU NN of size 7, depth 1, width 3 and weights and biases bounded by 1, since for all $x,y\in\R$
		\[
		\max\{x,y\} = \sigma_1(y) - \sigma_1(-y) + \sigma_1(x-y) \;, \quad 
		\min\{x,y\} = \sigma_1(y) - \sigma_1(-y) - \sigma_1(y-x) \;.
		\]
		This implies the claimed bounds on 
    the size, depth, and weights of $\psi$,
    obtained using Definition \ref{def:Parallelization} and Definition \ref{def:sparseConc}.
	\end{proof}
	We summarize some basic results on Hermite polynomials for later use in the proof of corresponding expression rates.
	Some of these results are taken from \cite[Section 2.2.]{SZ21_982}. \par
	
	We use the following representation of Hermite polynomials; 
	see, e.g., \cite[equation (5.5.4)]{Szego3rdOrthPol}:
	\begin{equation}\label{eq:HermitExplicit}
		H_n(x) = \sum_{j = 0}^{\lfloor n/2 \rfloor} \frac{\sqrt{n !}(-1)^j}{j!(n-2j)! 2^j} x^{n-2j}.
	\end{equation}
	The factor $\sqrt{n!}$ is due to a different scaling 
         (cf. \cite[equation (5.5.3)]{Szego3rdOrthPol} with \eqref{eq:Hermite}). 
	We then introduce the coefficients $c_{n, j}$ such that
  \begin{equation}
    \label{eq:cnj-def}
  H_n(x) = \sum_{j=0}^n c_{n,j} x^n. 
  \end{equation}
	The following lemma yields a bound on the $\ell^1$-norm of the vector of the
  coefficients of $H_n$ in the monomial basis.
	\begin{lemma}[{\cite[Lemma 2.3]{SZ21_982}}]\label{lemma:CoeffSummability}
		For all $n \in \N_0$,
		\begin{equation}
			\sum_{j = 0}^n |c_{n,j}| \leq 6^{n/2} \leq 3^n.
		\end{equation}
	\end{lemma}
	It follows from Lemma \ref{lemma:CoeffSummability} that 
	\begin{equation}\label{eq:HermitePointwiseBound}
		|H_n(x)| \leq (3 \, \max\{ 1, |x| \})^n \qquad \forall \, x \in \R.
	\end{equation}
	For $n \in \N_0 \cup \{-1\}$, let $n !!$ denote the double factorial, such that $-1 !! = 0 !! = 1$ and $n !! = n \cdot (n-2)!!$ for $n \geq 2$.
	\begin{lemma}[{\cite[Lemma 2.4]{SZ21_982}}]\label{lemma:IntegralBoundHermite}
		Let $M \geq 2$ and $n \in \N_0$. Then 
		\begin{equation}
			\int_{|x| > M} e^{-\frac{x^2}{2}} |x|^n \, {\rm d}x \leq n!! M^n e^{-\frac{M^2}{2}}.
		\end{equation}
	\end{lemma}
	 Lemma \ref{lemma:IntegralBoundHermite} and \eqref{eq:HermitePointwiseBound}
   imply that, for every $M \geq 2$, $p \geq 1$, and $n \in \N_0$,
	\begin{equation}\label{eq:LpBoundComplement}
		\int_{|x| > M} |H_n(x)|^p \, {\rm d}\mu(x) = \frac{1}{\sqrt{2\pi}} \int_{|x| > M} |H_n(x)|^p e^{-\frac{x^2}{2}} \, {\rm d}x \leq \frac{1}{\sqrt{2\pi}} (pn)!! (3M)^{pn} e^{-\frac{M^2}{2}}.
	\end{equation}
	This result can be found in \cite[Section 2.2. Equation (2.11)]{SZ21_982}.
	
	We finally recall the following relation for the derivative of Hermite polynomials (see, e.g., \cite[Lemma 2.4 $(ii)$]{DNSZ23_2957}):
	\begin{equation}\label{eq:recursiveDerivativeHermite}
		H_n'(x) = \sqrt{n} \, H_{n-1}(x) \qquad \text{for all } n \in \N.
	\end{equation}
	
		\begin{proposition}\label{prop:ExpressionRateUnivariateHermite}
			Let $n \in \N_0, \, M \geq 2$, and $\varepsilon \in (0, {1/2})$. 
                Then there exists a ReLU NN $\widetilde{H}_{n, M, \varepsilon} \colon \R \to \R$ such that 
			\begin{enumerate}[label=(\roman*)]
				\item $\|H_n - \widetilde{H}_{n,M,\varepsilon}\|_{{H^1(\R;\mu)}} 
                                 \leq \varepsilon + {6}\sqrt{{2n}(2n)!!} (3{(M+1)})^n e^{-\frac{M^2}{4}}$;
				\item $\sup_{x\in\R} |\widetilde{H}_{n,M,\varepsilon}(x)| \leq 1 + (3M)^n$;
				\item for a constant $C > 0$ independent of $n, M, \varepsilon$
				\begin{align*}
					\size(\widetilde{H}_{n,M,\varepsilon}) 
					&\leq C \, (1 + n^2\log(M) + n\log(\varepsilon^{-1}))\\
					\depth(\widetilde{H}_{n,M,\varepsilon}) 
					&\leq C \, (1 + n^2\log(M) + n\log(\varepsilon^{-1})),\\
					\mpar(\widetilde{H}_{n, M, \varepsilon}) 
					&\leq {(1+3^n)M^{n+1}} .
				\end{align*}
			\end{enumerate}
		\end{proposition}

\begin{proof}
	Without loss of generality assume $n \geq 1, \, n \in \N$.
	By \eqref{eq:HermitExplicit} and Proposition \ref{prop: W1inftyPolynomApprox} there exists a ReLU NN $\hat{H}_{n,M,\varepsilon}$ such that 
	\begin{equation}\label{eq: worst case H1 error on compact interval}
		\|H_n - \hat{H}_{n,M,\varepsilon}\|_{W^{1,\infty}([-M-1,M+1])} \leq \frac{\varepsilon}{\sqrt{2}}.
	\end{equation}
	Let $c = (c_{n, j})_{j=0}^n \in \R^{n+1}$ be the vector containing the coefficients of $H_n$ in the
  monomial basis (see \eqref{eq:cnj-def}) and define $c_{\infty} \coloneqq \max\{1, \|c\|_{\infty}\}$. Due to Lemma \ref{lemma:CoeffSummability} it holds that $\|c\|_{\infty} \leq \|c\|_{\ell^1} \leq 3^n$. This implies the size and depth bounds
	\begin{align*}
		\size(\hat{H}_{n,M,\varepsilon}) 
		&\leq C  \, (n^2 \log(M) + n\log(\varepsilon^{-1}) + n\log(n) + n^2), \\
		\depth(\hat{H}_{n,M,\varepsilon}) 
		&\leq C  \, (n^2 \log(M) + n\log(\varepsilon^{-1}) + n\log(n) + n^2),
	\end{align*}
	based on Proposition \ref{prop: W1inftyPolynomApprox}.
	Remark that $n\log(n) + n^2\leq 2 n^2\log(M)$ since $\log(M) \geq 1$ due to $M \geq 2$.
	We additionally have
	\begin{equation}
		\mpar(\hat{H}_{n.M,\varepsilon}) \leq 1
	\end{equation}
	from Proposition \ref{prop: W1inftyPolynomApprox}.
	We now restrict the network $\hat{H}_{n,M,\varepsilon}$ to the interval
  $[-M-1, M+1]$ using Lemma \ref{lemma:truncationNN} and denote $R \coloneqq
  \sup_{x\in[-M,M]}|\hat{H}_{n,M,\varepsilon}(x)|$. From this, we define the
  ReLU NN $\widetilde{H}_{n,M,\varepsilon} \colon \R \to \R$ such that
	\begin{equation}\label{eq:defTildeH}
		\widetilde{H}_{n,M,\varepsilon}(x) =
		\begin{cases}
			\hat{H}_{n,M,\varepsilon}(x), \qquad &x \in [-M, M], \\
			0, &x \in \R\backslash[-M-1,M+1]
		\end{cases}
	\end{equation}
	with $\sup_{x\in\R}|\widetilde{H}_{n,M,\varepsilon}(x)| \leq R$ and ${\rm Lip}(\widetilde{H}_{n,M,\varepsilon}) \leq \max\{R, {\rm Lip(\hat{H}_{n,M,\varepsilon})}\}$. From this and \eqref{eq:HermitePointwiseBound} we obtain 
	\begin{align}\label{eq: H tilde uniform bound}
		\underset{x \in \R}{\sup} |\widetilde{H}_{n,M,\varepsilon}(x)| 
		&\leq \underset{x \in [-M,M]}{\sup} |\hat{H}_{n,M,\varepsilon}(x)| \leq \frac{\varepsilon}{\sqrt{2}} + \underset{x \in [-M,M]}{\sup} |H_n(x)| \leq 1 + (3M)^n,
	\end{align}
	which shows item $(ii)$.
	By Lemma \ref{lemma:truncationNN}, we also obtain
	\begin{align*}
		\size(\widetilde{H}_{n,M,\varepsilon}) &\leq C \, \left(1 + \size(\hat{H}_{n,M,\varepsilon})\right) 
		\leq C \, (1 + n^2\log(M) + n\log(\varepsilon^{-1})),\\
		\depth(\widetilde{H}_{n,M,\varepsilon}) &\leq C \, \left(1 + \depth(\hat{H}_{n,M,\varepsilon})\right) 
		\leq C \, (1 + n^2\log(M) + n\log(\varepsilon^{-1})), \\
		\mpar(\widetilde{H}_{n,M,\varepsilon}) &\leq \max\{1, RM\} \leq (1 + 3^n)M^{n+1},
	\end{align*}
	since $R \leq 1 + (3M)^n$.
	Next, we consider the approximation error
	\begin{equation}\label{eq:approxErrorSplitting}
		\begin{aligned}
			\|H_n - \widetilde{H}_{n,M,\varepsilon}\|_{H^1(\R;\, \mu)} 
			&\leq 
			\|H_n - \widetilde{H}_{n,M,\varepsilon}\|_{H^1([-M,M];\, \mu)} \\
			&+ 
			\|H_n - \widetilde{H}_{n,M,\varepsilon}\|_{H^1([-M-1,M+1]\backslash[-M,M];\, \mu)} \\
			&+ 
			\|H_n - \widetilde{H}_{n,M,\varepsilon}\|_{H^1(\R\backslash[-M-1,M+1];\, \mu)}.
		\end{aligned}
	\end{equation}
Note that here derivatives of functions realized by ReLU NN are to be understood in the sense of weak derivatives. 

For $x \in [-M,M]$ it holds that $\widetilde{H}_{n.M,\varepsilon}(x) = \hat{H}_{n,M,\varepsilon}(x)$ 
by construction of $\widetilde{H}_{n,M,\varepsilon}$. 
Then, \eqref{eq: worst case H1 error on compact interval} implies  
	\begin{equation*}
		\begin{aligned}
			\|H_n - \widetilde{H}_{n,M,\varepsilon}\|_{H^1([-M,M];\, \mu)}^2 
			&= 
			\|H_n - \widetilde{H}_{n,M,\varepsilon}\|_{L^2([-M,M];\,\mu)}^2
			+ 
			\|H_n' - \widetilde{H}_{n,M,\varepsilon}'\|_{L^2([-M,M];\,\mu)}^2 \\
			&= 
			\|H_n - \hat{H}_{n,M,\varepsilon}\|_{L^2([-M,M];\,\mu)}^2
			+ 
			\|H_n' - \hat{H}_{n,M,\varepsilon}'\|_{L^2([-M,M];\,\mu)}^2 \\
			&\leq 
			\frac{\varepsilon^2}{2} + \frac{\varepsilon^2}{2} = \varepsilon^2.
		\end{aligned}
	\end{equation*}
	Next, let $x \in \R\backslash[-M-1,M+1]$. By \eqref{eq:defTildeH} it holds that $\widetilde{H}_{n,M,\varepsilon}(x) = 0$ for $x \in \R\backslash[-M-1,M+1]$. Thus we obtain from \eqref{eq:recursiveDerivativeHermite} and \eqref{eq:LpBoundComplement}:
	\begin{equation*}
		\begin{aligned}
			\|H_n - \widetilde{H}_{n,M,\varepsilon}\|_{H^1(\R\backslash[-M-1,M+1];\,\mu)}^2
			&= \|H_n\|_{H^1(\R\backslash[-M-1,M+1];\,\mu)}^2 \\
			&= \|H_n\|_{L^2(\R\backslash[-M-1,M+1];\,\mu)}^2
			+ \|H_n'\|_{L^2(\R\backslash[-M-1,M+1];\,\mu)}^2 \\
			&= \|H_n\|_{L^2(\R\backslash[-M-1,M+1];\,\mu)}^2
			+ n\,\|H_{n-1}\|_{L^2(\R\backslash[-M-1,M+1];\,\mu)}^2 \\
			&\leq \frac{1}{\sqrt{2\pi}} (2n)!! (3M)^{2n} e^{-\frac{M^2}{2}}
			+ \frac{n}{\sqrt{2\pi}} (2(n-1))!! (3M)^{2(n-1)} e^{-\frac{M^2}{2}} \\
			&\leq 2n\,(2n)!!(3M)^{2n}e^{-\frac{M^2}{2}}.
		\end{aligned}
	\end{equation*}
	Finally, consider $x \in I \coloneqq [-M-1,M+1]\backslash[-M, M]$. For the error in this region it holds that
	\begin{equation*}
		\begin{aligned}
			\|H_n - \widetilde{H}_{n,M,\varepsilon}\|_{H^1(I;\,\mu)}^2 
			&= 
			\|H_n - \widetilde{H}_{n,M,\varepsilon}\|_{L^2(I;\,\mu)}^2 
			+
			\|H_n' - \widetilde{H}_{n,M,\varepsilon}'\|_{L^2(I;\,\mu)}^2.
		\end{aligned}
	\end{equation*}
	From \eqref{eq:LpBoundComplement} and \eqref{eq: H tilde uniform bound}, we obtain
	\begin{equation*}
		\begin{aligned}
			\|H_n - \widetilde{H}_{n,M,\varepsilon}\|_{L^2(I;\,\mu)}
			&\leq  
			\|H_n\|_{L^2(I;\,\mu)} + \|\widetilde{H}_{n,M,\varepsilon}\|_{L^2(I;\,\mu)} \\
			&\leq \sqrt{(2n)!!}(3M)^n e^{-\frac{M^2}{4}} + R e^{-\frac{M^2}{4}} \\
			&\leq 
			\sqrt{(2n)!!}(3M)^n e^{-\frac{M^2}{4}} + (1 + (3M)^n) e^{-\frac{M^2}{4}} \\
			&\leq 
			3 \sqrt{(2n)!!}(3M)^ne^{-\frac{M^2}{4}}.
		\end{aligned}
	\end{equation*}
	We proceed similarly with the derivative of $H_n$ and
  $\widetilde{H}_{n,M,\varepsilon}$. Because of item $(iv)$ of \Cref{lemma:truncationNN}, we have 
	\[
	\underset{x \in I}{{\esssup}}|\widetilde{H}'_{n,M,\varepsilon}(x)| 
        \leq \max\left\{R, \underset{x \in [-M-1,M+1]}{\sup}|\hat{H}'_{n,M,\varepsilon}(x)|\right\}.
	\]
	Due to \eqref{eq: worst case H1 error on compact interval}, we have 
	\begin{align*}
		\esssup_{x\in [-M-1,M+1]} |\hat{H}'_{n,M,\varepsilon}(x)| 
		&\leq \frac{\varepsilon}{2} + \sup_{x \in [-M-1,M+1]} |H'_n(x)| 
                \\
		&\leq 1 + \sqrt{n}\underset{x \in [-M-1,M+1]}{\sup} |H_{n-1}(x)| 
                \\
		&\leq 1 + \sqrt{n} (3(M+1))^{n-1}.
	\end{align*}
	Since $R \leq 1 + (3M)^n$, we obtain the bound 
	\begin{equation}\label{eq: uniform bound H tilde prime}
		\esssup_{x \in I} |\widetilde{H}'_{n,M,\varepsilon}(x)| \leq 1 + \sqrt{n}\,(3(M+1))^n.
	\end{equation}
	Furthermore, 
	\begin{align*}
		\|H_n' - \widetilde{H}_{n,M,\varepsilon}'\|_{L^2(I;\,\mu)} 
		&\leq 
		\|H_n'\|_{L^2(I;\,\mu)} + \|\widetilde{H}_{n,M,\varepsilon}'\|_{L^2(I;\,\mu)} \\
		&\leq 
		\sqrt{n} \|H_{n-1}\|_{L^2(I;\,\mu)} + \|\widetilde{H}_{n,M,\varepsilon}'\|_{L^2(I;\,\mu)} \\
		&\leq 
		\sqrt{n} \sqrt{(2n)!!} (3M)^n e^{-\frac{M^2}{4}} + \|\widetilde{H}_{n,M,\varepsilon}'\|_{L^2(I;\,\mu)}.
	\end{align*}
	Due to \eqref{eq: uniform bound H tilde prime} and \eqref{eq:LpBoundComplement}, we have 
	\[
	\|\widetilde{H}_{n,M,\varepsilon}'\|_{L^2(I;\,\mu)} \leq (1 + \sqrt{n} (3(M+1))^n)e^{-\frac{M^2}{4}}.
	\]
	This implies 
	\begin{align*}
		\|H_n' - \widetilde{H}_{n,M,\varepsilon}'\|_{L^2(I;\,\mu)} 
		&\leq 
		\sqrt{n} \sqrt{(2n)!!} (3M)^n e^{-\frac{M^2}{4}} + e^{-\frac{M^2}{4}} + \sqrt{n}(3(M+1))^n e^{-\frac{M^2}{4}} \\
		&\leq 
		3\sqrt{n}\sqrt{(2n)!!}(3(M+1))^n e^{-\frac{M^2}{4}}.
	\end{align*}
	This shows
	\begin{align*}
		\|H_n - \widetilde{H}_{n,M,\varepsilon}\|_{H^1(I;\,\mu)}^2
		&\leq
		\|H_n - \widetilde{H}_{n,M,\varepsilon}\|_{L^2(I;\,\mu)}^2 + \|H_n' - \widetilde{H}_{n,M,\varepsilon}'\|_{L^2(I;\,\mu)}^2 \\
		&\leq 
		9 (2n)!! (3M)^{2n} e^{-\frac{M^2}{2}} + 9n (2n)!! (3(M+1))^{2n} e^{-\frac{M^2}{2}} \\
		&\leq 
		18 n (2n)!! (3(M+1))^{2n} e^{-\frac{M^2}{2}}.
	\end{align*}
	The overall approximation error is therefore bounded by 
	\begin{equation*}
		\begin{aligned}
			\|H_n - \widetilde{H}_{n,M,\varepsilon}\|_{H^1(\R;\,\mu)}
			&\leq 
			\varepsilon
			+ 
			\sqrt{2n}\sqrt{(2n)!!} (3M)^n e^{-\frac{M^2}{4}}
			+ 
			\sqrt{18n} \sqrt{(2n)!!} (3(M+1))^n e^{-\frac{M^2}{4}} \\
			&\leq 
			\varepsilon + 2\sqrt{18n}\sqrt{(2n)!!} (3(M+1))^n e^{-\frac{M^2}{4}} \\
			&\leq 
			\varepsilon + 6\sqrt{2} \sqrt{n} \sqrt{(2n)!!} (3(M+1))^n e^{-\frac{M^2}{4}}.
		\end{aligned}
	\end{equation*}
\end{proof}

		\begin{corollary}\label{cor:UnivariateBoundsSimplyfied}
			Consider the setting of Proposition
      \ref{prop:ExpressionRateUnivariateHermite} and set, for $\varepsilon \in
      (0, {1/2})$ and $n\in\N$
			\begin{equation}\label{eq:M=M(n,eps)}
				M(n, \varepsilon) \coloneqq \sqrt{24 ((n+1) \log(2n)- \log(\varepsilon))} .
			\end{equation}
			With this choice of $M$, 
define the ReLU NN $\widetilde{H}_{n,\varepsilon} \coloneqq \widetilde{H}_{n,M,\varepsilon} \colon \R \to \R$. 
This satisfies
			\begin{enumerate}[label=(\roman*)]
				\item $\|H_n - \widetilde{H}_{n,\varepsilon}\|_{{H^1(\R;\mu)}} \leq {7} \varepsilon$; 
				\item $\sup_{x \in \R} |\widetilde{H}_{n, \varepsilon}(x)| \leq 1 + (3M)^n = 1 + 3^n \left( 24 ({(n+1)} \log(2n)-\log(\varepsilon)) \right)^{\frac{n}{2}}$;
				\item for some $C > 0$ independent of $n$ and $\varepsilon$,
				\begin{align*}
					\size(\widetilde{H}_{n,\varepsilon}) &\leq C \, (1 + n^2 (\log(n) + \log(-\log(\varepsilon)))-n\log(\varepsilon)), \\
					\depth(\widetilde{H}_{n,\varepsilon}) &\leq C \, (1 + n^2 (\log(n) + \log(-\log(\varepsilon)))-n\log(\varepsilon)), \\
					\mpar(\widetilde{H}_{n,\varepsilon})  
					&\leq {(1+3^n)(24((n+1)\log(2n)-\log(\varepsilon)))^{\frac{n+1}{2}}}.
				\end{align*}
			\end{enumerate}
		\end{corollary}

	{
	\begin{proof}
		We begin by showing item $(i)$. By Proposition
    \ref{prop:ExpressionRateUnivariateHermite} $(i)$, we only need to show that
    definition \eqref{eq:M=M(n,eps)} of $M(n,\varepsilon)$ implies $\sqrt{2n(2n)!!}(3(M+1))^n e^{-\frac{M^2}{4}} \leq \varepsilon$. Notice that $\sqrt{(2n)!!} \leq (2n)^n$ and therefore $\sqrt{2n(2n)!!} \leq (2n)^{n+1}$, which means it is sufficient to show that 
		\begin{equation}\label{eq:condition on M}
			- \frac{M^2}{4} + (n+1) \log(2n) + n \log(3(M+1)) - \log(\varepsilon) \leq 0.
		\end{equation}
		The definition of $M$ implies $\frac{M^2}{24} \geq (n+1) \log(2n) - \log(\varepsilon)$, which yields
		\begin{equation}\label{eq: condition on M (2)}
			-\frac{M^2}{24} + (n+1) \log(2n) - \log(\varepsilon) \leq 0.
		\end{equation}
		Now, we show 
		\[
		-\frac{M^2}{6} + n \log(3(M+1)) \leq 0,
		\]
		which then implies \eqref{eq:condition on M}, since $\frac{1}{24} + \frac{1}{6} \leq \frac{1}{4}$. The last inequality can equivalently be written as $\frac{M^2}{\log(3(M+1))} \geq 6n$. The function $x \mapsto \frac{x^2}{\log(3(x+1))}$ is monotonically increasing for $x \geq 1$ and by \eqref{eq:M=M(n,eps)} it holds that $M \geq \sqrt{24 (n+1) \log(2n)} \eqqcolon x$. Therefore
		\[
		\frac{M^2}{\log(3(M+1))} \geq \frac{x^2}{\log(3(x+1))} \geq 6n \frac{4 \log(2n)}{\log(3 \sqrt{24 (n+1) \log(2n)}+1)}.
		\]
		Now, it is sufficient to show that $\frac{4 \log(2n)}{\log(3 \sqrt{24 (n+1) \log(2n)}+1)} \geq 1$ for all $n \in \N$. Since this term is monotonically increasing in $n \geq 1$, it suffices to check the case $n = 1$, which can be verified directly.
Combining this with \eqref{eq: condition on M (2)}, 
we obtain \eqref{eq:condition on M}. 
Applying this result to Proposition \ref{prop:ExpressionRateUnivariateHermite} $(i)$ 
yields $\|H_n - \widetilde{H}_{n,\varepsilon}\|_{H^1(\R, \mu)} \leq 7 \varepsilon$. 

		We now continue with the bounds on the size and depth of $\widetilde{H}_{n,\varepsilon}$. 
The following calculation is performed for $\size(\widetilde{H}_{n,\varepsilon})$. 
The bound on $\depth(\widetilde{H}_{n,\eps})$ follows completely analogously 
(compare Proposition \ref{prop:ExpressionRateUnivariateHermite} item $(iii)$). 
Now, inserting the choice of $M(n,\varepsilon)$ from \eqref{eq:M=M(n,eps)} 
into the size bound from Proposition \ref{prop:ExpressionRateUnivariateHermite} $(iii)$,
we get
		\[
		\size(\widetilde{H}_{n,\varepsilon}) 
		\leq
		C \, \left( 1 + \frac{1}{2}n^2 \log(24 ((n+1) \log(2n)-\log(\varepsilon))) + n \log(\varepsilon^{-1}) \right).
		\]
For $a,b\geq 1$ it holds
\[
	\log(a+b) \leq 1 + \log(a) + \log(b)	
\]
due to $a+b \leq 2ab$ for $a \geq 1$ and $b\geq 1$ and the monotonicity of the logarithm.
		
		Now set $a \coloneqq 24 (n+1) \log(2n)$ and $b \coloneqq -24 \log(\varepsilon)$, which yields
		\begin{align*}
			\size(\widetilde{H}_{n,\varepsilon})
			&\leq 
			C \, \left( 1 + \frac{1}{2}n^2 (1 + \log(24(n+1)) + \log(\log(2n)) + \log(-24\log(\varepsilon))) + n \log(\varepsilon^{-1}) \right) \\
			&\leq C \, \left( 1 + n^2 ( \log(n) + \log(-\log(\varepsilon))) + n \log(\varepsilon^{-1})  \right)
		\end{align*}
		with a constant $C$ independent of $n$ and $\varepsilon$. 
    This shows the bound on the size (and the depth) of $\widetilde{H}_{n, \epsilon}$from item $(iii)$. 
		Inserting the definition of $M(n,\varepsilon)$ into item $(ii)$ of
    Proposition \ref{prop:ExpressionRateUnivariateHermite},
    and the weight bound from 
    item $(iii)$ Proposition \ref{prop:ExpressionRateUnivariateHermite}
    finishes the proof.
	\end{proof}
}
		Emulating univariate Hermite polynomials using RePU-NN is significantly less involved compared to the ReLU-case.
		\begin{proposition}\label{prop:RePUHermiteUnivariate}
			Let $n \in \N_0, \, q \in \N, \, q \geq 2$. 
                        Then there exists a $\sigma_q$-network $\widetilde{H}_n \colon \R \to \R$ such that 
			\[
			\widetilde{H}_n (x) = H_n (x)
			\]
			for all $x \in \R$.  
			Furthermore, there exists a constant $C = C(q) > 0$ only dependent on $q$ such that 
			\begin{equation}
				\depth(\widetilde{H}_n) \leq C n, \quad 
				\size(\widetilde{H}_n) \leq C n, \quad 
				\mpar(\widetilde{H}_n) \leq C.
			\end{equation}
		\end{proposition}
		
		\begin{proof}
		Let $c \in \R^{n+1}$ be the vector containing the coefficients from \eqref{eq:HermitExplicit} 
                and define $c_\infty \coloneq \max\{1, \|c\|_\infty\}$. 
                Lemma \ref{lemma:CoeffSummability} shows that $\|c\|_\infty \leq \|c\|_{\ell^1} \leq 3^n$. 
                Now Proposition \ref{prop:RePUPolynomialEmulation} implies the desired result.
		\end{proof}

	\subsection{Multivariate Hermite Polynomials} 
		We proceed with ReLU NN expression bounds for multivariate 
        tensorized Hermite polynomials.  
        We follow the proof of \cite[Theorem 3.9]{SZ21_982} with minor adaptations.
        Recall that the maximal polynomial degree $m(\Lambda)$ and the effective
        dimension $d(\Lambda)$ of a multiindex $\Lambda$ are introduced in
        Definition~\ref{def:dmLambda}. 
        As we did previously, in what follows we identify the NNs
          $\widetilde{H}_{\varepsilon, \mathbf{\bsnu}}$ with functions 
          $\R^\N \to \R$ that only depend on the variables with indices in $\supp\bsnu$.
		\begin{proposition}\label{prop:ExpressionRatesMultHermite.}
			Let $\Lambda \subseteq \calF$ be finite and downward closed. 
                 Then, for every $\varepsilon \in (0, {1/2})$ 
                 there exists a ReLU NN 
                 $\Phi = \{\widetilde{H}_{\varepsilon, \mathbf{\bsnu}}\}_{\mathbf{\bsnu} \in \Lambda} \colon \R^{\left|\supp{\Lambda}\right|} 
                   \to \R^{|\Lambda|}$ depending solely on the variables $(y_i)_{i \in \supp{\Lambda}}$, 
                   such that
		\begin{equation*}
				\underset{\mathbf{\bsnu} \in \Lambda}{\max} \|H_{\mathbf{\bsnu}} - \widetilde{H}_{\varepsilon, \mathbf{\bsnu}}\|_{{H^1(\R^\N; \bsmu)}} \leq \varepsilon,
		\end{equation*}
		and there exists a positive constant $C$ (independent of $m(\Lambda), d(\Lambda)$, 
                and of $\varepsilon \in (0, {1/2})$) 
                such that 
			\begin{align*}
				\size(\Phi) &\leq C (1 + |\Lambda|m(\Lambda)^3 \log(1+m(\Lambda)){\log(1+d(\Lambda))}d(\Lambda)^{{3}}\log(\varepsilon^{-1})), \\
				\depth(\Phi) &\leq C (1 + d(\Lambda)^{{2}} \log(1+d(\Lambda))^{{2}}m(\Lambda)^2\log(1+m(\Lambda))\log(\varepsilon^{-1})), \\
				\mpar(\Phi) &\leq {C \, 3^{m(\Lambda)} (m(\Lambda)\log(1+m(\Lambda))d(\Lambda)\log(\varepsilon^{-1}))^{\frac{m(\Lambda)+1}{2}}} .
			\end{align*}
		\end{proposition}
		
		\begin{proof}
			Fix $\varepsilon \in (0, {1/2})$ and set $m \coloneqq m(\Lambda)$ and $d \coloneqq d(\Lambda)$.
Without loss of generality, assume that $m \geq 1$ and $d \geq 1$. 
Otherwise, we have $\Lambda = \emptyset$ or $\Lambda = \{\mathbf{0}\}$ and the proof concludes immediately. 
Take $M$ as defined in \eqref{eq:M=M(n,eps)} and set 
			\begin{equation*}
				D \coloneqq 1 + (3M)^m.
			\end{equation*}
	Now let $\eps_1 := \eps/\sqrt{md}$ and choose $\eps_2 > 0$ 
such that $d (\sqrt{1 + m})^{d-1} \eps_2 (\frac{1}{2} + \sqrt{1+m})^{d-1} \leq \varepsilon.$
	We start by defining $\widetilde{H}_{\varepsilon, \mathbf{\nu}}$ and showing that $\|H_{\mathbf{\nu}} - \widetilde{H}_{\varepsilon, \mathbf{\nu}}\|_{H^1(\R^\N;\, \bsmu)} \leq \varepsilon$. 
	Define $H_{\mathbf{0}} \coloneqq 1$, and for $\mathbf{0} \neq \mathbf{\nu} \in \Lambda$, 
		\begin{equation*}
			\widetilde{H}_{\eps_2, \mathbf{\nu}}(x) \coloneqq \widetilde{\prod}_{|\mathbf{\nu}|_0, D, \eps_1} ((\widetilde{H}_{\eps_2, \nu_j}(x_j))_{j \in \supp{\mathbf{\nu}}}).
		\end{equation*}
		We can now decompose the overall approximation error into two contributions. 
For $\mathbf{0} \neq \mathbf{\nu} \in \Lambda$, we have
			\begin{equation}\label{eq: error decomposition}
				\begin{aligned}
					\|H_{\mathbf{\nu}} - \widetilde{H}_{\eps_2, \mathbf{\nu}}\|_{H^1(\R^\N;\, \bsmu)} 
					&\leq \left\| H_{\mathbf{\nu}} - \prod_{j \in \supp{\mathbf{\nu}}} \widetilde{H}_{\eps_2, \nu_j} \right\|_{H^1(\R^\N;\, \bsmu)} \\
					&+ \left\| \prod_{j \in \supp{\mathbf{\nu}}} \widetilde{H}_{\eps_2, \nu_j} - \widetilde{\prod}_{|\mathbf{\nu}|_0, D, \varepsilon_1}((\widetilde{H}_{\eps_2, \nu_j})_{j \in \supp{\mathbf{\nu}}}) \right\|_{H^1(\R^\N;\, \bsmu)}.
				\end{aligned}
			\end{equation}
			By Corollary \ref{cor:UnivariateBoundsSimplyfied} $(ii)$ we have $\sup_{x \in \R}|\widetilde{H}_{\eps_2, \nu_j}(x)| \leq 1 + (3M)^{\nu_j} \leq D$ for all $j \in \supp{\mathbf{\nu}}$. Now, Proposition \ref{prop:H1MultWithNNumbers} implies 
			\[
			\left\|\prod_{j \in \supp{\mathbf{\nu}}} \widetilde{H}_{\eps_2,\nu_j} - \widetilde{H}_{\eps_2, \mathbf{\nu}}\right\|_{H^1(\R^\N;\, \bsmu)} \leq \varepsilon.
			\]
			Next, we bound the first term in \eqref{eq: error decomposition} and compute
			\begin{align*}
				\left\| \prod_{j \in \supp{\mathbf{\nu}}} H_{\nu_j} - \prod_{j \in \supp{\mathbf{\nu}}} \widetilde{H}_{\eps_2, \nu_j} \right\|_{H^1(\R^\N;\, \bsmu)} 
				&\leq \sum_{j \in \supp{\mathbf{\nu}}} \, \prod_{\substack{i\in\supp{\mathbf{\nu}}\\{i < j}}} \|H_{\nu_i}\|_{H^1(\R;\, \mu)} \\
				&\cdot \|H_{\nu_j} - \widetilde{H}_{\eps_2,\nu_j}\|_{H^1(\R;\,\mu)} 
				\cdot \prod_{\substack{i \in \supp{\mathbf{\nu}}\\{i > j}}} \|\widetilde{H}_{\eps_2,\nu_i}\|_{H^1(\R;\,\mu)}.
			\end{align*}
			For all $i$ it holds that $\|H_{\nu_i}\|_{L^2(\R,\mu)} = 1$ and therefore
      $\|H_{\nu_i}\|_{H^1(\R;\,\mu)} = \sqrt{1 + \nu_i} \leq \sqrt{1+m}$. It
      also holds that $\|H_{\nu_i} - \widetilde{H}_{\eps_2, \nu_i}\|_{H^1(\R;\,\mu)} \leq \eps_2$ by Corollary \ref{cor:UnivariateBoundsSimplyfied}, which implies $\|\widetilde{H}_{\eps_2, \nu_i}\|_{H^1(\R;\,\mu)} \leq \eps_2 + \sqrt{1+m} \leq \frac{1}{2} + \sqrt{1+m}$, since $\eps_2 \leq \varepsilon < \frac{1}{2}$. This leads to 
			\begin{align*}
				\left\| \prod_{j \in \supp{\mathbf{\nu}}} H_{\nu_j} - \prod_{j \in \supp{\nu}} \widetilde{H}_{\eps_2, \nu_j} \right\|_{H^1(\R^\N;\, \bsmu)} 
				&\leq
				|\mathbf{\nu}|_0\left(\sqrt{1+m}\right)^{|\mathbf{\nu}|_0-1} \eps_2 \left(\frac{1}{2} + \sqrt{1+m}\right)^{|\mathbf{\nu}|_0-1} \\
				&\leq d \left(\sqrt{1+m}\right)^{d-1} \eps_2 \left(\frac{1}{2} + \sqrt{1 + m}\right)^{d-1} \\
				&\leq \varepsilon.
			\end{align*}
			Now, we construct $\Phi = \{\widetilde{H}_{\varepsilon, \mathbf{\nu}}\}_{\mathbf{\nu}\in\Lambda}$ and derive bounds on the size, depth and weights of $\Phi$. Let $\Phi_1 \colon \R^{\left|\supp{\Lambda}\right|} \to \R^{m \left|\supp{\Lambda}\right|}$, with output 
			\begin{equation}\label{def of Psi_1}
				\Phi_1(\mathbf{y}) \coloneqq \left\{ \widetilde{H}_{\eps_2, j}(y_i) \right\}_{i \in \supp{\Lambda}, j \in \{1,\dots,m\}}.
			\end{equation}
			By the definition of $\eps_2$ we have 
			\begin{align*}
				\log(\eps_2) 
				&\leq 
				\log(\varepsilon) - \log(d) + \frac{1-d}{2}\log(1+m) + (1-d) \log\left(\frac{1}{2} + \sqrt{1+m}\right) \\
				&\leq 
				C \, ( \log(\varepsilon) - \log(d) + (1-d)\log(1+m) + (1-d) )
			\end{align*}
			with a constant $C > 0$ independent of $d, m$ and $\varepsilon$.
			Because of Corollary \ref{cor:UnivariateBoundsSimplyfied}, for each $j \leq m$,
			\begin{equation}\label{size bound for tilde H}
				\begin{aligned}
					\size(\widetilde{H}_{\eps_2, j}) 
					&\leq
					C \, ( 1 + j^2 (\log(j) + \log(-\log(\eps_2))) - j \log(\eps_2) ) \\
					&\leq 
					C \, ( 1 + j^2 ( \log(j) + \log(-\log(\varepsilon)) + \log(d) + \log(1+m) ) \\
					&\qquad + j\log(\varepsilon^{-1}) + j\log(d) + jd\log(1+m) + jd ) \\
					&\leq 
					C \, ( 1 + m^2\log(m) + m^2\log(-\log(\varepsilon)) + m^2\log(d) + m^2\log(1+m) + m\log(\varepsilon^{-1}) \\
					&\qquad + m\log(d) + md\log(1+m) + md ) \\
					&\leq 
					C \, ( 1 + m^2\log(1+m) + m^2\log(1+d) + m^2\log(-\log(\varepsilon)) + md - m\log(\varepsilon) ) \\
					&\eqqcolon C \, C_0(m,d,\varepsilon).
				\end{aligned}
			\end{equation}
			We also obtain 
			\begin{equation}\label{depth bound tilde H}
				\depth(\widetilde{H}_{\eps_2, j}) 
				\leq C C_0(m,d,\varepsilon), 
			\end{equation}
			since the size and depth bounds in Corollary
                        \ref{cor:UnivariateBoundsSimplyfied} are the same 
                        (with possibly different constant $C$). 
                        By invoking Lemma \ref{lemma:Id-ReLU} to extend the depth if necessary,
      we may assume that each
      $\widetilde{H}_{\tilde\varepsilon, j}(y_j)$ in the definition of
      $\Phi_1(\mathbf{y})$ has the same depth $\lceil C C_0(m,d,\varepsilon)\rceil$, 
      and the size is bounded by $C C_0(m,d,\varepsilon)$
      with
      a suitable constant $C > 0$ independent of $m, d$, and $\varepsilon$. 
			Now denote by $\Phi_2 \colon \R^{\sum_{\mathbf{\nu}\in\Lambda}|\mathbf{\nu}|_0} \to \R^{|\Lambda|}$ the network 
			\begin{equation}\label{def of Psi_2}
				\Phi_2 \coloneqq \left\{ \widetilde{\prod}_{|\mathbf{\nu}|_0, D, \varepsilon_1} \right\}_{\mathbf{\nu}\in\Lambda}.
			\end{equation}
			We now concatenate the networks $\Phi_2$ and $\Phi_1$ such that each subnetwork $\widetilde{\prod}_{|\mathbf{\nu}|_0, D,\varepsilon_1}$ obtains the inputs $(\widetilde{H}_{\eps_2,\nu_j}(y_j))_{j \in \supp{\mathbf{\nu}}} \in \R^{|\mathbf{\nu}|_0}$.
			This leads to the concatenated network 
			\begin{equation*}
				\Phi \coloneqq \left\{ \widetilde{\prod}_{|\mathbf{\nu}|_0, D, \varepsilon_1} ((\widetilde{H}_{\eps_2,\nu_j})_{j \in \supp{\mathbf{\nu}}}) \right\}_{\mathbf{\nu}\in\Lambda}.
			\end{equation*}
			The size, depth, and weight bounds can be obtained from Proposition \ref{prop:H1MultWithNNumbers}. We have 
			\begin{align*}
				\size(\widetilde{\prod}_{|\mathbf{\nu}|_0, D,\varepsilon_1}) 
				&\leq 
				C \, (1 + |\mathbf{\nu}|_0 \log(|\mathbf{\nu}|_0) + |\mathbf{\nu}|_0^2 \log(D) - |\mathbf{\nu}|_0\log(\varepsilon_1)) \\
				&\leq C \, (1 + d \log(d) + d^2\log(D) - d\log(\varepsilon_1)) \\
				&\leq 
				C \, (1 + d\log(d) + d^2\log(D) -d\log(\eps) + d\log(md)).
			\end{align*}
			By definition of $D$ and $M$, using $\log(1 + x) \leq 1 + \log(x)$ for $x \geq 1$, 
			\begin{align*}
				\log(D) 
				&\leq
				1 + m \log(3 \sqrt{24((m+1)\log(2m)-\log(\tilde\varepsilon))}) \\
				&\leq
				C \, (1 + m\log(m)+m\log(-\log(\tilde \varepsilon))) \\
				&\leq 
				C \, (1 + m\log(m)+m\log(-\log(\varepsilon))+md\log(1+m)).
			\end{align*}
			Thus
			\begin{multline}\label{size bound prod operator}
					\size(\widetilde{\prod}_{|\mathbf{\nu}|_0, D,\varepsilon_1})
					\leq C (1 + d\log(d) + d^2m\log(m) + d^2m\log(-\log(\varepsilon)) \\
					+ md^3\log(1+m) - d\log(\varepsilon) + d\log(md)) \eqqcolon C B_0(m,d,\varepsilon).
			\end{multline}
			Additionally, by Proposition \ref{prop:H1MultWithNNumbers}, we have
			\begin{equation}\label{depth bound prod operator}
				\begin{aligned}
					\depth(\widetilde{\prod}_{|\mathbf{\nu}|_0, D,\varepsilon_1})
					&\leq 
					C (1 + \log(d)(\log(d)+d\log(D)-\log(\varepsilon_1))) \\
					&\leq C (1 + \log(d)^2 + d\log(d)m\log(m) + d\log(d)m\log(-\log(\varepsilon)) \\
					&\qquad + \log(d)d^2m\log(1+m) -\log(d)\log(\varepsilon)+\log(d)\log(md)) \\
					&\eqqcolon C B_1(m,d,\varepsilon). 
				\end{aligned}
			\end{equation}
			As before, we assume that all networks $\left(
        \widetilde{\prod}_{|\mathbf{\nu}|_0, D,\varepsilon_1}
      \right)_{\mathbf{\nu}\in\Lambda}$ have the same depth by concatenating
      $\widetilde{\prod}_{|\mathbf{\nu}|_0, D,\varepsilon_1}$ with identity
      networks.
      This, again, yields a uniform (w.r.t. $\bsnu \in \Lambda$)
      bound on the size of $\tilde \prod_{|\bsnu|_0, D, \eps_1}$
      given by $C B_0(m,d,\varepsilon)$.
	We are now in a position to bound the size and depth of $\Phi$. 
      Downward closedness of $\Lambda$ implies $\left|\supp{\Lambda}\right| \leq |\Lambda|$. 
      Moreover, the number of required connections (i.e., nonvanishing network weights) between $\Phi_1$ and $\Phi_2$
      can be bounded by $O(m |\Lambda|)$. 
      By \eqref{def of Psi_1}, \eqref{size bound for tilde H} and \eqref{def of Psi_2}, 
      \eqref{size bound prod operator} there exists a constant $C > 0$ 
      such that for all $\varepsilon \in (0, 1/2)$ it holds that 
			\begin{align*}
				\size(\Phi)
				&\leq
				C (1 + \size(\Phi_1) + \size(\Phi_2) + m |\Lambda|) \\
				&\leq C (1 + (|\supp{\Lambda}|m)C_0(m,d,\varepsilon)+|\Lambda|B_0(m,d,\varepsilon)+m|\Lambda|) \\
				&\leq C\, |\Lambda| \, ( 1 + m^3\log(1+m) + m^3\log(1+d) + m^3\log(-\log(\varepsilon)) + m^2d \\
				&\qquad  -m^2\log(\varepsilon) + d\log(d) + d^2m\log(m) d^2m\log(-\log(\varepsilon)) + md^3\log(1+m) - d\log(\varepsilon)+d\log(md) ) \\
				&\leq C |\Lambda| \, ( 1 + m^3 \log(1+m)\log(1+d)\log(\varepsilon^{-1})d^3\log(md) ).
			\end{align*}
      In a similar way, by \eqref{depth bound tilde H} and \eqref{depth bound prod operator}
			\begin{align*}
				\depth(\Phi)
				&\leq 
				C (1 + \depth(\Phi_1) + \depth(\Phi_2) \\
				&\leq 
				C (1 + C_0(m,d,\varepsilon) + B_1(m,d,\varepsilon)) \\
				&\leq C \, ( 1 + m^2\log(1+m) m^2\log(1+d) + m^2\log(-\log(\varepsilon)) + md \\
				&\qquad - \log(\varepsilon) + \log(d)^2 + d\log(d)m\log(m) + d\log(d)m\log(-\log(\varepsilon)) \\
				&\qquad + \log(d)d^2m\log(1+m) -\log(d)\log(\varepsilon)+\log(d)\log(md) ) \\
				&\leq C \, (1 + m^2\log(1+m)\log(1+d)^2\log(\varepsilon^{-1})d^2\log(md))
			\end{align*}
	Since the approximation to the product operator can be realized with weights and biases bounded by $1$, 
	we obtain $\mpar(\Phi_2) \leq 1$. 
	From Corollary \ref{cor:UnivariateBoundsSimplyfied} we obtain 
			\begin{align*}
				\mpar(\widetilde{H}_{\eps_2, j})
				&\leq (1 + 3^{j}) \left( 24 ((j+1) \log(2j)-\log(\eps_2)) \right)^{\frac{j+1}{2}} \\
				&\leq C \, (1 + 3^j)(24 (j+1)\log(2j) + \log(\varepsilon^{-1}) + d\log(1+m))^{\frac{j+1}{2}} \\
				&\leq C \, 3^m (m \log(1+m) + \log(\varepsilon^{-1}) + d \log(1+m))^{\frac{m+1}{2}}
			\end{align*}
	for all $j \in \{1,\dots,m\}$. 
	This implies $\mpar(\Phi_1) \leq C \, 3^m (m \log(1+m) + \log(\varepsilon^{-1}) + d \log(1+m))^{\frac{m+1}{2}}$ 
        and finally
			\begin{align*}
				\mpar(\Phi) 
				&\leq \max\{\mpar(\Phi_1),\mpar(\Phi_2)\} \\
				&\leq C \, 3^m (m \log(1+m) + \log(\varepsilon^{-1}) + d \log(1+m))^{\frac{m+1}{2}} \\
				&\leq C \, 3^m (m\log(1+m)d\log(\varepsilon^{-1}))^{\frac{m+1}{2}}.
			\end{align*}
		\end{proof}
Now, Lemma \ref{lemma:normalizeWeights} implies Proposition \ref{prop:reluhermite}.

	\begin{proof}[Proof of Proposition \ref{prop:reluhermite}]
		Assume, as in the proof of Proposition \ref{prop:ExpressionRatesMultHermite.}, without loss of generality $m \coloneqq m(\Lambda)$ and $d \coloneqq d(\Lambda)$ to be at least one.
 Denote by $\widetilde{\Phi}$ a ReLU NN satisfying Proposition \ref{prop:ExpressionRatesMultHermite.}. 
 The logarithm of the maximum absolute value of its weights can be bounded by
		\begin{align*}
			\log\left(\mpar(\widetilde{\Phi})\right) &\leq C \, (1 + m\log(1+m) + m\log(1+d) + m\log(\varepsilon^{-1})) \\
			&\leq C \, (1 + m\log(1+m)\log(1+d)\log(\varepsilon^{-1})).
		\end{align*}
		Lemma \ref{lemma:normalizeWeights} implies that there exists a ReLU NN
    $\Phi$ such that $\Phi(x) = \widetilde\Phi (x)$ everywhere, with ${\rm
      mpar}(\Phi)\leq 1$ and
		\begin{align*}
			\depth(\Phi) &\leq C \, \left(\log\left(\mpar(\widetilde{\Phi})\right)+1\right) \, \left(\depth(\widetilde{\Phi})+1\right) \\
			&\leq 
			C \, (1 + m\log(1+m)\log(1+d)\log(\varepsilon^{-1})) \, (1 + d^2\log(1+d)^2m^2\log(1+m)\log(\varepsilon^{-1})\log(md)) \\
			&\leq 
			C \, (1 + m^3 \log(1+m)^2\log(1+d)^3\log(\varepsilon^{-1})^2d^2\log(md)).
		\end{align*}
    Finally,
		\begin{align*}
			\size(\Phi) &\leq C \, \left( \size(\widetilde{\Phi}) + |\Lambda| \left(\left(\depth(\widetilde{\Phi})+1\right)\log\left(\mpar(\widetilde{\Phi})\right) + 1\right) \right) \\
			&\leq 
			C \, (1 + |\Lambda|m^3\log(1+m)\log(1+d)d^3\log(\varepsilon^{-1})\log(md) \\
			&\qquad + |\Lambda|(1 + m^3\log(1+m)^2\log(1+d)^3\log(\varepsilon^{-1})^2d^2\log(md))) \\
			&\leq C \, (1 + |\Lambda|m^3\log(1+m)^2\log(1+d)^3d^3\log(\varepsilon^{-1})^2\log(md)),
		\end{align*}
    which concludes the proof.
	\end{proof}
{In the following statement, we consider the RePU analog to Proposition~\ref{prop:ExpressionRatesMultHermite.}.
	\begin{proposition}\label{prop:RePUMultHermite}
			Let $\Lambda \subseteq \calF$ be finite and downward closed and let $q \in \N, \, q \geq 2$. 
			Then there exists a $\sigma_q$-network 
			$\Phi \coloneq \{\widetilde{H}_\bsnu\} \colon \R^{\left|\supp \Lambda\right|} \to \R^{|\Lambda|}$ 
			(dependent solely on the variables $(y_i)_{i \in \supp \Lambda}$) 
			such that 
			\[
			\widetilde{H}_\bsnu (\bsy) = H_\bsnu (\bsy) \qquad \forall \, \bsnu \in \Lambda, \, \bsy \in \R^\N,
			\] 
			and there exists a constant $C = C(q) > 0$ 
			such that 
			\[
			\depth(\Phi) \leq C m(\Lambda)\log(d(\Lambda)+1), \quad
			\size(\Phi) \leq C |\Lambda| m(\Lambda)^2 d(\Lambda), \quad
			\mpar(\Phi) \leq C.
			\]
	\end{proposition}
	\begin{proof}
	The network construction in this proof is analogous to the one in Proposition \ref{prop:ExpressionRatesMultHermite.}.
        We detail the construction for the convenience of the reader. 
        Since the constitution parts can now be realized exactly using RePU-networks, 
        we emulate the multivariate Hermite polynomials without error. \par
	Set $m \coloneq m(\Lambda)$ and $d \coloneq d(\Lambda)$. 
        Without loss of generality, assume that $m \geq 1$ and $d \geq 1$. 
        In this proof $C = C(q)> 0$ denotes a generic constant only dependent on $q$.
	We begin by defining $\widetilde{H}_\bsnu$. Set $H_\bszero \coloneq 1$, and for $\bszero \neq \bsnu \in \Lambda$, 
			\[
			\widetilde{H}_\bsnu(y) \coloneq \widetilde{\prod}_{|\bsnu|_0}((\widetilde{H}_{\nu_j}(y_j))_{j\in\supp \bsnu}).
			\]
			Due to Proposition \ref{prop:RePUHermiteUnivariate} and Proposition \ref{prop:RePUMultOp} we have
			\[
			H_\bsnu (\bsy) = \widetilde{H}_\bsnu (\bsy) \qquad \forall \, \bsnu \in \Lambda, \, \bsy \in \R^\N.	
			\]
			Next, we construct $\Phi = \{\widetilde{H}_\bsnu\}_{\bsnu \in \Lambda}$
      and derive upper bounds on its size, depth, and weight magnitude. 
      Set $\Phi_1 \colon \R^{\left|\supp \Lambda\right|} \to \R^{m \left|\supp \Lambda\right|}$ with output 
			\[
			\Phi_1(\bsy) = \left\{\widetilde{H}_j(y_i)\right\}_{i \in \supp \Lambda,\, j \in \{1,\dots,m\}}.
			\]
		Because of Proposition \ref{prop:RePUHermiteUnivariate}, for each $j \leq m$, we have 
		\begin{align*}
			\size(\widetilde{H}_j) &\leq C m.
		\end{align*}
		Now denote by $\Phi_2 \colon \R^{\sum_{\bsnu \in \Lambda}|\bsnu|_0} \to \R^{|\Lambda|}$ the network 
		\[
		\Phi_2 \coloneq \left\{\widetilde{\prod}_{|\bsnu|_0}\right\}_{\bsnu \in \Lambda}.
		\]
	We concatenate the networks $\Phi_2$ and $\Phi_1$ 
	such that each subnetwork  $\widetilde{\prod}_{|\bsnu|_0}$ 
	obtains the inputs $(\widetilde{H}_{\nu_j}(y_j))_{j \in \supp \bsnu} \in \R^{|\bsnu|_0}$.  
	From this we obtain the concatenated network
			\[
			\Phi \coloneq \left\{\widetilde{\prod}_{|\bsnu|_0}((\widetilde{H}_{\nu_j}))_{j \in \supp \bsnu}\right\}_{\bsnu \in \Lambda}.
			\]
	Next we estimate the size and depth of $\Phi_2$. 
	Proposition \ref{prop:RePUMultOp} implies that there is a constant $C>0$ such that
			\[
			\size\left(\widetilde{\prod}_{|\bsnu|_0}\right)  \leq C |\bsnu|_0 \leq C d,  \;\;
			\depth\left(\widetilde{\prod}_{|\bsnu|_0}\right) \leq C \log(|\bsnu|_0) \leq C \log(d) 
			\;.
			\]
			Now the size and depth of $\Phi$ can be bounded. 
			Downward closedness of $\Lambda$ implies $\left|\supp \Lambda\right| \leq |\Lambda|$. 
			Moreover, the number of required network connections
      (i.e., nonvanishing network weights) between $\Phi_1$ and $\Phi_2$ 
      can be bounded by $O(m|\Lambda|)$. 
			This implies 
			\begin{align*}
				\size(\Phi) &\leq C (1 + \size(\Phi_1) + \size(\Phi_2) + m |\Lambda|) \\
				&\leq C (m^2 |\supp \Lambda| + |\Lambda|d + m |\Lambda|) \\
				&\leq C (m^2 |\Lambda| + |\Lambda|d + m|\Lambda|) \\
				&\leq C |\Lambda|m^2d.
			\end{align*}
	For the depth we obtain
			\begin{align*}
				\depth(\Phi) &\leq C (1 + \depth(\Phi_1) + \depth(\Phi_2)) \leq C m\log(d+1).
			\end{align*}
	Since all the subnetworks in the construction of $\Phi$ have weights bounded by a constant $C = C(q) > 0$,
        the same holds true for $\Phi$ itself. 
	\end{proof}
	}

	\section{Auxiliary Results}
	\label{app:AuxiliaryResults/Definitions}

	We start with an integration by parts formula in the Gaussian setting which is 
        similar to what can be found in \cite[Lemma 10.1]{DPIntro06} and \cite[Lemma 9.1.1.]{lunardi2015}.
	\begin{lemma}\label{lemma:finiteIBP}
		Let $f \in C^1(\R) \cap H^1(\R; \mu_{0,\lambda})$ with $\mu_{0,\lambda} \coloneqq \calN(0,\lambda)$. Then it holds that
		\[
			\int_{\R} f'(x) \, {\rm d}\mu_{0,\lambda}(x) 
			=
			\int_{\R} \frac{x}{\lambda} f(x) \, {\rm d}\mu_{0,\lambda}(x).
		\]
	\end{lemma}
	\begin{proof}
	Let $f \in C^1(\R) \cap H^1(\R; \mu_{0,\lambda})$. 
        By the usual integration by parts formula,
		\begin{align*}
			\int_\R f'(x) \, {\rm d}\mu_{0,\lambda}(x)
			=
			\frac{1}{\sqrt{2\pi\lambda}}\int_\R f'(x) e^{-\frac{x^2}{2\lambda}} \, {\rm d}x	
			=
			\frac{1}{\sqrt{2\pi\lambda}} \int_\R \frac{x}{\lambda}f(x)e^{-\frac{x^2}{2\lambda}} \, {\rm d}x 
			=
			\int_\R \frac{x}{\lambda} f(x) \, {\rm d}\mu_{0,\lambda}(x).
		\end{align*}
	Note that the boundary term vanishes since $f \in C^1(\R) \cap H^1(\R;\mu_{0,\lambda})$.
	\end{proof}
	\begin{lemma}\label{lemma:infiniteIBP}
Let $Q \in \calL(\calX)$ denote the covariance operator of the centered Gaussian measure $\gamma$
        on $(\calX, \calB(\calX))$ with eigenvalues and eigenfunctions $(\lambda_j,\phi_j)$.
	Let $f \in C^1(\calX) \cap L_\gamma^2(\calX)$ be such that
  $\frac{\partial}{\partial \phi_i}f \in L_\gamma^2(\calX)$ for all $i\in \N$.
        Then 
		\begin{equation*}
			\int_{\calX} \frac{\partial}{\partial \phi_i}f(X) \, {\rm d}\gamma(X) 
			=
			\int_\calX \frac{\langle X, \phi_i \rangle}{\lambda_i} f(X) \, {\rm d}\gamma(X)
		\end{equation*}
	for all $i \in \N$.
	\end{lemma}
	\begin{proof}
	First, define the coordinate maps $X \mapsto \xi_j(X) \coloneqq \langle X, \phi_j \rangle$ and set $\bsxi \coloneqq (\xi_j)_{j\in\N}$.
	Denote by $\calJ_\calX \colon \calX \to \ell^2(\N), \; X \mapsto (\langle X,
    \phi_j \rangle)_{j\in\N}$ the coordinate isomorphism for the ONB
    $\{\phi_j\}_{j\in\N} \subset \calX$. 
    Let $\bslambda \coloneqq (\lambda_j)_{j\in\N}$ be 
    the sequence of covariance eigenvalues and set $\bsmu_\bslambda \coloneqq
    \bigotimes_{j\in\N} \mu_{0,\lambda_j}$, $\bsmu_\bslambda^{<i} \coloneqq
    \bigotimes_{j=1}^{i-1}\mu_{0,\lambda_j}$, and $\bsmu_\bslambda^{>i} \coloneqq
    \bigotimes_{j=i+1}^{\infty} \mu_{0,\lambda_j}$, i.e., $\bsmu_\bslambda = \bsmu_\bslambda^{<i}
    \otimes \mu_{0,\lambda_i} \otimes \bsmu_\bslambda^{>i}$ and $\bsmu_\bslambda =
    (\calJ_\calX)_{\#}\gamma$.
    Furthermore, define $\bsxi_{<i}, \bsxi_{>i}$ such
    that $\bsxi = \bsxi_{<i} \times \xi_i \times \bsxi_{>i}$ and similarly
    consider the splitting $\R^\N = \R^{\N_{<i}} \times \R \times \R^{\N_{>i}}$.
    Finally, define $\tilde f  \coloneqq
    f\circ\calJ_\calX^{-1}:\ell^2(\N)\to \R$. 
A change of variables, Fubini's theorem and Lemma \ref{lemma:finiteIBP} imply
		\begin{align*}
			\int_\calX \frac{\partial}{\partial \phi_i} f(X) \, {\rm d}\gamma(X)
			&=
			\int_{\R^\N} \frac{\partial}{\partial \xi_i} \tilde f(\bsxi) \, {\rm d}\bsmu_\bslambda(\bsxi) \\
			&= 
			\int_{\R^{\N_{<i}}} \int_\R \int_{\R^{\N_{>i}}} \frac{\partial}{\partial \xi_i} \tilde f(\bsxi_{<i},\xi_i,\bsxi_{>i}) \, {\rm d}\bsmu_\bslambda^{>i}(\bsxi_{>i}) {\rm d}\mu_{0,\lambda_i}(\xi_i){\rm d}\bsmu_\bslambda^{<i}(\bsxi_{<i}) \\
			&=
			\int_{\R^{\N_{<i}}} \int_\R \int_{\R^{\N_{>i}}} 
                         \frac{\xi_i}{\lambda_i} \tilde f(\bsxi_{<i},\xi_i,\bsxi_{>i}) \, 
                        {\rm d}\bsmu_\bslambda^{>i}(\bsxi_{>i}) {\rm d}\mu_{0,\lambda_i}(\xi_i){\rm d}\bsmu_\bslambda^{<i}(\bsxi_{<i})  
                        \\
			&=
			\int_{\R^\N} \frac{\xi_i}{\lambda_i} \tilde f(\bsxi) \, {\rm d}\bsmu_\bslambda(\bsxi) \\
			&=
			\int_{\calX} \frac{\langle X, \phi_i \rangle}{\lambda_i} f(X) \, {\rm d}\gamma(X).
		\end{align*}
	\end{proof}
	Applying Lemma \ref{lemma:infiniteIBP} to the product $fg$ we get
	\begin{equation}\label{eq:GaussIBP}
		\int_\calX f(X) \frac{\partial}{\partial \phi_i}g(X) \, {\rm d}\gamma(X) 
		= 
		-\int_\calX g \frac{\partial}{\partial \phi_i}f(X) \, {\rm d}\gamma(X)
		+ 
		\int_\calX \frac{\langle X, \phi_i \rangle}{\lambda_i} f(X)g(X) \, {\rm d}\gamma(X), 
	\end{equation}
	for $f,g$ satisfying the above assumptions.
        
	We now prove Proposition \ref{prop:weightedParseval}. 
This result is a direct adaptation of \cite[Proposition
C.7.]{adcock2024learninglipschitzoperatorsrespect}, where a weighted
$\ell^2$-characterization of
weighted Sobolev spaces for
operators $\calG \colon \calX \to \calY$ is given.
We however assume the existence of (Fr\'{e}chet) derivatives in the strong sense.
	\begin{proof}[Proof of Proposition \ref{prop:weightedParseval}]
		Throughout the proof, fix the Gaussian measure $\gamma_s$ with $s \geq 0$ as in Definition \ref{def:GM} and denote for ease of notation the eigenvalues of its covariance by $\lambda_i \coloneqq \lambda_i(s)$.
		For $\bsnu \in \calF$ and $i\in\N$ define $\bsnu^{(i)} \coloneqq (\nu_k^{(i)})_{k\in\N} \in \calF$ as follows: If $\nu_i = 0$, set $\bsnu^{(i)} \coloneqq 0$ and if $\nu_i > 0$, set 
		\[
		\nu_k^{(i)} \coloneqq 
		\begin{cases}
			\nu_k - 1 \qquad &\text{if } k = i, \\
			\nu_k &\text{if } k \neq i.
		\end{cases}
		\]
		We make use of the ONB $\bsphi^{(r)} = \{\phi^{(r)}_i\}_{i\in\N}$ of $\calX^r$ 
                given by $\phi^{(r)}_i \coloneqq w_i^{r+1/2}\phi_i$
		and first aim at showing that
      \begin{equation}
        \label{eq:dphi-G}
		\frac{\partial}{\partial \phi^{(r)}_i} \calG(X) = \sum_{\bsnu\in\calF} w_i^{r+1/2} \sqrt{\frac{\nu_i}{\lambda_i}} g_{\bsnu,\bslambda} H_{\bsnu^{(i)},\bslambda}(X) \qquad \text{for all } i \in \N \text{ and } X \in \calX.
      \end{equation}
		Note that $\frac{\partial}{\partial \phi^{(r)}_i} \calG \in
    L_{\gamma_s}^2(\calX,\calY)$ since
    \begin{equation*}
    \|\frac{\partial}{\partial
      \phi^{(r)}_i}\calG(X)\|_\calY^2 \leq \sum_{i=1}^{\infty}
    \|\frac{\partial}{\partial \phi^{(r)}_i}\calG(X)\|_\calY^2 =
    \|D_{\calX^r}\calG(X)\|_{{\rm HS}(\calX^r,\calY)}^2
    \end{equation*}
     for all $X \in \calX$.
    Therefore, in order to prove \eqref{eq:dphi-G}, 
    it suffices to compute
    the Wiener-Hermite coefficients of $\frac{ \partial }{\partial \phi_i^{(r)}}\calG$ 
    and to show that
		\begin{equation}\label{eq:HermiteCoeffDerivative}
		\int_\calX \left(\frac{\partial}{\partial \phi^{(r)}_i}\calG(X)\right) H_{\bsnu^{(i)},\bslambda}(X) \, {\rm d}\gamma_s(X) = w_i^{r + 1/2} \sqrt{\frac{\nu_i}{\lambda_i}} g_{\bsnu,\bslambda}. 
		\end{equation}
Let $\ell \in \calY^*$ be arbitrary. 
We show that applying $\ell$ to the left-hand side and to the right-hand side of \eqref{eq:HermiteCoeffDerivative} 
yields the same result. We have
		\begin{align*}
			\ell \left(\int_\calX \left(\frac{\partial}{\partial \phi^{(r)}_i}\calG(X)\right) H_{\bsnu^{(i)},\bslambda}(X) \, {\rm d}\gamma_s(X) \right) 
			&=
			\int_\calX \ell\left(\frac{\partial}{\partial \phi^{(r)}_i}\calG(X)\right) H_{\bsnu^{(i)},\bslambda}(X) \, {\rm d}\gamma_s(X) \\
			&=
			\int_\calX \frac{\partial}{\partial \phi^{(r)}_i}(\ell \circ \calG)(X) H_{\bsnu^{(i)},\bslambda}(X) \, {\rm d}\gamma_s(X).
		\end{align*}
		The last equality holds because $\ell$ is linear and bounded. Next note that $\frac{\partial}{\partial \phi^{(r)}_i} = w_i^{r+1/2} \frac{\partial}{\partial \phi_i}$ and use the integration by parts formula from Equation \eqref{eq:GaussIBP}. For ease of notation define $f \coloneqq \ell \,\circ\, \calG \colon \calX \to \R$. Then,
		\begin{align*}
			&\int_\calX \left(\frac{\partial}{\partial \phi^{(r)}_i} f(X)\right) H_{\bsnu^{(i)},\bslambda}(X) \, {\rm d}\gamma_s(X) \\
      &\quad =
			w_i^{r+1/2} \int_\calX \left(\frac{\partial}{\partial \phi_i}f(X)\right) H_{\bsnu^{(i)},\bslambda}(X) \, {\rm d}\gamma_s(X) \\
			&\quad =
			-w_i^{r+1/2} \int_\calX f(X)\frac{\partial}{\partial \phi_i}
        H_{\bsnu^{(i)},\bslambda}(X) \, {\rm d}\gamma_s(X)
 + w_i^{r+1/2} \int_\calX \frac{\langle X,\phi_i \rangle}{\lambda_i} f(X) H_{\bsnu^{(i)},\bslambda}(X) \, {\rm d}\gamma_s(X) \\
			&\quad =
			w_i^{r+1/2} \int_\calX f(X) \prod_{j\in\N,j\neq i} H_{\nu_j} \left(\frac{\langle X,\phi_j \rangle}{\sqrt{\lambda_j}}\right) \\
			&\quad\qquad  \times 
			\left[ \frac{\langle X, \phi_i \rangle}{\lambda_i} H_{\nu_i - 1}\left(\frac{\langle X, \phi_i \rangle}{\sqrt{\lambda_i}}\right) - \sqrt{\frac{\nu_i-1}{\lambda_i}} H_{\nu_i-2}\left(\frac{\langle X, \phi_i \rangle}{\sqrt{\lambda_i}}\right)\right] {\rm d}\gamma_s(X).
		\end{align*}
		Now, \cite[Eq. (9.7)]{DPIntro06} implies 
		\begin{equation*}
			\frac{\langle X, \phi_i \rangle}{\lambda_i} H_{\nu_i-1}\left(\frac{\langle X, \phi_i \rangle}{\sqrt{\lambda_i}}\right) 
			= 
			\sqrt{\frac{\nu_i}{\lambda_i}} H_{\nu_i}\left(\frac{\langle X, \phi_i \rangle}{\sqrt{\lambda_i}}\right)
			+
			\sqrt{\frac{\nu_i-1}{\lambda_i}} H_{\nu_i-2}\left(\frac{\langle X, \phi_i \rangle}{\sqrt{\lambda_i}}\right),
		\end{equation*}
		which shows 
		\begin{equation*}
			\int_\calX \left(\frac{\partial}{\partial \phi^{(r)}_i}f(X)\right) H_{\bsnu^{(i)},\bslambda}(X) \, {\rm d}\gamma_s(X) 
			=
			w_i^{r+1/2} \sqrt{\frac{\nu_i}{\lambda_i}} \int_\calX f(X) H_{\bsnu,\bslambda}(X) \, {\rm d}\gamma_s(X).
		\end{equation*}
		From this we obtain
		\begin{align*}
			\ell\left( \int_\calX \left(\frac{\partial}{\partial \phi^{(r)}_i}\calG(X)\right)H_{\bsnu^{(i)},\bslambda}(X) \, {\rm d}\gamma_s(X) \right)
			&=
			w_i^{r+1/2} \sqrt{\frac{\nu_i}{\lambda_i}} \int_\calX (\ell \circ \calG)(X) H_{\bsnu,\bslambda}(X) \, {\rm d}\gamma_s(X) \\
			&= w_i^{r+1/2} \sqrt{\frac{\nu_i}{\lambda_i}} \ell\left(\int_\calX \calG(X) H_{\bsnu,\bslambda}(X)\,{\rm d}\gamma_s(X)\right) \\
			&= 
			w_i^{r+1/2} \sqrt{\frac{\nu_i}{\lambda_i}} \ell(g_{\bsnu,\bslambda}).
		\end{align*}
		Since $\ell \in \calY^*$ is arbitrary this implies \eqref{eq:HermiteCoeffDerivative} as a consequence of the Hahn-Banach theorem.
		Since in addition $\sum_{i=1}^n\|\frac{\partial}{\partial \phi^{(r)}_i}\calG(X)\|_\calY^2 \leq \|D_{\calX^r}\calG(X)\|_{{\rm HS}(\calX^r,\calY)}^2$ for all $n\in\N$ and $X \in \calX$, the dominated convergence theorem and Parseval's identity \eqref{eq:Parsev} imply
		\begin{align*}
			\int_\calX \|D_{\calX^r}\calG(X)\|_{{\rm HS}(\calX^r,\calY)}^2 \, {\rm d}\gamma_s(X)
			&=
			\int_\calX \sum_{i=1}^{\infty} \left\|\frac{\partial}{\partial \phi^{(r)}_i}\calG(X)\right\|_\calY^2 \, {\rm d}\gamma_s(X) \\
			&=
			\sum_{i=1}^{\infty} \int_\calX \left\|\frac{\partial}{\partial \phi^{(r)}_i} \calG(X)\right\|_\calY^2 \, {\rm d}\gamma_s(X) \\
			&=
			\sum_{i=1}^{\infty} \sum_{\bsnu\in\calF} w_i^{2r+1} \nu_i \lambda_i^{-1} \|g_{\bsnu,\bslambda}\|_\calY^2. 
		\end{align*}
		Finally, Parseval's identity \eqref{eq:Parsev} implies 
		\[
		\int_\calX \|\calG(X)\|_\calY^2 \, {\rm d}\gamma_s(X)
		+
		\int_\calX \|D_{\calX^r}\calG(X)\|_{{\rm HS}(\calX^r,\calY)}^2 \, {\rm d}\gamma_s(X) = \sum_{\bsnu\in\calF} \left(1+\sum_{i=1}^{\infty} w_i^{2r+1} \nu_i \lambda_i^{-1}\right) \|g_{\bsnu,\bslambda}\|_\calY^2.
		\]
	\end{proof}
	
We conclude with the proof of Lemma \ref{lemma:uGaussSobolevNorm}.
	\begin{proof}[Proof of Lemma \ref{lemma:uGaussSobolevNorm}]
	We fix the Gaussian measure $\gamma_s$ with $s \geq 0$ as in Definition \ref{def:GM} 
and denote by $\bsmu$ the infinite product Gaussian measure on $(\R^\N,\calB(\R^\N))$ as introduced in Section \ref{sec:Notat}. 
Furthermore, let $r \geq 0$ and let $T_s \colon \R^\N \to \calX$ with $T_s(\bsy) = \sum_{j=1}^\infty y_j \sqrt{\lambda_j(s)} \phi_j$. 
Note that $\gamma_s = (T_s)_\# \bsmu$ and $T_s^{-1}(X) = (\lambda_j(s)^{-1/2}\langle X,\phi_j \rangle_\calX)_{j\in\N}$. 
For $u \colon \R^\N \to \calY$ such that $\|u\|_{W_{\bsmu,r,s}^{1,2}(\R^\N,\calY)} < \infty$ 
we set $\calG \coloneqq u \circ T_s^{-1} \colon \calX \to \calY$. 
Through a change of variables and the definition of $\calG$ we obtain 
		\begin{equation*}
			\int_\calX \|\calG(X)\|_\calY^2 \, \dd\gamma_s(X) 
			=
			\int_{\R^\N} \|\calG(T_s(\bsy))\|_\calY^2 \, \dd\bsmu(\bsy)
			=
			\int_{\R^\N} \|u(\bsy)\|_\calY^2 \, \dd\bsmu(\bsy).
		\end{equation*}
	Next we consider the Hilbert-Schmidt norm $\|D_{\calX^r}\calG(X)\|_{{\rm HS}(\calX^r,\calY)}^2$ for $X \in \calX$ 
and the Fr\'{e}chet derivative restricted to $\calX^r \subset \calX$. 
Let $\phi_j^{(r)} \coloneqq w_j^{r+1/2}\phi_j$ for all $j\in\N$. 
The collection $\{\phi_j^{(r)}\}_{j\in\N}$ then defines an ONB of $\calX^r$. 
From this we obtain for all $X\in \calX$
		\begin{equation*}
			\|D_{\calX^r}\calG(X)\|_{{\rm HS}(\calX^r,\calY)}^2
			= 
			\sum_{j=1}^\infty \|D\calG(X)\phi_j^{(r)}\|_\calY^2
			=
			\sum_{j=1}^\infty w_j^{2r+1}\|D\calG(X) \phi_j\|_\calY^2.
		\end{equation*}
		Applying the chain rule we obtain
		\begin{equation*}
			D\calG(X) = Du(T_s^{-1}X) \circ DT_s^{-1}(X).
		\end{equation*}
		At $X = T_s(\bsy),$ with $\bsy \in \R^\N$, it holds that
		\begin{equation*}
			D\calG(T_s(\bsy)) = Du(\bsy) \circ DT_s^{-1}(T_s(\bsy)).
		\end{equation*}
		Next we compute the derivative of $T_s^{-1}$. 
    Since $T_s(\bsy) = \sum_{j=1}^\infty y_j
    \sqrt{\lambda_j(s)}\phi_j$ we have $DT_s(\bsy) \bse_j =
    \sqrt{\lambda_j(s)}\phi_j$, with
    $\bse_j \coloneqq (\delta_{ij})_{i\in\N} \in \R^\N$. 
    Therefore $DT_s^{-1}(T_s(\bsy)) \phi_j = \lambda_j(s)^{-1/2} \bse_j$. 
    From this we obtain 
		\begin{equation*}
			D\calG(T_s(\bsy))\phi_j = \lambda_j(s)^{-1/2} Du(\bsy)\bse_j = \lambda_j(s)^{-1/2}\, \partial_{y_j}u(\bsy).
		\end{equation*}
		Hence
		\begin{equation*}
			\|D_{\calX^r}\calG(X)\|_{{\rm HS}(\calX^r,\calY)}^2
			= 
			\sum_{j=1}^\infty w_j^{2r+1} \lambda_j(s)^{-1} \left\|\partial_{y_j}u(\bsy)\right\|_\calY^2.
		\end{equation*}
		This now finally shows that
		\begin{equation*}
			\|u\|_{W_{\bsmu,r,s}^{1,2}(\R^\N,\calY)}^2
			=
			\int_{\R^\N} \|u(\bsy)\|_\calY^2 \, \dd\bsmu(\bsy)
			+
			\int_{\R^\N} \sum_{j=1}^\infty w_j^{2r+1}\lambda_j(s)^{-1} \left\|\partial_{y_j} u(\bsy)\right\|_\calY^2 \, \dd\bsmu(\bsy).
		\end{equation*}
		
	To complete the proof, we consider the Wiener-Hermite pc coefficients 
        $g_{\bsnu,\bslambda(s)} \in \calY$ of $\calG = u \circ T_s^{-1}$ 
        and apply a change of variables:
		\begin{align*}
			g_{\bsnu,\bslambda(s)}
			&=
			\int_\calX \calG(X) H_{\bsnu, \bslambda(s)}(X) \, \dd\gamma_s(X) \\
			&=
			\int_{\R^\N} \calG(T_s(\bsy)) H_{\bsnu,\bslambda(s)}(T_s(\bsy)) \, \dd\bsmu(\bsy) \\
			&= 
			\int_{\R^\N} u(\bsy) H_\bsnu(\bsy) \, \dd\bsmu(\bsy) \\
			&= u_\bsnu.
		\end{align*}
		Note that $H_{\bsnu,\bslambda(s)}(X) = H_\bsnu(T_s^{-1}(X))$ for all $X \in \calX$.
		Therefore, applying Proposition \ref{prop:weightedParseval} to $\calG = u \circ T_s^{-1}$ we obtain
		\begin{equation*}
			\|u\|_{W_{\bsmu,r,s}^{1,2}(\R^\N,\calY)}^2 
			=
			\|\calG\|_{W_{\gamma_s,r}^{1,2}(\calX,\calY)}^2 
			=
			\sum_{\bsnu\in\calF} \Gamma_\bsnu(r,s) \|g_{\bsnu,\bslambda(s)}\|_\calY^2
			=
			\sum_{\bsnu\in\calF} \Gamma_\bsnu(r,s) \|u_\bsnu\|_\calY^2.
		\end{equation*}
	\end{proof}
\end{document}